\documentclass[11pt,oneside,reqno]{amsart}
\usepackage[latin9]{inputenc}
\usepackage{geometry}
\usepackage{mathrsfs}
\usepackage{amstext}
\usepackage{amsthm}
\usepackage{amssymb}
\usepackage{graphicx}
\usepackage{setspace}
\usepackage[unicode=true,
 bookmarks=false,
 breaklinks=false,pdfborder={0 0 1},colorlinks=false]
 {hyperref}

\makeatletter

\numberwithin{equation}{section}
\numberwithin{figure}{section}
\theoremstyle{plain}
\newtheorem{thm}{\protect\theoremname}[section]
\theoremstyle{remark}
\newtheorem{rem}[thm]{\protect\remarkname}
\theoremstyle{definition}
\newtheorem{defn}[thm]{\protect\definitionname}
\theoremstyle{remark}
\newtheorem{notation}[thm]{\protect\notationname}
\theoremstyle{plain}
\newtheorem{lem}[thm]{\protect\lemmaname}
\theoremstyle{plain}
\newtheorem{prop}[thm]{\protect\propositionname}
\theoremstyle{remark}
\newtheorem*{acknowledgement*}{\protect\acknowledgementname}

\usepackage{amsthm}\usepackage{mathrsfs}\usepackage{nameref}\usepackage{cite}\usepackage{upgreek}
\usepackage{stmaryrd}

\DeclareFontFamily{U}{matha}{\hyphenchar\font45}
\DeclareFontShape{U}{matha}{m}{n}{
      <5> <6> <7> <8> <9> <10> gen * matha
      <10.95> matha10 <12> <14.4> <17.28> <20.74> <24.88> matha12
      }{}
\DeclareSymbolFont{matha}{U}{matha}{m}{n}
\DeclareFontSubstitution{U}{matha}{m}{n}

\DeclareFontFamily{U}{mathx}{\hyphenchar\font45}
\DeclareFontShape{U}{mathx}{m}{n}{
      <5> <6> <7> <8> <9> <10>
      <10.95> <12> <14.4> <17.28> <20.74> <24.88>
      mathx10
      }{}
\DeclareSymbolFont{mathx}{U}{mathx}{m}{n}
\DeclareFontSubstitution{U}{mathx}{m}{n}

\numberwithin{equation}{section}

\def\th@plain{\thm@notefont{}\itshape}
\def\th@definition{\thm@notefont{}\normalfont}

\makeatother

\providecommand{\acknowledgementname}{Acknowledgement}
\providecommand{\definitionname}{Definition}
\providecommand{\lemmaname}{Lemma}
\providecommand{\notationname}{Notation}
\providecommand{\propositionname}{Proposition}
\providecommand{\remarkname}{Remark}
\providecommand{\theoremname}{Theorem}

\begin{document}
\title[Metastability of Blume--Capel Model]{Metastability of Blume--Capel Model with Zero Chemical Potential
and Zero External Field}
\author{Seonwoo Kim}
\address{S. Kim. Department of Mathematical Sciences, Seoul National University,
Republic of Korea.}
\email{ksw6leta@snu.ac.kr}
\begin{abstract}
In this study, we investigate the metastable behavior of Metropolis-type
Glauber dynamics associated with the Blume--Capel model with zero
chemical potential and zero external field at very low temperatures.
The corresponding analyses for the same model with zero chemical potential
and positive small external field were performed in {[}Cirillo and
Nardi, Journal of Statistical Physics, 150: 1080-1114, 2013{]} and
{[}Landim and Lemire, Journal of Statistical Physics, 164: 346-376,
2016{]}. We obtain both large deviation-type and potential-theoretic
results on the metastable behavior in our setting. To this end, we
perform highly thorough investigation on the energy landscape, where
it is revealed that no critical configurations exist and alternatively
a massive flat plateau of saddle configurations resides therein.
\end{abstract}

\maketitle
\tableofcontents{}

\section{\label{sec1}Introduction}

Within the context of statistical mechanics, metastability is a phenomenon
of first-order phase transition that occurs in various systems consisting
of multiple locally stable states. Extensive research has been carried
out on metastability since the mid-20th century, ranging from the
early works \cite{B-E-G-K,C-G-O-V} to recently developed methodologies
\cite{B-L TM,B-L TM2,B-L MG,L-M-Seo2}. As a result, various stochastic
systems have been known to exhibit such behavior; important examples
include the small random perturbations of dynamical systems \cite{B-E-G-K,L-Mariani-Seo,R-S},
condensing interacting particle systems \cite{Kim,Kim-Seo NRIP,L-M-Seo2,Seo NRZRP},
and ferromagnetic spin systems at low temperatures \cite{BA-C,C-O,L-Le,L-Le-M 19,N-Z,N-S Ising1}.
We refer to the classic monographs \cite{B-denH,O-V} for detailed
explanation on the history and perspectives regarding the phenomenon
of metastability.

We investigate the metastable behavior of the well-known \textit{Blume--Capel
model} on two-dimensional (2D) lattices. This model is a ferromagnetic
spin system that consists of three spins, namely $-1$, $0$, and
$+1$, and it was originally introduced to study the $^{3}\mathrm{He}$--$^{4}\mathrm{He}$
phase transition. In this system, spin $0$ at a site indicates the
absence of particles, whereas spin $-1$ (resp. $+1$) at a site means
that the site is occupied by a particle with spin $-1$ (resp. $+1$).
The system is controlled by the Hamiltonian function (cf. \eqref{e_Horiginal})
that is defined on the collection of spin configurations. This Hamiltonian
represents the ferromagnetic behavior of the spins in the sense that
more aligned spin configurations exhibit greater stability. Thus,
the most stable configurations are the monochromatic ones (cf. \eqref{e_S}).
The system is controlled by Metropolis-type Glauber dynamics (cf.
\eqref{e_cbeta}), where $\beta$ is the inverse temperature, so that
it becomes exponentially difficult to overcome the energy barrier
in each spin update. According to the dynamics, we investigate the
long-time metastable transitions between the monochromatic spin configurations
in the low-temperature regime $\beta\rightarrow\infty$.

An inspection of formula \eqref{e_Horiginal} reveals that the Hamiltonian
has two variables: the \textit{chemical potential} $\lambda$ and
\textit{external magnetic field} $h$. We are interested in the metastable
behavior when these external effects are small; that is, when $(\lambda,\,h)$
is close to $(0,\,0)$. The case of $0<|\lambda|<h$ was thoroughly
investigated in \cite{C-O,M-O}, where \cite{C-O} worked on fixed
finite square tori and \cite{M-O} worked on the infinite lattice
$\mathbb{Z}^{2}$. Subsequently, the case of $\lambda=0$ and $h>0$
was studied in \cite{C-N,C-N-Spi 17,L-Le,L-Le-M 19}, where \cite{C-N,C-N-Spi 17,L-Le}
considered fixed finite square tori and \cite{L-Le-M 19} considered
finite square tori whose lengths increase to infinity. In all of the
above works, the authors established the existence of a special form
of \textit{critical saddle configurations}, whereby the metastable
transitions between the monochromatic configurations must pass through
a configuration of this type.

In this study, we investigate the \textit{Blume--Capel model on fixed
finite lattices in the case of $\lambda=h=0$}, in which it is remarkable
that \textit{no critical saddle configurations exist.} Instead, a
metastable transition starting from a monochromatic configuration
must occur along a \textit{massive flat plateau of saddle configurations}
to reach another. Hence, to analyze the exact behavior of the metastable
transitions quantitatively, \textit{the overall energy landscape of
spin configurations in the system must be investigated.} This is the
main mathematical obstacle that is successfully overcome in the current
study.

The massive flat saddle plateau is indeed the essence of the energy
landscape in our model. We denote by \textit{typical configurations}
(cf. Definition \ref{d_typ}) those that are reachable by metastable
transitions with respect to the correct scale. Then, we obtain two
types of typical configurations, namely \textit{bulk} ones and \textit{edge}
ones. Bulk typical configurations form the main component of metastable
transitions and their structure is very simple in that the transitions
occur one-dimensionally therein. Edge typical configurations constitute
the initiating and finalizing components of metastable transitions
and their structure is complex compared to that of the bulk ones.
Hence, the edge typical configurations need to be handled much more
delicately, as described in Section \ref{sec6}.

The structure of edge typical configurations is strongly dependent
on the boundary conditions on the lattice. More specifically, if the
lattice has open boundaries (i.e., $\Lambda=\llbracket1,\,K\rrbracket\times\llbracket1,\,L\rrbracket\subseteq\mathbb{Z}^{2}$),
the structure is relatively simple and the exact behavior of the dynamics
can be computed. However, if the lattice has periodic boundaries (i.e.,
$\Lambda=\mathbb{T}_{K}\times\mathbb{T}_{L}$), the situation becomes
complex and the structure becomes a Markov chain on certain subtrees
of a $K\times2$-shaped ladder graph. Although we cannot characterize
the exact behavior of the dynamics in this case, our estimate is sufficient
to deduce the main results of this study. We refer to Remarks \ref{r_bdry},
\ref{r_bdry2}, and \ref{r_bdry3} for further details.

The main results obtained in this study are divided into two types:
large deviation-type results (cf. Section \ref{sec2.2}) and potential
theory-type results (cf. Section \ref{sec2.3}). For the former, we
use the pathwise approach \cite{C-G-O-V} to metastability; in particular,
the recent methodology \cite{N-Z-B}, which enables us to estimate
certain concepts regarding metastability (the transition time, mixing
time, and spectral gap) by analyzing the valley depths of the energy
landscape. For the latter, we use the potential-theoretic \cite{B-E-G-K}
and martingale \cite{B-L TM,B-L TM2,B-L MG} approaches to metastability.
These methods offer the advantage of providing the sharp asymptotics
of the transition time by analyzing the capacity (cf. \eqref{e_Capdef}),
which is unattainable with the classic pathwise approach to metastability.

The Blume--Capel model has many similar features to the stochastic
\textit{Potts model} with three spins, which generalizes the number
of spins of the well-known stochastic \textit{Ising model} (which
has two spins, $-1$ and $+1$). The metastable behavior of the Ising
and Potts models has been studied extensively in the past several
decades \cite{BA-C,N-Z,N-S Ising1}. Recently, we conducted in
\cite{Kim-Seo Ising-Potts} (which is our companion paper) a quantitative
analysis on the metastable behavior of the Ising and Potts models
with zero external fields on two- and three- dimensional lattices,
and we frequently refer to \cite{Kim-Seo Ising-Potts} for insights
into and details on the deductions presented in this article.

Natural open questions arise in two directions. The first objective
is to investigate the dynamics with $\lambda=h=0$ on growing lattices,
as demonstrated by the authors of \cite{L-Le-M 19} for the case of
$\lambda=0$ and $h>0$. In this case, the growth rate of the lattice
and transition rate between the saddle configurations need to be compared
to deduce the exact time scale. The second objective is to study the
dynamics with $\lambda=h=0$ in the infinite volume lattice, as accomplished
in \cite{M-O}, for which it is necessary to investigate whether a
specific form of critical configurations still fails to exist in the
infinite volume case or emerges in this particular setting.

\section{\label{sec2}Main Results}

\subsection{\label{sec2.1}Model definition}

\subsubsection*{Blume--Capel model}

We define the Blume--Capel model on the finite 2D lattice box $\Lambda=\llbracket1,\,K\rrbracket\times\llbracket1,\,L\rrbracket$,
where $K$ and $L$ are fixed positive integers. For convenience,
we assume that 
\begin{equation}
5\le K\le L.\label{e_KL}
\end{equation}
We impose either open or periodic boundary conditions on $\Lambda$.
If $K=L$ under the periodic boundary conditions, the lattice is indeed
$\mathbb{T}_{L}\times\mathbb{T}_{L}$ as in the previous studies \cite{C-N,C-N-Spi 17,L-Le}.
For $x,\,y\in\Lambda$\footnote{If we take elements from a set by writing $a,\,b\in A$, we implicitly
imply that $a$ and $b$ are different.}, we write $x\sim y$ if they are nearest neighbors; that is, $|x-y|=1$.

We have three spins in this model, namely $-1$, $0$, and $+1$.
We denote by $\mathcal{X}=\{-1,\,0,\,+1\}^{\Lambda}$ the space of
the spin configurations on $\Lambda$. Subsequently, we define the
Hamiltonian $H:\mathcal{X}\rightarrow\mathbb{R}$ as
\begin{equation}
H(\sigma)=\sum_{x\sim y}\{\sigma(x)-\sigma(y)\}^{2}-\lambda\sum_{x\in\Lambda}\sigma(x)^{2}-h\sum_{x\in\Lambda}\sigma(x).\label{e_Horiginal}
\end{equation}
Here, $\sigma(x)$ is the spin of configuration $\sigma\in\mathcal{X}$
at site $x\in\Lambda$. Moreover, we assume that the \textit{chemical
potential} $\lambda$ and \textit{external field} $h$ are both zero,
so that
\begin{equation}
H(\sigma)=\sum_{x\sim y}\{\sigma(x)-\sigma(y)\}^{2}.\label{e_H}
\end{equation}
We denote by $\mu_{\beta}$ the Gibbs measure on $\mathcal{X}$ associated
with the Hamiltonian $H$ at the inverse temperature $\beta>0$:
\begin{equation}
\mu_{\beta}(\sigma)=\frac{1}{Z_{\beta}}e^{-\beta H(\sigma)},\;\;\;\;Z_{\beta}=\sum_{\sigma\in\mathcal{X}}e^{-\beta H(\sigma)}.\label{e_mu}
\end{equation}
We denote by $\mathbf{-1},\,\mathbf{0},\,\mathbf{+1}\in\mathcal{X}$
the monochromatic configurations, of which all spins are $-1,\,0,\,+1$,
respectively. We write
\begin{equation}
\mathcal{S}=\{\mathbf{-1},\,\mathbf{0},\,\mathbf{+1}\}.\label{e_S}
\end{equation}
When we select spins $a$ or $b$, the corresponding monochromatic
configuration is denoted by $\mathbf{a}\in\mathcal{S}$ or $\mathbf{b}\in\mathcal{S}$,
respectively. It is precisely on $\mathcal{S}$ that $H(\cdot)$ attains
its minimum $0$, and hence, $\mathcal{S}$ denotes the collection
of \textit{ground states}. The following estimates are straightforward:\footnote{For two collections $(a_{\beta})_{\beta>0}=(a_{\beta}(K,\,L))_{\beta>0}$
and $(b_{\beta})_{\beta>0}=(b_{\beta}(K,\,L))_{\beta>0}$ of real
numbers, we denote $a_{\beta}=O(b_{\beta})$ if there exists $C=C(K,\,L)>0$
such that $|a_{\beta}|\le Cb_{\beta}$ for all $\beta>0$ and $K,\,L$.
We denote $a_{\beta}=o(b_{\beta})$ if $\lim_{\beta\rightarrow\infty}a_{\beta}/b_{\beta}=0$
for all $K,\,L$. Moreover, we state that $a_{\beta}$ and $b_{\beta}$
are \textit{asymptotically equal} and denote by $a_{\beta}\simeq b_{\beta}$
if $\lim_{\beta\rightarrow\infty}a_{\beta}/b_{\beta}=1$ for all $K,\,L$.}
\begin{equation}
Z_{\beta}=3+O(e^{-2\beta})\;\;\;\;\text{and}\;\;\;\;\lim_{\beta\rightarrow\infty}\mu_{\beta}(\mathbf{s})=\frac{1}{3}\text{ for all }\mathbf{s}\in\mathcal{S}.\label{e_Zmu}
\end{equation}

\subsubsection*{Continuous-time Metropolis dynamics}

For $\sigma\in\mathcal{X}$, $x\in\Lambda$, and spin $a$, we denote
by $\sigma^{x,a}\in\mathcal{X}$ the configuration obtained from $\sigma$
by updating the spin at site $x$ to $a$. Thereafter, the dynamics
is defined as the continuous-time Markov chain $\{\sigma_{\beta}(t)\}_{t\ge0}$
on $\mathcal{X}$, the transition rates of which are given by
\begin{equation}
c_{\beta}(\sigma,\,\zeta)=\begin{cases}
e^{-\beta[H(\zeta)-H(\sigma)]_{+}} & \text{if }\zeta=\sigma^{x,a}\ne\sigma\text{ for some }x\in\Lambda\text{ and spin }a,\\
0 & \text{otherwise},
\end{cases}\label{e_cbeta}
\end{equation}
where $[t]_{+}=\max\{t,\,0\}$. It is easy to observe that $\sigma_{\beta}(\cdot)$
is irreducible. For $\sigma,\,\zeta\in\mathcal{X}$, we write $\sigma\sim\zeta$
if $c_{\beta}(\sigma,\,\zeta)>0$. It is clear that $\sigma\sim\zeta$
if and only if $\zeta\sim\sigma$, and that the relation $\sigma\sim\zeta$
does not depend on the exact value of $\beta>0$. Moreover, for each
$\mathcal{A}\subseteq\mathcal{X}$, we define the collection of edges
in $\mathcal{A}$ as follows:
\begin{equation}
E(\mathcal{A})=\{\{\sigma,\,\zeta\}\subseteq\mathcal{A}:\sigma\sim\zeta\}.\label{e_edgenot}
\end{equation}

For the above dynamics, the detailed balance condition holds; that
is,
\begin{equation}
\mu_{\beta}(\sigma)c_{\beta}(\sigma,\,\zeta)=\mu_{\beta}(\zeta)c_{\beta}(\zeta,\,\sigma)=\begin{cases}
\min\{\mu_{\beta}(\sigma),\,\mu_{\beta}(\zeta)\} & \text{if }\sigma\sim\zeta,\\
0 & \text{otherwise}.
\end{cases}\label{e_detbal}
\end{equation}
Hence, the invariant measure of this Metropolis dynamics $\sigma_{\beta}(\cdot)$
is exactly $\mu_{\beta}$, and $\sigma_{\beta}(\cdot)$ is reversible
with respect to $\mu_{\beta}$. We denote by $\mathbb{P}_{\sigma}^{\beta}$
and $\mathbb{E}_{\sigma}^{\beta}$ the law and expectation, respectively,
of the process $\sigma_{\beta}(\cdot)$ starting from $\sigma\in\mathcal{X}$.
\begin{rem}
\label{r_modelsym}We remark on the model symmetry. First, our model
is fully symmetric with respect to the spin correspondence $-1\leftrightarrow+1$.
\textit{However, our model is not symmetric with respect to $-1\leftrightarrow0$
or $0\leftrightarrow+1$.} Therefore, spins $-1$ and $+1$ play the
same role, but spin $0$ does not. This is the \textit{main difference}
from the Potts model studied in \cite{Kim-Seo Ising-Potts,N-Z}, in
which all of the spins play the same role. More specifically, we present
the following differentiated features in this study:
\begin{itemize}
\item The canonical transitions occur only along \textit{good} pairs of
spins (cf. Notation \ref{n_good}). Thus, when analyzing the relevant
configurations, care should be taken with this underlying asymmetry
of the model.
\item The typical configurations are defined individually for each good
pair, whereas the corresponding ones are globally defined in \cite{Kim-Seo Ising-Potts}.
This is because the edge typical configurations near $\mathbf{-1}$
and $\mathbf{+1}$ possess a different structure compared to those
near $\mathbf{0}$ (cf. Section \ref{sec6}; see also Remark \ref{r_Ia}).
\item We cannot estimate the capacities in a unified manner owing to the
model asymmetry; thus, we first construct fundamental test functions
and flows in Section \ref{sec7}, which serve as the building blocks
for the actual test objects. Subsequently, in Section \ref{sec8},
we construct individual test objects for each capacity (cf. Theorem
\ref{t_Cap}).
\end{itemize}
\end{rem}

\subsection{\label{sec2.2}Main results: large deviation-type results}

In this subsection, we explain the large deviation-type main results
on the metastable behavior.

\subsubsection*{Energy barrier between ground states}

First, we introduce the energy barrier of the energy landscape, which
is the level of energy that must be overcome to enable a metastable
transition from one ground state to another.
\begin{defn}
\label{d_Ebarrier}We define the following objects:
\begin{enumerate}
\item A sequence of configurations $(\omega_{n})_{n=0}^{N}=(\omega_{0},\,\omega_{1},\,\dots,\,\omega_{n})\subseteq\mathcal{X}$
is called a \textit{path} if $\omega_{n}\sim\omega_{n+1}$ for all
$n\in\llbracket0,\,N-1\rrbracket$\footnote{For integers $m$ and $n$, $\llbracket m,\,n\rrbracket$ denotes
$[m,\,n]\cap\mathbb{Z}$ (i.e., integers from $m$ to $n$).}. We state that this path connects $\sigma$ and $\zeta$ if $\omega_{0}=\sigma$
and $\omega_{N}=\zeta$, or vice versa. Moreover, we state that this
path is in $\mathcal{A}\subseteq\mathcal{X}$ if $\omega_{n}\in\mathcal{A}$
for all $n\in\llbracket0,\,N\rrbracket$. For $c\in\mathbb{R}$, a
path $(\omega_{n})_{n=0}^{N}$ is called a \textit{$c$-path} if $H(\omega_{n})\le c$
for all $n\in\llbracket0,\,N\rrbracket$.
\item The \textit{communication height} between two configurations $\sigma,\,\zeta\in\mathcal{X}$
is defined by 
\[
\Phi(\sigma,\,\zeta)=\min_{(\omega_{n})_{n=0}^{N}}\max_{n\in\llbracket0,N\rrbracket}H(\omega_{n}),
\]
where the minimum is taken over all paths $(\omega_{n})_{n=0}^{N}$
that connect $\sigma$ and $\zeta$. Furthermore, the communication
height between two disjoint sets $\mathcal{A},\,\mathcal{B}\subseteq\mathcal{X}$
is defined by
\[
\Phi(\mathcal{A},\,\mathcal{B})=\min_{\sigma\in\mathcal{A}}\min_{\zeta\in\mathcal{B}}\Phi(\sigma,\,\zeta).
\]
\item For two spins $a$ and $b$, we define the \textit{energy barrier}
between $\mathbf{a},\,\mathbf{b}\in\mathcal{S}$ by
\begin{equation}
\Gamma_{a,b}=\Gamma_{a,b}(K,\,L)=\Phi(\mathbf{a},\,\mathbf{b}).\label{e_Gammaab}
\end{equation}
It is clear that $\Gamma_{a,b}=\Gamma_{b,a}$.
\end{enumerate}
\end{defn}

The following theorem characterizes the exact energy barrier; we recall
\eqref{e_KL}.
\begin{thm}[Energy barrier]
\label{t_Ebarrier} Define a constant $\Gamma$ by
\begin{equation}
\Gamma=\begin{cases}
2K+2 & \text{under periodic boundary conditions},\\
K+1 & \text{under open boundary conditions}.
\end{cases}\label{e_Gamma}
\end{equation}
Then, it holds that 
\begin{equation}
\Gamma_{-1,0}=\Gamma_{0,+1}=\Gamma_{-1,+1}=\Gamma.\label{e_Ebarrier}
\end{equation}
\end{thm}

The proof of Theorem \ref{t_Ebarrier} is provided in Section \ref{sec4.2}.

\subsubsection*{Large deviation-type results}

We first define the following concepts:
\begin{itemize}
\item For $\mathcal{A}\subseteq\mathcal{X}$, we denote by $\tau_{\mathcal{A}}$
the hitting time of the set $\mathcal{A}$. Subsequently, for $\mathbf{s}\in\mathcal{S}$,
the hitting times $\tau_{\mathcal{S}\setminus\{\mathbf{s}\}}$ and
$\tau_{\mathbf{s}'}$, $\mathbf{s}'\in\mathcal{S}\setminus\{\mathbf{s}\}$
are called the \textit{(metastable) transition times} starting from
$\mathbf{s}$.
\item The \textit{mixing time} with respect to $\epsilon\in(0,\,1)$ is
defined by
\[
t_{\beta}^{\mathrm{mix}}(\epsilon)=\min\big\{ t\ge0:\max_{\sigma\in\mathcal{X}}\Vert\mathbb{P}_{\sigma}^{\beta}[\sigma_{\beta}(t)\in\cdot]-\mu_{\beta}(\cdot)\Vert_{\mathrm{TV}}\le\epsilon\big\},
\]
where $\Vert\cdot\Vert_{\mathrm{TV}}$ denotes the total variation
distance (cf. \cite[Chapter 4]{L-P-W}).
\item We denote by $\lambda_{\beta}$ the \textit{spectral gap} of our dynamics
(cf. \cite[Chapter 12]{L-P-W}).
\end{itemize}
\begin{thm}[Large deviation-type results]
\label{t_LDTresults} The following statements hold.
\begin{enumerate}
\item \textbf{(Transition time)} For all $\mathbf{s},\,\mathbf{s}'\in\mathcal{S}$
and $\epsilon>0$, we have
\begin{equation}
\lim_{\beta\rightarrow\infty}\mathbb{P}_{\mathbf{s}}^{\beta}[e^{\beta(\Gamma-\epsilon)}<\tau_{\mathcal{S}\setminus\{\mathbf{s}\}}\le\tau_{\mathbf{s}'}<e^{\beta(\Gamma+\epsilon)}]=1,\label{e_tt1}
\end{equation}
\begin{equation}
\lim_{\beta\rightarrow\infty}\frac{1}{\beta}\log\mathbb{E}_{\mathbf{s}}^{\beta}[\tau_{\mathcal{S}\setminus\{\mathbf{s}\}}]=\lim_{\beta\rightarrow\infty}\frac{1}{\beta}\log\mathbb{E}_{\mathbf{s}}^{\beta}[\tau_{\mathbf{s}'}]=\Gamma.\label{e_tt2}
\end{equation}
Moreover, under $\mathbb{P}_{\mathbf{s}}^{\beta}$, as $\beta\rightarrow\infty$,
\begin{equation}
\frac{\tau_{\mathcal{S}\setminus\{\mathbf{s}\}}}{\mathbb{E}_{\mathbf{s}}^{\beta}[\tau_{\mathcal{S}\setminus\{\mathbf{s}\}}]}\rightharpoonup\mathrm{Exp}(1)\;\;\;\;\text{and}\;\;\;\;\frac{\tau_{\mathbf{s}'}}{\mathbb{E}_{\mathbf{s}}^{\beta}[\tau_{\mathbf{s}'}]}\rightharpoonup\mathrm{Exp}(1),\label{e_tt3}
\end{equation}
where $\mathrm{Exp}(1)$ represents the exponential distribution with
parameter $1$.
\item \textbf{(Mixing time)} For all $\epsilon\in(0,\,1/2)$, the mixing
time $t_{\beta}^{\mathrm{mix}}(\epsilon)$ satisfies 
\[
\lim_{\beta\rightarrow\infty}\frac{1}{\beta}\log t_{\beta}^{\mathrm{mix}}(\epsilon)=\Gamma.
\]
\item \textbf{(Spectral gap)} There exist constants $0<c_{1}=c_{1}(K,\,L)\le c_{2}=c_{2}(K,\,L)$
such that 
\[
c_{1}e^{-\beta\Gamma}\le\lambda_{\beta}\le c_{2}e^{-\beta\Gamma}.
\]
\end{enumerate}
\end{thm}

\begin{rem}
\label{r_LDTresults}The connection between Theorems \ref{t_Ebarrier}
and \ref{t_LDTresults} is that the concepts discussed in Theorem
\ref{t_LDTresults} (the transition time, mixing time, and inverse
spectral gap) have an exponential scale with respect to the inverse
temperature $\beta\rightarrow\infty$, and the precise scale is the
energy barrier $\Gamma$ between the ground states that are determined
in Theorem \ref{t_Ebarrier}.
\end{rem}

\begin{rem}
\label{r_bdry}We remark that in Theorem \ref{t_LDTresults}, the
only difference between the two boundary types (periodic and open)
relates to the exact value of $\Gamma$, whereas the other features
regarding the three concepts are identical. Thus, we state that they
share the same \textit{exponential features} in the study of metastability.
However, crucial differences between them arise in more quantitative
analyses of the metastable transitions, which are presented in Section
\ref{sec2.3}. That is, the \textit{sub-exponential prefactor} differs
between the two boundary types because it depends on the number of
possible metastable transition paths between the ground states. The
reason for this difference is briefly discussed in Section \ref{sec9}.
\end{rem}

The proof of Theorem \ref{t_LDTresults} is provided in Section \ref{sec4.4}.

\subsubsection*{Metastable transition paths between ground states}

We obtain the following theorem for the metastable transition paths.
We remark that part (1) of Theorem \ref{t_eqp} implies the same behavior
of the metastable transition from $\mathbf{-1}$ to $\mathbf{+1}$
as that demonstrated in \cite[Proposition 2.1]{L-Le}, where the authors
investigated the case of $\lambda=0$ and $h>0$.
\begin{thm}[Transition paths]
\label{t_eqp} We have the following asymptotics for the metastable
transitions:
\begin{enumerate}
\item Starting from $\mathbf{-1}$, the chain must visit $\mathbf{0}$ on
its way to visiting $\mathbf{+1}$:
\[
\lim_{\beta\rightarrow\infty}\mathbb{P}_{\mathbf{-1}}^{\beta}[\tau_{\mathbf{0}}<\tau_{\mathbf{+1}}]=1.
\]
Similarly, we have $\lim_{\beta\rightarrow\infty}\mathbb{P}_{\mathbf{+1}}^{\beta}[\tau_{\mathbf{0}}<\tau_{\mathbf{-1}}]=1$.
\item Starting from $\mathbf{0}$, the probability of hitting $\mathbf{-1}$
before $\mathbf{+1}$ is equal to the opposite case; that is,
\[
\lim_{\beta\rightarrow\infty}\mathbb{P}_{\mathbf{0}}^{\beta}[\tau_{\mathbf{-1}}<\tau_{\mathbf{+1}}]=\lim_{\beta\rightarrow\infty}\mathbb{P}_{\mathbf{0}}^{\beta}[\tau_{\mathbf{+1}}<\tau_{\mathbf{-1}}]=\frac{1}{2}.
\]
\end{enumerate}
\end{thm}

Using the potential-theoretic terminology (which is reviewed in Section
\ref{sec3}), the above theorem is equivalent to
\[
\lim_{\beta\rightarrow\infty}h_{\mathbf{0},\mathbf{+1}}^{\beta}(\mathbf{-1})=\lim_{\beta\rightarrow\infty}h_{\mathbf{0},\mathbf{-1}}^{\beta}(\mathbf{+1})=1\;\;\;\;\text{and}\;\;\;\;\lim_{\beta\rightarrow\infty}h_{\mathbf{-1},\mathbf{+1}}^{\beta}(\mathbf{0})=\frac{1}{2}.
\]
We remark that part (2) of Theorem \ref{t_eqp} is straightforward
based on the symmetry of our model (cf. Remark \ref{r_modelsym}).
The proof of part (1) of this theorem is presented in Section \ref{sec5.4}.

\subsection{\label{sec2.3}Main results: potential-theoretic results}

Whereas the preceding main results focused on the exponential estimates
(as $\beta\rightarrow\infty$) of the metastable quantities, the following
main results provide more quantitative analyses based on potential-theoretic
methods. The \textit{Eyring--Kramers formula} (Theorem \ref{t_EK})
substantially generalizes \eqref{e_tt2}, and the \textit{Markov chain
reduction} (Theorem \ref{t_MC}), in the sense of \cite{B-L TM,B-L TM2},
describes the successive metastable transitions between the ground
states.

A crucial difference between the results in the current and preceding
subsections is that the quantitative results in this subsection \textit{are
dependent on the selection of the boundary conditions}. For simplicity,
\textit{we assume open boundary conditions in this subsection.} The
periodic case can be handled in a similar manner; thus, we briefly
discuss the periodic case in Section \ref{sec9}.

\subsubsection*{Eyring--Kramers formula}

The following result generalizes \eqref{e_tt2}, in the sense that
it characterizes the sub-exponential prefactor with respect to the
exponential factor $e^{\beta\Gamma}$ that appears in the quantities
in Theorem \ref{t_LDTresults}.
\begin{thm}[Eyring--Kramers law]
\label{t_EK} Under open boundary conditions on $\Lambda$, there
exists a constant $\kappa=\kappa(K,\,L)>0$ such that the following
estimates hold:
\begin{enumerate}
\item $\mathbb{E}_{\mathbf{-1}}^{\beta}[\tau_{\{\mathbf{0},\mathbf{+1}\}}]=\mathbb{E}_{\mathbf{+1}}^{\beta}[\tau_{\{\mathbf{-1},\mathbf{0}\}}]\simeq\kappa e^{\beta\Gamma}$
and $\mathbb{E}_{\mathbf{0}}^{\beta}[\tau_{\{\mathbf{-1},\mathbf{+1}\}}]\simeq\frac{\kappa}{2}e^{\beta\Gamma}$.
\item $\mathbb{E}_{\mathbf{-1}}^{\beta}[\tau_{\mathbf{0}}]=\mathbb{E}_{\mathbf{+1}}^{\beta}[\tau_{\mathbf{0}}]\simeq\kappa e^{\beta\Gamma}$.
\item $\mathbb{E}_{\mathbf{0}}^{\beta}[\tau_{\mathbf{-1}}]=\mathbb{E}_{\mathbf{0}}^{\beta}[\tau_{\mathbf{+1}}]\simeq2\kappa e^{\beta\Gamma}$.
\item $\mathbb{E}_{\mathbf{-1}}^{\beta}[\tau_{\mathbf{+1}}]=\mathbb{E}_{\mathbf{+1}}^{\beta}[\tau_{\mathbf{-1}}]\simeq3\kappa e^{\beta\Gamma}$.
\end{enumerate}
Moreover, the constant $\kappa$ satisfies (cf. \eqref{e_KL})
\begin{equation}
\lim_{K\rightarrow\infty}\frac{\kappa(K,\,L)}{KL}=\begin{cases}
1/4 & \text{if }K<L,\\
1/8 & \text{if }K=L.
\end{cases}\label{e_EKkappa}
\end{equation}
\end{thm}

Part (1) of Theorem \ref{t_EK} provides the estimate of $\mathbb{E}_{\mathbf{s}}^{\beta}[\tau_{\mathcal{S}\setminus\{\mathbf{s}\}}]$
for $\mathbf{s}\in\mathcal{S}$, which is the expected time for a
transition from $\mathbf{s}$ to another ground state. This is the
so-called \textit{Eyring--Kramers law} for the Metropolis dynamics.
The proof of Theorem \ref{t_EK} is discussed in Section \ref{sec3}.
\begin{rem}
The limit \eqref{e_EKkappa} provides the prefactor
estimate of the metastable transition times. According to Remark \ref{r_bdry},
it can be expected that in the periodic boundary case, a different
estimate on the prefactor $\kappa=\kappa(K,\,L)$ will be obtained.
This is indeed the case and the precise estimate in the periodic case
is \eqref{e_EKkappa-1} provided in Section \ref{sec9}. The asymptotic
factor difference between the conditions on the boundaries is $KL$,
which is fundamentally owing to the number of possible paths for the
canonical transitions (cf. Definition \ref{d_can}). We refer to Section
\ref{sec9} for a more detailed explanation of this comparison.
\end{rem}

\begin{rem}
\label{r_bdry2}A notable feature that only the open boundary model
possesses is that \textit{we can explicitly compute the constant $\kappa$},
which is provided in Definition \ref{d_const}. More specifically,
the edge constant $\mathfrak{e}=\mathfrak{e}(K)$ can be completely
characterized, which is described in Section \ref{secA} by \textit{solving
the symmetric recurrence formulas (cf. \eqref{e_auxproc2} and \eqref{e_auxproc3}).}
This is not the case in the periodic boundary case; we can clearly
characterize the asymptotic limit \eqref{e_EKkappa-1}, but we cannot
obtain such an explicit formula for the edge constant $\mathfrak{e}'=\mathfrak{e}'(K,\,L)$
(note that $\mathfrak{e}'$ depends on both $K$ and $L$). We overcome
this drawback in the periodic case by providing a sufficient upper
bound on $\mathfrak{e}'$ (cf. \eqref{e_e'}).
\end{rem}

\begin{rem}
\label{r_lambdah}We compare the precise asymptotics obtained in Theorem
\ref{t_EK} to those obtained in \cite[Propositions 2.4 and 2.5]{L-Le}
and \cite[Theorems 5 and 6]{C-N-Spi 17} for the case of $\lambda=0$
and $h>0$. The main observable difference is that the asymptotics
are dependent on the lattice size $K\times L$, which was not the
case in previous studies. This is because in our setting, canonical
metastable transitions (cf. Definition \ref{d_can}) occur by updating
the spins of the entire lattice line by line; each spin update of
a line constitutes a positive portion of the expected transition time.
Hence, the exact lattice size is relevant in this case. However, in
the case of $\lambda=0$ and $h>0$, the essence of the metastable
transition is the construction of a specific form of critical saddle
configurations. Following the formulation, the process rapidly proceeds
to the target ground state. Hence, the lattice only needs to be sufficiently
large to contain such critical configurations and the exact size is
irrelevant to the sharp transition time.
\end{rem}

\begin{rem}
An interesting phenomenon occurs in \cite[Propositions 2.4 and 2.5]{L-Le}
for the case of $\lambda=0$ and $h>0$, which is that the time scale
of the expected transition time $\mathbb{E}_{\mathbf{-1}}^{\beta}[\tau_{\mathbf{0}}]$
is larger than the time scale of $\mathbb{E}_{\mathbf{-1}}^{\beta}[\tau_{\mathbf{+1}}]$
and $\mathbb{E}_{\mathbf{0}}^{\beta}[\tau_{\mathbf{+1}}]$. This is
owing to the fact that the main contribution to the quantity $\mathbb{E}_{\mathbf{-1}}^{\beta}[\tau_{\mathbf{0}}]$
originates from the event that the process (starting from $\mathbf{-1}$)
first hits $\mathbf{+1}$ and subsequently arrives at $\mathbf{0}$,
which means that the valley with respect to $\mathbf{+1}$ is much
deeper than the others. This is not the case in our model, because
the valley depths are all equal to $\Gamma$ according to Theorem
\ref{t_Ebarrier}. Hence, we determine that all of the relevant expected
transition times share the same time scale, which is $e^{\beta\Gamma}$.
\end{rem}

\subsubsection*{Markov chain reduction}

In our model, the ground states in $\mathcal{S}$ have the same depth
of energy $0$. Moreover, Theorem \ref{t_Ebarrier} states that the
energy barriers between these are also identical as $\Gamma$. Therefore,
the metastable transitions between the ground states occur in the
same time scale $e^{\beta\Gamma}$. From this perspective, we attempt
to analyze all of these successive transitions simultaneously. The
general method for carrying this out is the Markov chain reduction
technique that was introduced in \cite{B-L TM,B-L TM2,B-L MG}. According
to this methodology, we prove that the process of the (properly accelerated)
metastable transitions converges to a certain Markov chain on the
ground states.

To explain this result, we first introduce the trace process on $\mathcal{S}$.
In view of Theorem \ref{t_EK}, the process needs to be accelerated
by the factor $e^{\beta\Gamma}$ to govern the metastable transitions
in the ordinary time scale. Hence, we denote by $\widehat{\sigma}_{\beta}(t)=\sigma_{\beta}(e^{\beta\Gamma}t)$,
$t\ge0$ the accelerated process. Subsequently, we define a random
time $T(t)$, $t\ge0$ as 
\[
T(t)=\int_{0}^{t}\mathbf{1}\{\widehat{\sigma}_{\beta}(u)\in\mathcal{S}\}du,
\]
which is the local time of the process $\widehat{\sigma}_{\beta}(\cdot)$
in $\mathcal{S}$. Let $S(t)$, $t\ge0$ be the generalized inverse
of $T$; that is,
\[
S(t)=\sup\{u\ge0:T(u)\le t\}.
\]
The \textit{trace process} $\{X_{\beta}(t)\}_{t\ge0}$ on the set
$\mathcal{S}$ is defined by
\begin{equation}
X_{\beta}(t)=\widehat{\sigma}_{\beta}(S(t)).\label{e_trace}
\end{equation}
Subsequently, the trace process $X_{\beta}(\cdot)$ is the continuous-time,
irreducible Markov chain on $\mathcal{S}$. We refer to \cite[Proposition 6.1]{B-L TM}
for the proof of this fact.

Thereafter, we define the limiting Markov chain $\{X(t)\}_{t\ge0}$
on $\mathcal{S}$ as the continuous-time Markov chain that is associated
with the transition rate

\begin{equation}
r_{X}(\mathbf{s},\,\mathbf{s}')=\begin{cases}
\kappa^{-1} & \text{if }\{\mathbf{s},\,\mathbf{s}'\}=\{\mathbf{-1},\,\mathbf{0}\}\text{ or }\{\mathbf{0},\,\mathbf{+1}\},\\
0 & \text{otherwise}.
\end{cases}\label{e_LMC}
\end{equation}

\begin{thm}[Markov chain reduction]
\label{t_MC} Under open boundary conditions on $\Lambda$, the following
statements hold.
\begin{enumerate}
\item For $\mathbf{s}\in\mathcal{S}$, the law of the Markov chain $X_{\beta}(\cdot)$
starting from $\mathbf{s}$ converges to the law of the limiting Markov
chain $X(\cdot)$ starting from $\mathbf{s}$ in the limit $\beta\rightarrow\infty$.
\item The accelerated process spends negligible time outside $\mathcal{S}$;
that is,
\[
\lim_{\beta\rightarrow\infty}\sup_{\mathbf{s}\in\mathcal{S}}\mathbb{E}_{\mathbf{s}}^{\beta}\Big[\int_{0}^{t}\mathbf{1}\{\widehat{\sigma}_{\beta}(u)\notin\mathcal{S}\}du\Big]=0.
\]
\end{enumerate}
\end{thm}

As the process spends most of its time in $\mathcal{S}$ according
to part (2) of Theorem \ref{t_MC}, the trace process $X_{\beta}(\cdot)$
on $\mathcal{S}$ indeed fully describes the process $\widehat{\sigma}_{\beta}(\cdot)$
in the limit $\beta\rightarrow\infty$. Based on this observation,
part (1) of Theorem \ref{t_MC} describes the successive metastable
transitions of the Metropolis dynamics. The proof of Theorem \ref{t_MC}
is discussed in Section \ref{sec3}.
\begin{rem}
\label{r_discretetime}In this study, we select the continuous-time
version of the Metropolis dynamics as in \cite{L-Le,L-Le-M 19,M-O}.
As an alternative, we may also select the discrete-time Metropolis
dynamics on the Blume--Capel model, as in \cite{C-N,C-N-Spi 17,C-O}.
In this case, the jump probability is defined as
\[
p_{\beta}(\sigma,\,\zeta)=\begin{cases}
\frac{1}{2|\Lambda|}e^{-\beta[H(\zeta)-H(\sigma)]_{+}} & \text{if }\zeta=\sigma^{x,a}\ne\sigma\text{ for some }x\in\Lambda\text{ and spin }a,\\
1-\sum_{\zeta':\,\zeta'\ne\sigma}p_{\beta}(\sigma,\,\zeta') & \text{if }\zeta=\sigma.
\end{cases}
\]
The only difference is that the process is $2|\Lambda|$ times slower
than the original continuous-time process. Therefore, Theorems \ref{t_Ebarrier},
\ref{t_LDTresults}, and \ref{t_eqp} hold without modification, whereas
Theorems \ref{t_EK} and \ref{t_MC} hold with $2|\Lambda|\kappa$
instead of $\kappa$. Rigorous verifications can be conducted in the
same manner, and thus, we omit the details.
\end{rem}

\section{\label{sec3}Outline of Proofs}

In this section, we provide an outline of the proofs of the main theorems
presented in Section \ref{sec2}. Henceforth, \textit{we assume that
the lattice $\Lambda$ is given open boundary conditions}; that is,
$\Lambda=\llbracket1,\,K\rrbracket\times\llbracket1,\,L\rrbracket\subseteq\mathbb{Z}^{2}$,
except in Section \ref{sec9}, where we briefly discuss the case of
periodic boundaries.

First, we introduce the potential-theoretic approach to metastability.
Thereafter, based on the methodologies, we reduce the proofs of Theorems
\ref{t_EK} and \ref{t_MC} to capacity estimates between the ground
states (cf. Theorem \ref{t_Cap}).

We review several potential-theoretic notions. The \textit{Dirichlet
form} $D_{\beta}(\cdot)$ is defined as follows for $f:\mathcal{X}\rightarrow\mathbb{R}$:
\begin{equation}
D_{\beta}(f)=\frac{1}{2}\sum_{\sigma,\zeta\in\mathcal{X}}\mu_{\beta}(\sigma)c_{\beta}(\sigma,\,\zeta)[f(\zeta)-f(\sigma)]^{2}.\label{e_Diri}
\end{equation}

\begin{defn}
\label{d_pot}Let $\mathcal{A}$ and $\mathcal{B}$ be disjoint and
non-empty subsets of $\mathcal{X}$. The \textit{equilibrium potential}
between $\mathcal{A}$ and $\mathcal{B}$ is the function $h_{\mathcal{A},\mathcal{B}}^{\beta}:\mathcal{X}\rightarrow\mathbb{R}$,
which is defined as
\begin{equation}
h_{\mathcal{A},\mathcal{B}}^{\beta}(\sigma)=\mathbb{P}_{\sigma}^{\beta}[\tau_{\mathcal{A}}<\tau_{\mathcal{B}}].\label{e_eqpdef}
\end{equation}
By definition, we immediately obtain
\begin{equation}
h_{\mathcal{A},\mathcal{B}}^{\beta}\equiv1\text{ on }\mathcal{A},\;\;\;\;h_{\mathcal{A},\mathcal{B}}^{\beta}\equiv0\text{ on }\mathcal{B},\;\;\;\;0\le h_{\mathcal{A},\mathcal{B}}^{\beta}\le1,\;\;\;\;\text{and}\;\;\;\;h_{\mathcal{A},\mathcal{B}}^{\beta}=1-h_{\mathcal{B},\mathcal{A}}^{\beta}.\label{e_eqpprop}
\end{equation}
Subsequently, we define the \textit{capacity} between $\mathcal{A}$
and $\mathcal{B}$ as
\begin{equation}
\mathrm{Cap}_{\beta}(\mathcal{A},\,\mathcal{B})=D_{\beta}(h_{\mathcal{A},\mathcal{B}}^{\beta}).\label{e_Capdef}
\end{equation}
\end{defn}

Moreover, we define the following constants that characterize the
constant $\kappa$ that appears in Theorem \ref{t_EK}.
\begin{defn}
\label{d_const}We define the constants $\mathfrak{b}$, $\mathfrak{e}$,
and $\kappa$.
\begin{itemize}
\item The bulk and edge constants $\mathfrak{b}=\mathfrak{b}(K,\,L)$ and
$\mathfrak{e}=\mathfrak{e}(K)$ are defined as 
\begin{equation}
\mathfrak{b}=\begin{cases}
\frac{K(L-4)}{4} & \text{if }K<L\\
\frac{K(L-4)}{8} & \text{if }K=L
\end{cases}\;\;\;\;\text{and}\;\;\;\;\mathfrak{e}=\begin{cases}
1/(4\mathfrak{c}_{K}) & \text{if }K<L,\\
1/(8\mathfrak{c}_{K}) & \text{if }K=L,
\end{cases}\label{e_be}
\end{equation}
where $\mathfrak{c}_{K}$ is the constant defined in \eqref{e_cK1}.
\item The constant $\kappa=\kappa(K,\,L)$ is defined as
\begin{equation}
\kappa=\mathfrak{b}+2\mathfrak{e}.\label{e_kappa}
\end{equation}
\end{itemize}
\end{defn}

We thus obtain the following theorem, which provides the main capacity
estimate.
\begin{thm}[Capacitiy estimates]
\label{t_Cap} The following estimates hold for the relevant capacities:
\begin{enumerate}
\item $\mathrm{Cap}_{\beta}(\mathbf{-1},\,\{\mathbf{0},\,\mathbf{+1}\})=\mathrm{Cap}_{\beta}(\mathbf{+1},\,\{\mathbf{-1},\,\mathbf{0}\})\simeq\frac{1}{3\kappa}e^{-\beta\Gamma}$.
\item $\mathrm{Cap}_{\beta}(\mathbf{-1},\,\mathbf{0})=\mathrm{Cap}_{\beta}(\mathbf{+1},\,\mathbf{0})\simeq\frac{1}{3\kappa}e^{-\beta\Gamma}$.
\item $\mathrm{Cap}_{\beta}(\mathbf{0},\,\{\mathbf{-1},\,\mathbf{+1}\})\simeq\frac{2}{3\kappa}e^{-\beta\Gamma}$.
\item $\mathrm{Cap}_{\beta}(\mathbf{-1},\,\mathbf{+1})\simeq\frac{1}{6\kappa}e^{-\beta\Gamma}$.
\end{enumerate}
\end{thm}

We explain the strategy for proving this theorem in Section \ref{sec3.1}.
At this point, we prove Theorems \ref{t_EK} and \ref{t_MC}, assuming
that Theorems \ref{t_eqp} and \ref{t_Cap} hold.
\begin{proof}[Proof of Theorem \ref{t_EK}]
 According to Definition \ref{d_const} and Lemma \ref{l_auxproc},
$\kappa$ satisfies the condition \eqref{e_EKkappa} stated in Theorem
\ref{t_EK}. Thus, it suffices to prove the formulas in parts (1)
to (4) of Theorem \ref{t_EK}.

We first prove part (1) of Theorem \ref{t_EK}. For $\mathbf{s}\in\mathcal{S}$,
according to \cite[Proposition 6.10]{B-L TM}, the following formula
holds for the mean transition time:
\[
\mathbb{E}_{\mathbf{s}}^{\beta}[\tau_{\mathcal{S}\setminus\{\mathbf{s}\}}]=\frac{1}{\mathrm{Cap}_{\beta}(\mathbf{s},\,\mathcal{S}\setminus\{\mathbf{s}\})}\sum_{\sigma\in\mathcal{X}}\mu_{\beta}(\sigma)h_{\mathbf{s},\mathcal{S}\setminus\{\mathbf{s}\}}^{\beta}(\sigma).
\]
By \eqref{e_Zmu} and \eqref{e_eqpprop}, we have
\[
\mathbb{E}_{\mathbf{s}}^{\beta}[\tau_{\mathcal{S}\setminus\{\mathbf{s}\}}]=\frac{1+o(1)}{3\mathrm{Cap}_{\beta}(\mathbf{s},\,\mathcal{S}\setminus\{\mathbf{s}\})}.
\]
Hence, we obtain the desired estimates from parts (1) and (3) of Theorem
\ref{t_Cap}.

For part (2), by symmetry (cf. Remark \ref{r_modelsym}), it suffices
to prove that $\mathbb{E}_{\mathbf{-1}}^{\beta}[\tau_{\mathbf{0}}]\simeq\kappa e^{\beta\Gamma}$.
Again, from \cite[Proposition 6.10]{B-L TM}, we have
\[
\mathbb{E}_{\mathbf{-1}}^{\beta}[\tau_{\mathbf{0}}]=\frac{1}{\mathrm{Cap}_{\beta}(\mathbf{-1},\,\mathbf{0})}\sum_{\sigma\in\mathcal{X}}\mu_{\beta}(\sigma)h_{\mathbf{-1},\mathbf{0}}^{\beta}(\sigma).
\]
By \eqref{e_Zmu}, \eqref{e_eqpprop}, and part (1) of Theorem \ref{t_eqp},
we have
\[
\mathbb{E}_{\mathbf{-1}}^{\beta}[\tau_{\mathbf{0}}]=\frac{1+o(1)}{3\mathrm{Cap}_{\beta}(\mathbf{-1},\,\mathbf{0})}.
\]
Hence, part (2) of Theorem \ref{t_Cap} concludes the proof of this
case.

For parts (3) and (4), similar deductions using Theorem \ref{t_eqp}
yield
\[
\mathbb{E}_{\mathbf{0}}^{\beta}[\tau_{\mathbf{-1}}]=\frac{2+o(1)}{3\mathrm{Cap}_{\beta}(\mathbf{0},\,\mathbf{-1})}\;\;\;\;\text{and}\;\;\;\;\mathbb{E}_{\mathbf{-1}}^{\beta}[\tau_{\mathbf{+1}}]=\frac{1+o(1)}{2\mathrm{Cap}_{\beta}(\mathbf{-1},\,\mathbf{+1})}.
\]
Therefore, we conclude the proof again by means of parts (2) and (4)
of Theorem \ref{t_Cap}.
\end{proof}
\begin{proof}[Proof of Theorem \ref{t_MC}]
We first consider part (1) of Theorem \ref{t_MC}. We denote by $r_{\beta}:\mathcal{S}\times\mathcal{S}\rightarrow[0,\,\infty)$
the transition rate of the trace process $X_{\beta}(\cdot)$ (cf.
\eqref{e_trace}). Subsequently, according to \cite[Theorem 2.7]{B-L TM},
it suffices to prove that $r_{\beta}$ converges to the limiting transition
rate $r_{X}$ in \eqref{e_LMC}. Thus, we claim that
\[
r_{\beta}(\mathbf{s},\,\mathbf{s}')=\begin{cases}
(1+o(1))/\kappa & \text{if }\{\mathbf{s},\,\mathbf{s}'\}=\{\mathbf{-1},\,\mathbf{0}\}\text{ or }\{\mathbf{0},\,\mathbf{+1}\},\\
o(1) & \text{otherwise}.
\end{cases}
\]
To this end, we recall the following result from \cite[Lemma 6.8]{B-L TM}:
\[
\mu_{\beta}(\mathbf{s})r_{\beta}(\mathbf{s},\,\mathbf{s}')=\frac{1}{2}\big[\mathrm{Cap}_{\beta}(\mathbf{s},\,\mathcal{S}\setminus\{\mathbf{s}\})+\mathrm{Cap}_{\beta}(\mathbf{s}',\,\mathcal{S}\setminus\{\mathbf{s}'\})-\mathrm{Cap}_{\beta}(\{\mathbf{s},\,\mathbf{s}'\},\,\mathcal{S}\setminus\{\mathbf{s},\,\mathbf{s}'\})\big].
\]
Hence, parts (1) and (3) of Theorem \ref{t_Cap} together with \eqref{e_Zmu}
conclude the proof of part (1) of Theorem \ref{t_MC}.

We consider part (2) of Theorem \ref{t_MC}. We denote by $\mathbb{P}_{\mu_{\beta}}^{\beta}$
the law of the Metropolis dynamics $\sigma_{\beta}(\cdot)$ for which
the initial distribution is $\mu_{\beta}$. Thus, for any $u>0$,
\begin{equation}
\mathbb{P}_{\mathbf{s}}^{\beta}[\sigma_{\beta}(u)\notin\mathcal{S}]\le\frac{1}{\mu_{\beta}(\mathbf{s})}\mathbb{P}_{\mu_{\beta}}^{\beta}[\sigma_{\beta}(u)\notin\mathcal{S}]=\frac{\mu_{\beta}(\mathcal{X}\setminus\mathcal{S})}{\mu_{\beta}(\mathbf{s})},\label{e_MC}
\end{equation}
where the equality holds because $\mu_{\beta}$ is the invariant distribution.
Therefore, by the Fubini theorem, we deduce
\[
\mathbb{E}_{\mathbf{s}}^{\beta}\Big[\int_{0}^{t}\mathbf{1}\{\widehat{\sigma}_{\beta}(u)\notin\mathcal{S}\}du\Big]=\int_{0}^{t}\mathbb{P}_{\mathbf{s}}^{\beta}[\sigma_{\beta}(e^{\beta\Gamma}u)\notin\mathcal{S}]du\le t\cdot\frac{\mu_{\beta}(\mathcal{X}\setminus\mathcal{S})}{\mu_{\beta}(\mathbf{s})},
\]
which vanishes as $\beta\rightarrow\infty$ by \eqref{e_Zmu}.
\end{proof}

\subsection{\label{sec3.1}Capacity estimates}

In this subsection, we describe the strategy for proving Theorem \ref{t_Cap}.
More specifically, we review two variational principles that provide
the upper and lower bounds for the capacities, and explain how to
adapt these principles to our model.

\subsubsection*{Upper bound via Dirichlet principle}

For two disjoint and non-empty subsets $\mathcal{A}$ and $\mathcal{B}$
of $\mathcal{X}$, we denote by $\mathfrak{C}(\mathcal{A},\,\mathcal{B})$
the class of functions $f:\mathcal{X}\rightarrow\mathbb{R}$ with
$f\equiv1$ on $\mathcal{A}$ and $f\equiv0$ on $\mathcal{B}$. Then,
the \textit{Dirichlet principle} provides sharp upper bounds for the
capacities.
\begin{thm}[Dirichlet principle]
\label{t_DP} Let $\mathcal{A}$ and $\mathcal{B}$ be two disjoint
and non-empty subsets of $\mathcal{X}$. Then, we have 
\[
\mathrm{Cap}_{\beta}(\mathcal{A},\,\mathcal{B})=\inf_{f\in\mathfrak{C}(\mathcal{A},\mathcal{B})}D_{\beta}(f).
\]
The unique optimizer is the equilibrium potential $h_{\mathcal{A},\mathcal{B}}^{\beta}$
between $\mathcal{A}$ and $\mathcal{B}$ (cf. \eqref{e_eqpdef}).
\end{thm}

We refer to \cite[Theorem 2.7]{G-L}, in which the authors provide
proofs of the generalized version (for non-reversible systems).

Recall the definition \eqref{e_Capdef} of capacities. It is technically
impossible to obtain the exact values of the equilibrium potential
$h_{\mathcal{A},\mathcal{B}}^{\beta}$ to calculate the capacity.
Therefore, we typically construct a test function $h_{\mathrm{test}}\in\mathfrak{C}(\mathcal{A},\,\mathcal{B})$
which successfully approximates the equilibrium potential $h_{\mathcal{A},\mathcal{B}}^{\beta}$,
in the sense that $D_{\beta}(h_{\mathrm{test}})$ and $D_{\beta}(h_{\mathcal{A},\mathcal{B}}^{\beta})$
are close to one another. The Dirichlet principle asserts that we
indeed obtain the upper bound $\mathrm{Cap}_{\beta}(\mathcal{A},\,\mathcal{B})\le D_{\beta}(h_{\mathrm{test}})$.

\subsubsection*{Lower bound via generalized Thomson principle}

The opposite lower bound for the capacities are deduced from the (generalized)
Thomson principle. For this formulation, we first recall the flow
structure associated with the dynamics.
\begin{defn}
\label{d_flow}We define the flow structure associated with our Metropolis
dynamics.
\begin{enumerate}
\item A function $\phi:\mathcal{X}\times\mathcal{X}\rightarrow\mathbb{R}$
is called a\textit{ flow} on $\mathcal{X}$, if $\phi$ is \textit{compatible}
with $c_{\beta}(\cdot,\,\cdot)$, in the sense that
\begin{equation}
\phi(\sigma,\,\zeta)>0\text{ only if }c_{\beta}(\sigma,\,\zeta)>0,\label{e_flow1}
\end{equation}
and \textit{anti-symmetric}, in the sense that 
\begin{equation}
\phi(\sigma,\,\zeta)=-\phi(\zeta,\,\sigma)\text{ for all }\sigma\in\mathcal{X},\;\zeta\in\mathcal{X}.\label{e_flow2}
\end{equation}
We denote by $\mathfrak{F}=\mathfrak{F}_{\mathcal{X}}$ the collection
of flows on $\mathcal{X}$.
\item For each $\beta>0$, we assign an inner product $\langle\cdot,\,\cdot\rangle_{\beta}$
to $\mathfrak{F}$ as follows:
\begin{equation}
\langle\phi,\,\psi\rangle_{\beta}=\frac{1}{2}\sum_{\sigma,\zeta\in\mathcal{X}:\,\sigma\sim\zeta}\frac{\phi(\sigma,\,\zeta)\psi(\sigma,\,\zeta)}{\mu_{\beta}(\sigma)c_{\beta}(\sigma,\,\zeta)}\text{ for all }\phi\in\mathfrak{F},\;\psi\in\mathfrak{F},\label{e_inpr}
\end{equation}
where the summand is well defined by \eqref{e_flow1}. Consequently,
this induces the \textit{flow norm} $\|\cdot\|_{\beta}$ on $\mathfrak{F}$
by $\|\phi\|_{\beta}=\sqrt{\langle\phi,\,\phi\rangle_{\beta}}$ for
$\phi\in\mathfrak{F}$.
\item Given a flow $\phi\in\mathfrak{F}$, the \textit{divergence} of $\phi$
at $\sigma\in\mathcal{X}$ is defined as
\[
(\mathrm{div}\,\phi)(\sigma)=\sum_{\zeta\in\mathcal{X}}\phi(\sigma,\,\zeta)=\sum_{\zeta\in\mathcal{X}:\,\sigma\sim\zeta}\phi(\sigma,\,\zeta).
\]
\item For a function $f:\mathcal{X}\rightarrow\mathbb{R}$, we define $\Psi_{f}:\mathcal{X}\times\mathcal{X}\rightarrow\mathbb{R}$
as follows:
\begin{equation}
\Psi_{f}(\sigma,\,\zeta)=\mu_{\beta}(\sigma)c_{\beta}(\sigma,\,\zeta)[f(\sigma)-f(\zeta)]\;\;\;\;;\;\sigma\in\mathcal{X},\;\zeta\in\mathcal{X}.\label{e_fcnflow}
\end{equation}
Then, it is clear that $\Psi_{f}$ is a flow; that is, it satisfies
\eqref{e_flow1} and \eqref{e_flow2}. Moreover, by definition, we
obtain that
\begin{equation}
\|\Psi_{f}\|_{\beta}^{2}=\frac{1}{2}\sum_{\sigma,\zeta\in\mathcal{X}}\mu_{\beta}(\sigma)c_{\beta}(\sigma,\,\zeta)[f(\sigma)-f(\zeta)]^{2}=D_{\beta}(f).\label{e_fcnflow2}
\end{equation}
\end{enumerate}
\end{defn}

We state the \textit{(generalized) Thomson principle} (which was introduced
in \cite{Seo NRZRP}) for reversible Markov chains. We refer to \cite[Theorem 5.3]{Seo NRZRP}
for its proof.
\begin{thm}[Generalized Thomson principle]
\label{t_gTP} Let $\mathcal{A}$ and $\mathcal{B}$ be two disjoint
and non-empty subsets of $\mathcal{X}$. Then, we have
\begin{equation}
\mathrm{Cap}_{\beta}(\mathcal{A},\,\mathcal{B})=\sup_{\psi\ne0}\frac{1}{\|\psi\|_{\beta}^{2}}\Big[\sum_{\sigma\in\mathcal{X}}h_{\mathcal{A},\mathcal{B}}^{\beta}(\sigma)(\mathrm{div}\,\psi)(\sigma)\Big]^{2},\label{e_gTP}
\end{equation}
where $0$ is the zero flow. The optimizers are given by $c\Psi_{h_{\mathcal{A},\mathcal{B}}^{\beta}}$
for $c\neq0$.
\end{thm}

To apply Theorem \ref{t_gTP}, we use $c\Psi_{h_{\mathrm{test}}}$
where $h_{\mathrm{test}}$ is the test function used to approximate
$h_{\mathcal{A},\mathcal{B}}^{\beta}$; see Definition \ref{d_tfl}.

The remainder of this paper is organized as follows. In Section \ref{sec4},
we define several basic concepts that are crucial to understanding
the natural metastable transitions between the ground states. During
this process, we prove Theorems \ref{t_Ebarrier} and \ref{t_LDTresults}.
In Sections \ref{sec5} and \ref{sec6}, we define and investigate
the typical and gateway configurations that are the building blocks
of the overall energy landscape of our model. Such thorough investigation
results in the proof of Theorem \ref{t_eqp} in Section \ref{sec5}.
In Section \ref{sec7}, we construct the fundamental test functions
and flows, which are the components of the actual test objects, to
estimate the capacities. Thereafter, in Section \ref{sec8}, we prove
the capacity estimates in Theorem \ref{t_Cap}. Finally, in Section
\ref{sec9}, we discuss the periodic boundary case. The Appendix is
devoted to investigating the auxiliary process, which is used to handle
the edge typical configurations in Section \ref{sec6}.

\section{\label{sec4}Canonical Configurations and Energy Barrier}

The following notation is frequently used throughout the remainder
of the article.
\begin{notation}
\label{n_good}A pair $(a,\,b)$ of spins is called \textit{good},
if $\{a,\,b\}=\{-1,\,0\}$ or $\{0,\,+1\}$.
\end{notation}

Throughout the article, we use $v$ and $h$ to denote vertical and
horizontal lengths, respectively.

\subsection{\label{sec4.1}Canonical configurations and paths}

\begin{figure}
\includegraphics[width=13cm]{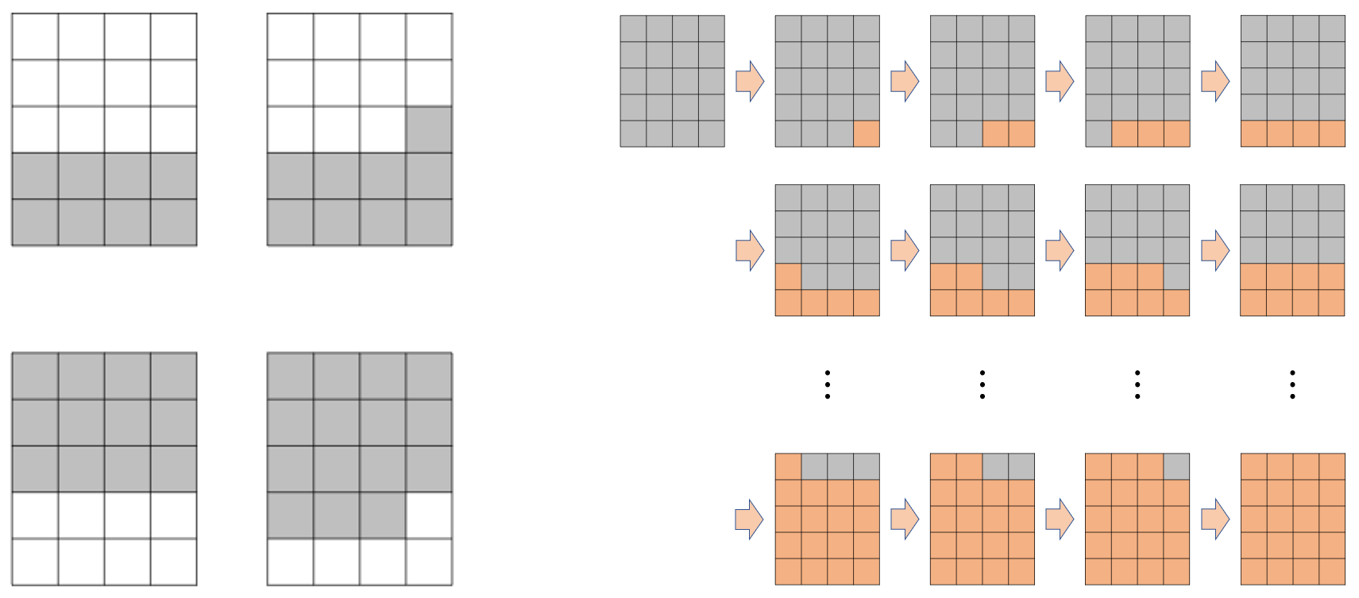}\caption{\label{fig4.1}In the figures in this article, white, gray, and orange
colors denote the spins $-1$, $0$, and $+1$, respectively. (Left)
canonical configurations for $(K,\,L)=(4,\,5)$; $\zeta_{2}^{+}$,
$\zeta_{2,1}^{+-}$ (upper-right), $\zeta_{3}^{-}$ (lower-left),
and $\zeta_{3,3}^{-+}$. (Right) a canonical path from $\mathbf{0}$
to $\mathbf{+1}$ for $(K,\,L)=(4,\,5)$.}
\end{figure}

\begin{defn}[Pre-canonical configurations and paths]
\label{d_precan} We define pre-canonical configurations between
$\mathbf{-1}$ and $\mathbf{0}$. We refer to Figure \ref{fig4.1}
(left) for an illustration.
\begin{itemize}
\item For $v\in\llbracket0,\,L\rrbracket$, we denote by $\zeta_{v}^{+}\in\mathcal{X}$
the spin configuration whose spins are $0$ on $\llbracket1,\,K\rrbracket\times\llbracket1,\,v\rrbracket$
and $-1$ on the remainder. Moreover, we denote by $\zeta_{v}^{-}\in\mathcal{X}$
the spin configuration whose spins are $0$ on $\llbracket1,\,K\rrbracket\times\llbracket L-v+1,\,L\rrbracket$
and $-1$ on the remainder. Hence, we have $\zeta_{0}^{+}=\zeta_{0}^{-}=\mathbf{-1}$
and $\zeta_{L}^{+}=\zeta_{L}^{-}=\mathbf{0}$. For $v\in\llbracket0,\,L\rrbracket$,
we write 
\begin{equation}
\mathcal{R}_{v}=\{\zeta_{v}^{+},\,\zeta_{v}^{-}\}.\label{e_Rv}
\end{equation}
\item For $v\in\llbracket0,\,L-1\rrbracket$ and $h\in\llbracket0,\,K\rrbracket$,
we denote by $\zeta_{v,h}^{++}\in\mathcal{X}$ the configuration whose
spins are $0$ on 
\[
\big[\llbracket1,\,K\rrbracket\times\llbracket1,\,v\rrbracket\big]\cup\big[\llbracket1,\,h\rrbracket\times\{v+1\}\big]
\]
and $-1$ on the remainder. Similarly, we denote by $\zeta_{v,h}^{+-}\in\mathcal{X}$
the configuration whose spins are $0$ on
\[
\big[\llbracket1,\,K\rrbracket\times\llbracket1,\,v\rrbracket\big]\cup\big[\llbracket K-h+1,\,K\rrbracket\times\{v+1\}\big]
\]
and $-1$ on the remainder. Namely, we obtain $\zeta_{v,h}^{++}$
(resp. $\zeta_{v,h}^{+-}$) from $\zeta_{v}^{+}$ by attaching a protuberance
of spin $0$ of size $h$ at its upper-left (resp. upper-right) corner
of the cluster of spin $0$. Similarly, we define $\zeta_{v,h}^{-+}$
and $\zeta_{v,h}^{--}$ by attaching a protuberance of spin $0$ of
size $h$ in $\zeta_{v}^{-}$. For $v\in\llbracket0,\,L-1\rrbracket$,
we write\textbf{
\begin{equation}
\mathcal{Q}_{v}=\bigcup_{h=1}^{K-1}\{\zeta_{v,h}^{++},\,\zeta_{v,h}^{+-},\,\zeta_{v,h}^{-+},\,\zeta_{v,h}^{--}\}.\label{e_Qv}
\end{equation}
}Concisely, $\mathcal{Q}_{v}$ consists of the configurations which
connect the ones in $\mathcal{R}_{v}$ and $\mathcal{R}_{v+1}$.
\item We define the collection $\mathcal{C}$ of\textit{ pre-canonical configurations
}as 
\[
\mathcal{C}=\bigcup_{v=0}^{L}\mathcal{R}_{v}\cup\bigcup_{v=0}^{L-1}\mathcal{Q}_{v}.
\]
\item Finally, a sequence $(\omega_{n})_{n=0}^{KL}$ of configurations is
a \textit{pre-canonical path} if it satisfies the following conditions;
see Figure \ref{fig4.1} (right).
\begin{itemize}
\item $\omega_{Kv}=\zeta_{v}^{+}$ for all $v\in\llbracket0,\,L\rrbracket$
\textit{(Type 1)} or $\omega_{Kv}=\zeta_{v}^{-}$ for all $v\in\llbracket0,\,L\rrbracket$
\textit{(Type 2)}.
\item \textit{(Type 1)} For each $v\in\llbracket0,\,L-1\rrbracket$, $\omega_{Kv+h}=\zeta_{v,h}^{++}$
for all $h\in\llbracket0,\,K\rrbracket$ or $\omega_{Kv+h}=\zeta_{v,h}^{+-}$
for all $h\in\llbracket0,\,K\rrbracket$.
\item \textit{(Type 2)} For each $v\in\llbracket0,\,L-1\rrbracket$, $\omega_{Kv+h}=\zeta_{v,h}^{-+}$
for all $h\in\llbracket0,\,K\rrbracket$ or $\omega_{Kv+h}=\zeta_{v,h}^{--}$
for all $h\in\llbracket0,\,K\rrbracket$.
\end{itemize}
\end{itemize}
We can readily verify that a pre-canonical path is indeed a path,
in the sense of Definition \ref{d_Ebarrier}. Moreover, pre-canonical
paths characterize all the possible paths from $\mathbf{-1}$ to $\mathbf{0}$
in $\mathcal{C}$ if $K<L$. However, more possible paths exist if
$K=L$; that is, the transposed pre-canonical paths.
\end{defn}

Based on this observation, we define canonical configurations and
paths between the ground states as follows:
\begin{defn}[Canonical configurations and paths]
\label{d_can} For two spins $a$ and $b$, we denote by $\mathcal{X}^{a,b}\subseteq\mathcal{X}$
the collection of configurations of which all spins are either $a$
or $b$. Then, we define the natural one-to-one correspondence $\Xi^{a,b}:\mathcal{X}^{-1,0}\rightarrow\mathcal{X}^{a,b}$
which maps spins $-1$ and $0$ to $a$ and $b$, respectively.

Now, we fix a good pair $(a,\,b)$ (cf. Notation \ref{n_good}). Then,
we divide into the cases of $K<L$ and $K=L$.
\begin{itemize}
\item \textbf{(Case $K<L$) }We define the collection $\mathcal{C}^{a,b}$
of \textit{canonical configurations} between $\mathbf{a}$ and $\mathbf{b}$
as
\[
\mathcal{C}^{a,b}=\Xi^{a,b}(\mathcal{C}).
\]
By symmetry, using $\Xi^{b,a}$ instead of $\Xi^{a,b}$ yields the
same result, so that $\mathcal{C}^{a,b}=\mathcal{C}^{b,a}$. Then,
we define (cf. \eqref{e_Rv} and \eqref{e_Qv})
\[
\mathcal{R}_{v}^{a,b}=\Xi^{a,b}(\mathcal{R}_{v})\;\;\;\;;\;v\in\llbracket0,\,L\rrbracket,\;\;\;\;\mathcal{Q}_{v}^{a,b}=\Xi^{a,b}(\mathcal{Q}_{v})\;\;\;\;;\;v\in\llbracket0,\,L-1\rrbracket.
\]
\item \textbf{(Case $K=L$)} We define a transpose operator $\Theta:\mathcal{X}\rightarrow\mathcal{X}$
by, for $\sigma\in\mathcal{X}$, 
\[
(\Theta(\sigma))(k,\,\ell)=\sigma(\ell,\,k)\;\;\;\;;\;k\in\llbracket1,\,K\rrbracket\text{ and }\ell\in\llbracket1,\,L\rrbracket.
\]
Then, we define the collection $\mathcal{C}^{a,b}$ of \textit{canonical
configurations} between $\mathbf{a}$ and $\mathbf{b}$ as
\[
\mathcal{C}^{a,b}=\Xi^{a,b}(\mathcal{C})\cup(\Theta\circ\Xi^{a,b})(\mathcal{C}).
\]
 We enlarge the collection of canonical configurations in this case,
because the transposed configurations also have the same energy due
to the condition $K=L$. Again, we have $\mathcal{C}^{a,b}=\mathcal{C}^{b,a}$.
Moreover, we define 
\begin{align*}
\mathcal{R}_{v}^{a,b} & =\Xi^{a,b}(\mathcal{R}_{v})\cup(\Theta\circ\Xi^{a,b})(\mathcal{R}_{v})\;\;\;\;;\;v\in\llbracket0,\,L\rrbracket,\\
\mathcal{Q}_{v}^{a,b} & =\Xi^{a,b}(\mathcal{Q}_{v})\cup(\Theta\circ\Xi^{a,b})(\mathcal{Q}_{v})\;\;\;\;;\;v\in\llbracket0,\,L-1\rrbracket.
\end{align*}
\end{itemize}
A sequence $(\omega_{n})_{n=0}^{KL}$ of configurations is a \textit{canonical
path} from $\mathbf{a}$ to $\mathbf{b}$ if there exists a pre-canonical
path $(\widetilde{\omega}_{n})_{n=0}^{KL}$ such that $\omega_{n}=\Xi^{a,b}(\widetilde{\omega}_{n})$
for all $n\in\llbracket0,\,KL\rrbracket$ (or additionally $\omega_{n}=(\Theta\circ\Xi^{a,b})(\widetilde{\omega}_{n})$
for all $n\in\llbracket0,\,KL\rrbracket$ if $K=L$).
\end{defn}

\begin{rem}
\label{r_H}It holds that $H(\sigma)\le\Gamma$ for all $\sigma\in\mathcal{C}^{-1,0}\cup\mathcal{C}^{0,+1}$
and 
\[
H(\sigma)=\begin{cases}
\Gamma-1 & \text{if }\sigma\in\mathcal{R}_{v}^{-1,0}\cup\mathcal{R}_{v}^{0,+1}\text{ for }v\in\llbracket1,\,L-1\rrbracket,\\
\Gamma & \text{if }\sigma\in\mathcal{Q}_{v}^{-1,0}\cup\mathcal{Q}_{v}^{0,+1}\text{ for }v\in\llbracket1,\,L-2\rrbracket.
\end{cases}
\]
These facts imply that canonical paths are $\Gamma$-paths.
\end{rem}

\begin{rem}
\label{r_can-1+1}One may be tempted to define similar objects between
$\mathbf{-1}$ and $\mathbf{+1}$ by choosing $(a,\,b)=(-1,\,+1)$
or $(+1,\,-1)$. However, the resulting configurations have too high
energy to be considered in our investigation. To explain this, recall
$\Xi^{-1,+1}:\mathcal{X}^{-1,0}\rightarrow\mathcal{X}^{-1,+1}$ from
Definition \ref{d_can}. Then, we can deduce that
\[
H(\sigma)=\begin{cases}
4\Gamma-4 & \text{if }\sigma\in\Xi^{-1,+1}(\mathcal{R}_{v})\text{ for }v\in\llbracket1,\,L-1\rrbracket,\\
4\Gamma & \text{if }\sigma\in\Xi^{-1,+1}(\mathcal{Q}_{v})\text{ for }v\in\llbracket1,\,L-2\rrbracket,
\end{cases}
\]
where $4\Gamma-4>\Gamma$. Hence, we cannot connect $\mathbf{-1}$
and $\mathbf{+1}$ by a direct canonical $\Gamma$-path, and thus
it is natural to expect that \textit{$\Gamma$-paths between $\mathbf{-1}$
and $\mathbf{+1}$ must visit at least a certain neighborhood of $\mathbf{0}$.}
Rigorously, this is exactly part (1) of Theorem \ref{t_eqp}.
\end{rem}

\subsection{\label{sec4.2}Proof of Theorem \ref{t_Ebarrier}}

Based on the canonical configurations, we are now ready to prove that
the energy barrier of the dynamics is exactly $\Gamma$.
\begin{proof}[Proof of Theorem \ref{t_Ebarrier}]
 First, we claim that for two spins $a$ and $b$,
\begin{equation}
\Gamma_{a,b}=\Phi(\mathbf{a},\,\mathbf{b})\le\Gamma.\label{e_EbarrierUB}
\end{equation}
Indeed, the canonical paths between $\mathbf{-1}$ and $\mathbf{0}$
assert that $\Gamma_{-1,0}=\Phi(\mathbf{-1},\,\mathbf{0})\le\Gamma$.
Similarly, the canonical paths between $\mathbf{0}$ and $\mathbf{+1}$
imply $\Gamma_{0,+1}\le\Gamma$. Hence,
\[
\Gamma_{-1,+1}=\Phi(\mathbf{-1},\,\mathbf{+1})\le\max\{\Phi(\mathbf{-1},\,\mathbf{0}),\,\Phi(\mathbf{0},\,\mathbf{+1})\}\le\Gamma.
\]
Thus, we get \eqref{e_EbarrierUB}. Therefore, to conclude the proof
of Theorem \ref{t_Ebarrier}, it suffices to prove that for distinct
spins $a$ and $b$,
\begin{equation}
\Gamma_{a,b}=\Phi(\mathbf{a},\,\mathbf{b})\ge\Gamma.\label{e_EbarrierLB}
\end{equation}
To provide a simple proof of \eqref{e_EbarrierLB}, we recall the
Metropolis dynamics of the 2D Potts model for $q=3$ with zero external
field \cite{Kim-Seo Ising-Potts,N-Z}. In this model, everything is
defined in the same way as in Section \ref{sec2.1}, except that the
Hamiltonian is given by
\begin{equation}
H_{\mathrm{Potts}}(\sigma)=\sum_{x\sim y}\mathbf{1}\{\sigma(x)\ne\sigma(y)\}\;\;\;\;;\;\sigma\in\mathcal{X}.\label{e_HPotts}
\end{equation}
Comparing this to our Hamiltonian \eqref{e_H}, we can easily notice
that 
\begin{equation}
H(\sigma)\ge H_{\mathrm{Potts}}(\sigma)\;\;\;\;;\;\sigma\in\mathcal{X}.\label{e_Hineq}
\end{equation}
Moreover, it is proved in \cite[Theorem 2.1]{N-Z} that the energy
barrier $\Phi_{\mathrm{Potts}}(\mathbf{s},\,\mathbf{s}')$, $\mathbf{s},\,\mathbf{s}'\in\mathcal{S}$,
of the Potts dynamics is exactly $\Gamma$. Therefore, as the energy
landscapes of the two models are identical, we deduce from \eqref{e_Hineq}
that
\[
\Phi(\mathbf{s},\,\mathbf{s}')\ge\Phi_{\mathrm{Potts}}(\mathbf{s},\,\mathbf{s}')=\Gamma\;\;\;\;;\;\mathbf{s},\,\mathbf{s}'\in\mathcal{S}.
\]
This is exactly \eqref{e_EbarrierLB}, and thus we conclude the proof
of Theorem \ref{t_Ebarrier}.
\end{proof}

\subsection{\label{sec4.3}Neighborhoods and configurations with small energy}

First, we review the concept of neighborhoods defined in \cite[Section 5]{Kim-Seo Ising-Potts}.
\begin{defn}[Neighborhoods]
\label{d_nbd} We define two types of neighborhoods of configurations
as in \cite[Definition 5.1]{Kim-Seo Ising-Potts}.
\begin{enumerate}
\item For $\sigma\in\mathcal{X}$, the neighborhoods $\mathcal{N}(\sigma)$
and $\widehat{\mathcal{N}}(\sigma)$ are defined as (cf. Definition
\ref{d_Ebarrier})
\begin{align*}
\mathcal{N}(\sigma) & =\{\zeta\in\mathcal{X}:\exists\text{a }(\Gamma-1)\text{-path }(\omega_{n})_{n=0}^{N}\text{ connecting }\sigma\text{ and }\zeta\},\\
\widehat{\mathcal{N}}(\sigma) & =\{\zeta\in\mathcal{X}:\exists\text{a }\Gamma\text{-path }(\omega_{n})_{n=0}^{N}\text{ connecting }\sigma\text{ and }\zeta\}.
\end{align*}
Then, for $\mathcal{A}\subseteq\mathcal{X}$, we define 
\[
\mathcal{N}(\mathcal{A})=\bigcup_{\sigma\in\mathcal{A}}\mathcal{N}(\sigma)\;\;\;\;\text{and\;\;\;\;}\widehat{\mathcal{N}}(\mathcal{A})=\bigcup_{\sigma\in\mathcal{A}}\mathcal{\widehat{\mathcal{N}}}(\sigma).
\]
\item Let $\mathcal{B}\subseteq\mathcal{X}$. For $\sigma\in\mathcal{X}$
with $\sigma\notin\mathcal{B}$, we define
\[
\widehat{\mathcal{N}}(\sigma;\mathcal{B})=\{\zeta\in\mathcal{X}:\exists\text{a }\Gamma\text{-path in }\mathcal{X}\setminus\mathcal{B}\text{ connecting }\sigma\text{ and }\zeta\}.
\]
Then, for $\mathcal{A}\subseteq\mathcal{X}$ disjoint with $\mathcal{B}$,
we define 
\[
\widehat{\mathcal{N}}(\mathcal{A};\mathcal{B})=\bigcup_{\sigma\in\mathcal{A}}\widehat{\mathcal{N}}(\sigma;\mathcal{B}).
\]
\end{enumerate}
\end{defn}

With these notions in mind, as $\beta\rightarrow\infty$, the only
configurations relevant to the study of metastability are those in
$\widehat{\mathcal{N}}(\mathcal{S})$ (in view of Theorem \ref{t_Ebarrier}).
Indeed, if we take $\sigma\in\mathcal{X}$ with $H(\sigma)>\Gamma$,
then by \eqref{e_detbal} it holds that, for any $\zeta\in\mathcal{X}$
with $\zeta\sim\sigma$,
\[
\mu_{\beta}(\sigma)c_{\beta}(\sigma,\,\zeta)=\mu_{\beta}(\zeta)c_{\beta}(\zeta,\,\sigma)\le\mu_{\beta}(\sigma)=O(e^{-\beta(\Gamma+1)}).
\]
This implies that any spin updates associated with $\sigma$ are irrelevant
to the study of metastability on the scale $e^{\beta\Gamma}$. Hence,
$\widehat{\mathcal{N}}(\mathcal{S})$ is the main object in our study
of the energy landscape.

The following lemma, which is a generalization of \cite[Lemma 5.2]{Kim-Seo Ising-Potts},
is useful to investigate the $\widehat{\mathcal{N}}$-neighborhoods.
We can prove this lemma in the same manner, and thus we omit it.
\begin{lem}
\label{l_set}Suppose that $\mathcal{A}$, $\mathcal{A}'$, and $\mathcal{B}$
are pairwise disjoint subsets of $\mathcal{X}$. Then, we have
\[
\widehat{\mathcal{N}}(\mathcal{A}\cup\mathcal{A}';\mathcal{B})=\widehat{\mathcal{N}}(\mathcal{A}';\mathcal{A}\cup\mathcal{B})\cup\widehat{\mathcal{N}}(\mathcal{A};\mathcal{A}'\cup\mathcal{B}).
\]
In particular, if $\mathcal{B}=\emptyset$, then we have $\widehat{\mathcal{N}}(\mathcal{A}\cup\mathcal{A}')=\widehat{\mathcal{N}}(\mathcal{A}';\mathcal{A})\cup\widehat{\mathcal{N}}(\mathcal{A};\mathcal{A}')$.
\end{lem}

We verified in Section \ref{sec4.2} that the energy barrier is exactly
$\Gamma$. Now, we fully characterize the spin configurations with
energy less than $\Gamma$. This result is an analogue of \cite[Proposition 6.8]{Kim-Seo Ising-Potts}
and can be proved in a similar manner; thus, we omit the proof. We
refer to Figure \ref{fig4.2} for some examples of such configurations.

\begin{figure}
\includegraphics[width=13cm]{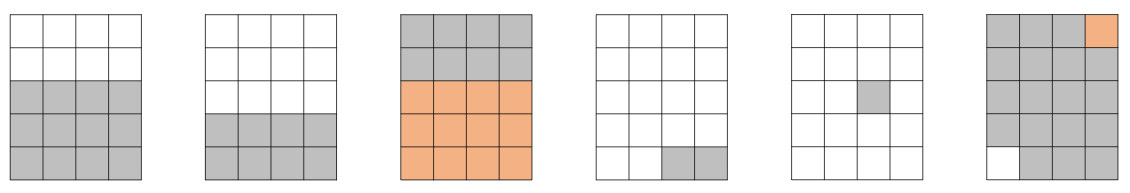}\caption{\label{fig4.2}Configurations with energy smaller than $\Gamma$:
type \textbf{(T1)} (the first three) and type \textbf{(T2)} (the last
three).}
\end{figure}

\begin{prop}
\label{p_E<Gamma}Suppose that $\sigma\in\mathcal{X}$ satisfies $H(\sigma)<\Gamma$.
Then, exactly one of \textbf{\textup{(T1)}} or \textbf{\textup{(T2)}}
below holds.
\begin{enumerate}
\item[\textbf{(T1)}]  There exist a good pair $(a,\,b)$ and $v\in\llbracket2,\,L-2\rrbracket$
such that $\sigma\in\mathcal{R}_{v}^{a,b}$. In particular, $\mathcal{N}(\sigma)$
is a singleton, i.e., $\mathcal{N}(\sigma)=\{\sigma\}$.
\item[\textbf{(T2)}]  The configuration $\sigma$ belongs to $\mathcal{N}(\mathbf{a})$
for exactly one spin $a$, so that $\mathcal{N}(\sigma)=\mathcal{N}(\mathbf{a})$.
\end{enumerate}
\end{prop}

\subsection{\label{sec4.4}Proof of Theorem \ref{t_LDTresults}}

In this subsection, we prove Theorem \ref{t_LDTresults}. To this
end, we need the following result regarding the valley depths of the
entire energy landscape.
\begin{lem}
\label{l_depth}We have the following upper bounds for the depths
of the valleys:
\begin{enumerate}
\item For all $\sigma\in\mathcal{X}$ and $\mathbf{s}\in\mathcal{S}$, it
holds that $\Phi(\sigma,\,\mathbf{s})-H(\sigma)\le\Gamma$.
\item For all $\sigma\in\mathcal{X}\setminus\mathcal{S}$, it holds that
$\Phi(\sigma,\,\mathcal{S})-H(\sigma)<\Gamma$.
\end{enumerate}
\end{lem}

\begin{proof}
The same assertions for the Metropolis dynamics on the Potts model
are proved in \cite[Theorem 2.1]{N-Z}. Because the same arguments
work for our Blume--Capel model as well, we omit the proof.
\end{proof}
\begin{rem}
An alternative proof can be found in \cite[Lemma 6.11]{Kim-Seo Ising-Potts}
which provides an explicit path that guarantees the upper bounds stated
in Lemma \ref{l_depth}.
\end{rem}

Based on the previous lemma, we give a formal proof of Theorem \ref{t_LDTresults}.
\begin{proof}[Proof of Theorem \ref{t_LDTresults}]
 By the general theory developed in \cite{N-Z,N-Z-B}, Theorem \ref{t_Ebarrier}
and Lemma \ref{l_depth} are sufficient to conclude the assertions
on the transition time, mixing time, and spectral gap given in Theorem
\ref{t_LDTresults}.
\end{proof}

\section{\label{sec5}Typical and Gateway Configurations}

In this section, we define the concepts of typical and gateway configurations
and investigate their several basic properties. The concepts are analogues
of those defined in \cite[Section 7]{Kim-Seo Ising-Potts}. We note
that even though the results are similar to those in \cite{Kim-Seo Ising-Potts},
we still thoroughly review the notation here because \textit{there
indeed exist technical differences due to the non-symmetry of the
Blume--Capel model (cf. Remark \ref{r_modelsym}).}

\subsection{\label{sec5.1}Typical configurations}

\begin{figure}
\includegraphics[width=13cm]{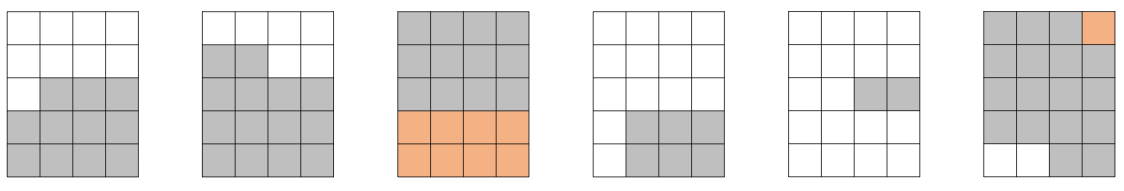}\caption{\label{fig5.1}Typical configurations: bulk ones (the first three)
and edge ones (the last three).}
\end{figure}

\begin{defn}[Typical configurations]
\label{d_typ} Here, we define typical configurations. We refer to
Figure \ref{fig5.1} for a visualization.
\begin{itemize}
\item Fix a good pair $(a,\,b)$. The collection of \textit{bulk typical
configurations} between $\mathbf{a}$ and $\mathbf{b}$ is defined
as
\begin{equation}
\mathcal{B}^{a,b}=\bigcup_{v=2}^{L-2}\mathcal{R}_{v}^{a,b}\cup\bigcup_{v=2}^{L-3}\mathcal{Q}_{v}^{a,b}.\label{e_Bab}
\end{equation}
Moreover, we define (cf. Remark \ref{r_H})
\[
\mathcal{B}_{\Gamma}^{a,b}=\bigcup_{v=2}^{L-3}\mathcal{Q}_{v}^{a,b}=\{\sigma\in\mathcal{B}^{a,b}:H(\sigma)=\Gamma\}.
\]
Clearly, we have $\mathcal{B}^{a,b}=\mathcal{B}^{b,a}$ and $\mathcal{B}_{\Gamma}^{a,b}=\mathcal{B}_{\Gamma}^{b,a}$.
\item For a spin $a$, the collection of \textit{edge typical configurations}
near $\mathbf{a}$ is defined as
\begin{equation}
\mathcal{E}^{a}=\widehat{\mathcal{N}}(\mathbf{a};\mathcal{B}_{\Gamma}^{-1,0}\cup\mathcal{B}_{\Gamma}^{0,+1}).\label{e_Ea}
\end{equation}
\item Finally, the collection of \textit{typical configurations} is defined
as
\begin{equation}
\mathcal{T}=\mathcal{B}^{-1,0}\cup\mathcal{B}^{0,+1}\cup\mathcal{E}^{-1}\cup\mathcal{E}^{0}\cup\mathcal{E}^{+1}.\label{e_T}
\end{equation}
\end{itemize}
\end{defn}

Then, we summarize the following properties for the typical configurations.
Rigorous verifications can be found in \cite[Section 7.2]{Kim-Seo Ising-Potts}
and thus we do not repeat them.
\begin{prop}
\label{p_typ}The following properties hold for the typical configurations.
\begin{enumerate}
\item The collections $\mathcal{E}^{-1}$, $\mathcal{E}^{0}$, and $\mathcal{E}^{+1}$
are disjoint.
\item We have
\begin{align}
\mathcal{E}^{-1}\cap\mathcal{B}^{-1,0} & =\mathcal{R}_{2}^{-1,0},\;\;\;\;\mathcal{E}^{0}\cap\mathcal{B}^{-1,0}=\mathcal{R}_{L-2}^{-1,0},\label{e_typ1}\\
\mathcal{E}^{+1}\cap\mathcal{B}^{0,+1} & =\mathcal{R}_{L-2}^{0,+1},\;\;\;\;\mathcal{E}^{0}\cap\mathcal{B}^{0,+1}=\mathcal{R}_{2}^{0,+1}.\label{e_typ2}
\end{align}
\item We have $\mathcal{E}^{-1}\cap\mathcal{B}^{0,+1}=\mathcal{E}^{+1}\cap\mathcal{B}^{-1,0}=\emptyset$.
\item Recall the definition \eqref{e_T} of $\mathcal{T}$. Then, $\widehat{\mathcal{N}}(\mathcal{S})=\mathcal{T}$.
\end{enumerate}
\end{prop}

\begin{figure}
\includegraphics[width=13cm]{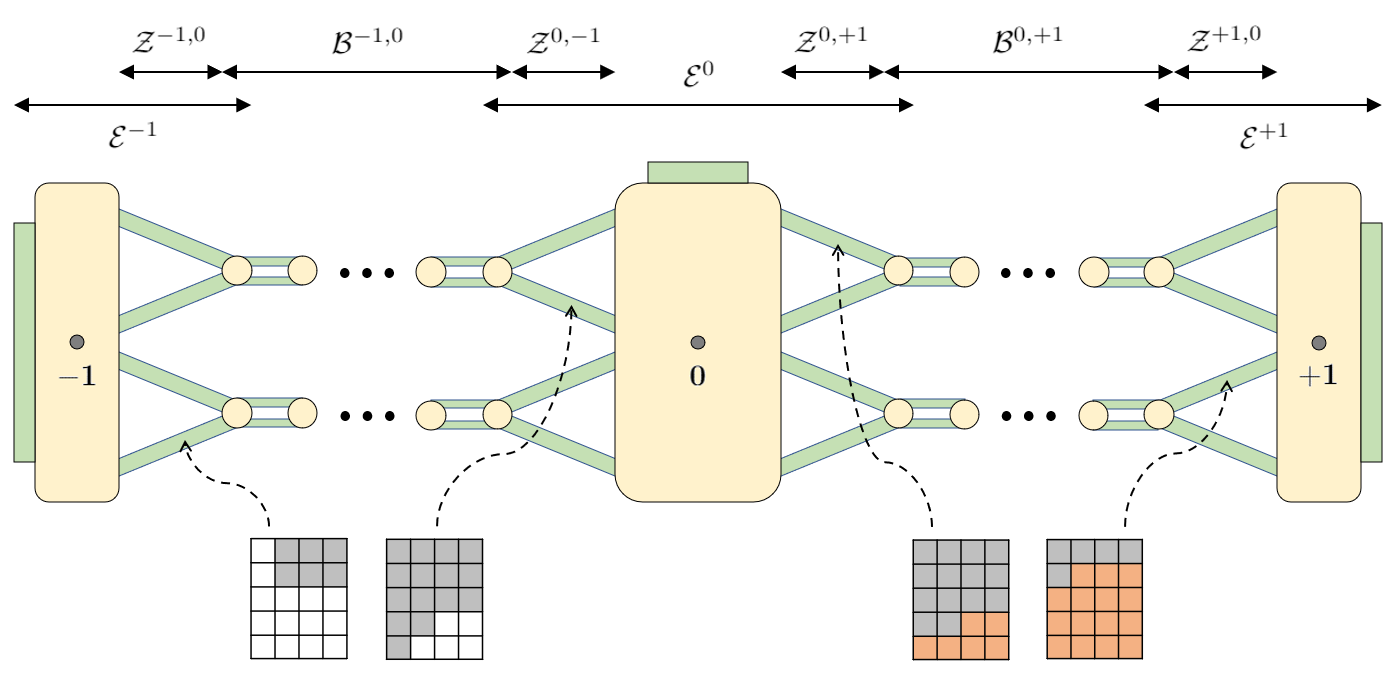}\caption{\label{fig5.2}Energy landscape of $\widehat{\mathcal{N}}(\mathcal{S})$
for the case of $K<L$. Green regions represent the configurations
with energy exactly $\Gamma$, and yellow regions represent the ones
with energy less than $\Gamma$. Configurations below are examples of gateway configurations.}
\end{figure}

\begin{rem}[Edge structure of typical configurations]
\label{r_typedge} Based on Proposition \ref{p_typ}, we have the
following decomposition of $E(\widehat{\mathcal{N}}(\mathcal{S}))=E(\mathcal{T})$
(see Figure \ref{fig5.2} for the full energy landscape):
\[
E(\widehat{\mathcal{N}}(\mathcal{S}))=E(\mathcal{B}^{-1,0})\cup E(\mathcal{B}^{0,+1})\cup E(\mathcal{E}^{-1})\cup E(\mathcal{E}^{0})\cup E(\mathcal{E}^{+1}).
\]
To prove this fact, we check that the members constituting $\mathcal{T}$
(cf. \eqref{e_T}) are \textit{separated}, in the sense that for members
$\mathcal{A}$ and $\mathcal{A}'$,
\[
\{\sigma,\,\sigma'\}\in E(\mathcal{A}\cup\mathcal{A}')\;\;\;\;\text{implies}\;\;\;\;\sigma,\,\sigma'\in\mathcal{A}\text{ or }\sigma,\,\sigma'\in\mathcal{A}'.
\]
Indeed, $\mathcal{E}^{a}$ for spins $a$ are separated by part (1)
of Proposition \ref{p_typ}. The collections $\mathcal{B}^{-1,0}$
and $\mathcal{B}^{0,+1}$ are clearly separated.

To check that a bulk collection $\mathcal{B}^{a,b}$ and an edge collection
$\mathcal{E}^{a'}$ are separated, it suffices to prove that if $\sigma\in\mathcal{B}^{a,b}$
and $\sigma'\in\mathcal{E}^{a'}\setminus\mathcal{B}^{a,b}$ with $\sigma\sim\sigma'$,
then $\sigma\in\mathcal{E}^{a'}$. To this end, as $\sigma'\notin\mathcal{B}^{a,b}$,
we must have $\sigma\in\mathcal{R}_{2}^{a,b}$ or $\sigma\in\mathcal{R}_{L-2}^{a,b}$.
For the former case, as $\mathcal{R}_{2}^{a,b}\subseteq\mathcal{E}^{a}$,
by part (1) of Proposition \ref{p_typ} we obtain $a=a'$ and thus
$\sigma\in\mathcal{E}^{a'}$. For the latter case, as $\mathcal{R}_{L-2}^{a,b}\subseteq\mathcal{E}^{b}$,
we obtain $b=a'$ and thus $\sigma\in\mathcal{E}^{a'}$.
\end{rem}

\subsection{\label{sec5.2}Gateway configurations}

Here, we define gateway configurations of the dynamics. We again refer
to Figure \ref{fig5.2} for a visualization of the role and examples
of gateway configurations.
\begin{defn}[Gateway configurations]
\label{d_gate} As for the typical configurations, we define gateway
configurations between $\mathbf{a}$ and $\mathbf{b}$ for good pairs
$(a,\,b)$. Thus, we fix a good pair $(a,\,b)$. We define $\mathcal{Z}^{a,b}$
as
\begin{align}
\{\sigma\in\mathcal{X}: & \exists\text{a path }(\omega_{n})_{n=0}^{N}\text{ in }\mathcal{X}\setminus\mathcal{B}_{\Gamma}^{a,b}\text{ with }N\ge1\text{ such that}\nonumber \\
 & \omega_{0}\in\mathcal{R}_{2}^{a,b},\;\omega_{N}=\sigma,\;\text{and}\;H(\omega_{n})=\Gamma\text{ for all }n\in\llbracket1,\,N\rrbracket\}.\label{e_Zab}
\end{align}
Note that $\mathcal{Z}^{a,b}\neq\mathcal{Z}^{b,a}$. Then, we define
the collection of \textit{gateway configurations} between $\mathbf{a}$
and $\mathbf{b}$ as
\begin{equation}
\mathcal{G}^{a,b}=\mathcal{Z}^{a,b}\cup\mathcal{B}^{a,b}\cup\mathcal{Z}^{b,a},\label{e_gate}
\end{equation}
which is indeed a decomposition of $\mathcal{G}^{a,b}$. As $\mathcal{B}^{a,b}=\mathcal{B}^{b,a}$,
we have $\mathcal{G}^{a,b}=\mathcal{G}^{b,a}$.
\end{defn}

Then, we have the following properties for the gateway configurations.
\begin{lem}
\label{l_gate}Fix a good pair $(a,\,b)$ and suppose that $\sigma,\,\zeta\in\mathcal{X}$
satisfy
\[
\sigma\in\mathcal{G}^{a,b},\;\zeta\notin\mathcal{G}^{a,b},\;\sigma\sim\zeta,\;\text{and}\;H(\zeta)\le\Gamma.
\]
Then, we have either $\zeta\in\mathcal{N}(\mathbf{a})$ and $\sigma\in\mathcal{Z}^{a,b}$
or $\zeta\in\mathcal{N}(\mathbf{b})$ and $\sigma\in\mathcal{Z}^{b,a}$.
\end{lem}

\begin{proof}
This lemma can be proved in an identical manner to \cite[Lemma 8.5]{Kim-Seo Ising-Potts}.
\end{proof}

\subsection{\label{sec5.4}Lemma on equilibrium potentials and proof of Theorem
\ref{t_eqp}}

In this subsection, we prove Theorem \ref{t_eqp}. Before providing
the proof, we give an elementary estimate on equilibrium potentials
(cf. \eqref{e_eqpdef}), which is a generalization of \cite[Lemma 10.4]{Kim-Seo Ising-Potts}.
This lemma is used in the proof of Theorem \ref{t_eqp} and later
in Section \ref{sec8} to estimate the test flow. We refer to \cite[Lemmas 10.4 and 16.5]{Kim-Seo Ising-Potts}
for the proof.
\begin{lem}
\label{l_eqp}For disjoint and non-empty subsets $\mathcal{A}$ and
$\mathcal{B}$ of $\mathcal{S}$, there exists $C=C(K,\,L)>0$ such
that for all $\mathbf{s}\in\mathcal{S}$, 
\begin{equation}
\max_{\zeta\in\mathcal{N}(\mathbf{s})}\big|\mathbb{P}_{\zeta}^{\beta}[\tau_{\mathcal{A}}<\tau_{\mathcal{B}}]-\mathbb{P}_{\mathbf{s}}^{\beta}[\tau_{\mathcal{A}}<\tau_{\mathcal{B}}]\big|\le Ce^{-\beta}.\label{e_eqp}
\end{equation}
\end{lem}

Then, we provide a proof of Theorem \ref{t_eqp}.
\begin{proof}[Proof of Theorem \ref{t_eqp}]
 Part (2) is obvious from the model symmetry. Thus, to conclude the
proof, we prove part (1). We first prove that
\begin{equation}
\lim_{\beta\rightarrow\infty}\mathbb{P}_{\mathbf{-1}}^{\beta}[\tau_{\mathcal{N}(\mathbf{0})}<\tau_{\mathbf{+1}}]=1.\label{e_eqppf1}
\end{equation}
We denote by $\tau^{*}$ the hitting time of the set $\{\sigma\in\mathcal{X}:H(\sigma)\ge\Gamma+1\}$.
Then, \cite[Theorem 3.2]{N-Z-B} implies that 
\[
\mathbb{P}_{\mathbf{-1}}^{\beta}[\tau^{*}>e^{\beta(\Gamma+1/2)}]=1-o(1).
\]
Hence, by part (1) of Theorem \ref{t_LDTresults} with $\epsilon=1/2$,
we have
\begin{align*}
\mathbb{P}_{\mathbf{-1}}^{\beta}[\tau_{\mathbf{+1}}<\tau^{*}]=1-\mathbb{P}_{\mathbf{-1}}^{\beta}[\tau_{\mathbf{+1}}\ge\tau^{*}] & =1-o(1)-\mathbb{P}_{\mathbf{-1}}^{\beta}[\tau_{\mathbf{+1}}\ge\tau^{*}>e^{\beta(\Gamma+1/2)}]\\
 & \ge1-o(1)-\mathbb{P}_{\mathbf{-1}}^{\beta}[\tau_{\mathbf{+1}}>e^{\beta(\Gamma+1/2)}]=1-o(1).
\end{align*}
Therefore, it suffices to prove that a $\Gamma$-path from $\mathbf{-1}$
to $\mathbf{+1}$ must visit $\mathcal{N}(\mathbf{0})$. To this end,
we fix a $\Gamma$-path $(\omega_{n})_{n=0}^{N}$ with $\omega_{0}=\mathbf{-1}\in\mathcal{E}^{-1}$
and $\omega_{N}=\mathbf{+1}\in\mathcal{E}^{+1}$. Then, by Proposition
\ref{p_typ} and Remark \ref{r_typedge}, starting from $\mathbf{-1}\in\mathcal{E}^{-1}$,
this path must successively visit $\mathcal{E}^{-1}\cap\mathcal{B}^{-1,0}=\mathcal{R}_{2}^{-1,0}$,
$\mathcal{B}^{-1,0}$, $\mathcal{B}^{-1,0}\cap\mathcal{E}^{0}=\mathcal{R}_{2}^{0,-1}$,
$\mathcal{E}^{0}$, $\mathcal{E}^{0}\cap\mathcal{B}^{0,+1}=\mathcal{R}_{2}^{0,+1}$,
$\mathcal{B}^{0,+1}$, and $\mathcal{B}^{0,+1}\cap\mathcal{E}^{+1}=\mathcal{R}_{2}^{+1,0}$
to finally arrive at $\mathbf{+1}\in\mathcal{E}^{+1}$. Thus, the
following time is well defined:
\[
n_{0}=\max\{n:\omega_{n}\in\mathcal{R}_{2}^{0,-1}\}.
\]
Then, by the definition of gateway configurations, we have $\omega_{n_{0}+1}\in\mathcal{Z}^{0,-1}$.
Then, by defining
\[
n_{1}=\min\{n>n_{0}:\omega_{n}\notin\mathcal{G}^{0,-1}\},
\]
we have $\omega_{n_{1}}\in\mathcal{N}(\mathbf{0})$ by Lemma \ref{l_gate},
which concludes the proof of \eqref{e_eqppf1}.

Moreover, Lemma \ref{l_eqp} with $\mathcal{A}=\{\mathbf{0}\}$, $\mathcal{B}=\{\mathbf{+1}\}$,
and $\mathbf{s}=\mathbf{0}$ implies that
\[
\max_{\sigma\in\mathcal{N}(\mathbf{0})}\mathbb{P}_{\sigma}^{\beta}[\tau_{\mathbf{0}}<\tau_{\mathbf{+1}}]=1-o(1).
\]
This is equivalent to
\begin{equation}
\lim_{\beta\rightarrow\infty}\max_{\sigma\in\mathcal{N}(\mathbf{0})}\mathbb{P}_{\sigma}^{\beta}[\tau_{\mathbf{0}}<\tau_{\mathbf{+1}}]=1.\label{e_eqppf2}
\end{equation}
Therefore, we conclude the proof of the first assertion of part (1)
with \eqref{e_eqppf1} and \eqref{e_eqppf2} by the casual argument
using the strong Markov property. The second assertion follows identically.
\end{proof}

\section{\label{sec6}Edge Typical Configurations}

In this section, we focus on the edge typical configurations defined
in Definition \ref{d_typ}, which have much more complex geometry
than the bulk typical configurations. This section is an analogue
of \cite[Section 7.3]{Kim-Seo Ising-Potts}, but \textit{we provide
here a much more detailed and quantitative analysis on the behavior
of the edge typical configurations.}

\begin{figure}
\includegraphics[width=13cm]{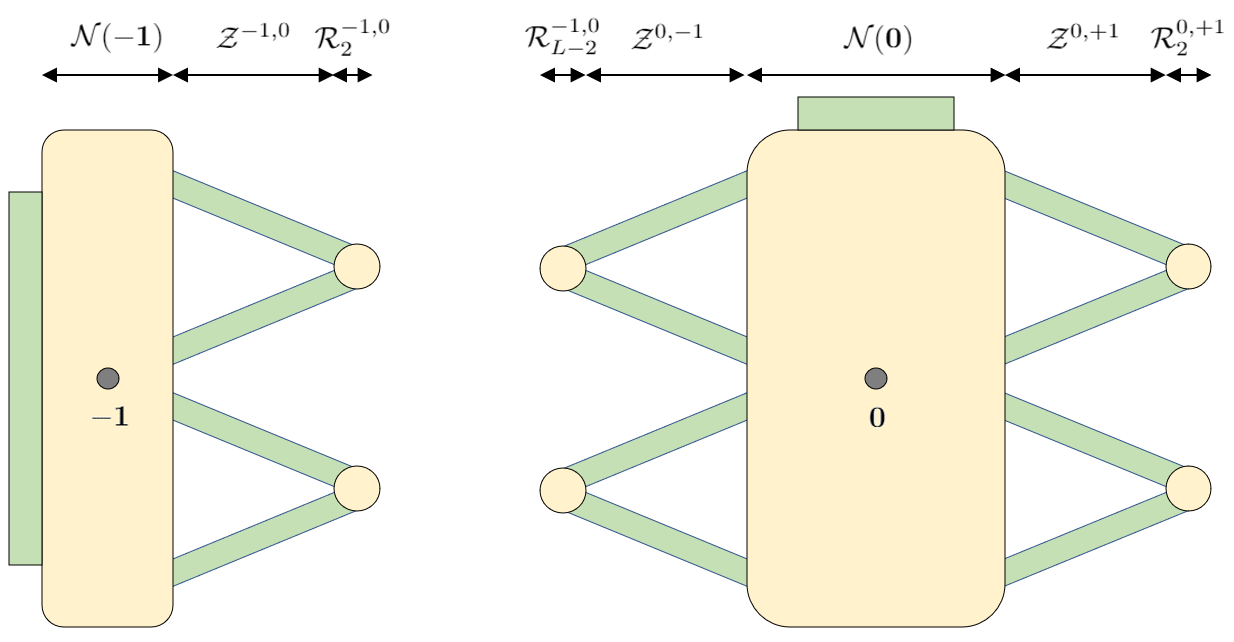}\caption{\label{fig6.1}Edge typical configurations when $K<L$.
(Left) structure of $\mathcal{E}^{-1}$. (Right) structure of $\mathcal{E}^{0}$.}
\end{figure}

\subsection{Projected graph}

We consistently refer to Figure \ref{fig6.1} for an illustration
of the notions defined in this subsection. For each spin $a$, we
decompose $\mathcal{E}^{a}=\mathcal{I}^{a}\cup\mathcal{O}^{a}$ where
\[
\mathcal{O}^{a}=\{\sigma\in\mathcal{E}^{a}:H(\sigma)=\Gamma\}\;\;\;\;\text{and}\;\;\;\;\mathcal{I}^{a}=\{\sigma\in\mathcal{E}^{a}:H(\sigma)<\Gamma\}.
\]
By Proposition \ref{p_E<Gamma}, we notice that 
\begin{equation}
\mathcal{I}^{a}=\mathcal{N}(\mathbf{a})\cup\Big[\bigcup_{b:\,(a,b)\text{ is good}}\mathcal{R}_{2}^{a,b}\Big].\label{e_Ia}
\end{equation}
We further define 
\begin{equation}
\mathcal{I}_{\mathfrak{0}}^{a}=\{\mathbf{a}\}\cup\Big[\bigcup_{b:\,(a,b)\text{ is good}}\mathcal{R}_{2}^{a,b}\Big],\label{e_Iarep}
\end{equation}
so that each $\sigma\in\mathcal{I}^{a}$ satisfies $\sigma\in\mathcal{N}(\zeta)$
for exactly one $\zeta\in\mathcal{I}_{\mathfrak{0}}^{a}$. Hence,
we get the following alternative decomposition of $\mathcal{E}^{a}$:
\begin{equation}
\mathcal{E}^{a}=\mathcal{O}^{a}\cup\Big[\bigcup_{\zeta\in\mathcal{I}_{\mathfrak{0}}^{a}}\mathcal{N}(\zeta)\Big].\label{e_Eadec}
\end{equation}
We chose the set of representatives $\mathcal{I}_{\mathfrak{0}}^{a}$
because configurations belonging to the same $\mathcal{N}$-neighborhood
are not distinguished in the study of metastability, in the sense
of Lemma \ref{l_eqp}.
\begin{rem}
\label{r_Ia}We remark on the display \eqref{e_Ia}. In details, we
have
\[
\mathcal{I}^{-1}=\mathcal{N}(\mathbf{-1})\cup\mathcal{R}_{2}^{-1,0}\;\;\;\;\text{and}\;\;\;\;\mathcal{I}^{+1}=\mathcal{N}(\mathbf{+1})\cup\mathcal{R}_{L-2}^{0,+1},
\]
whereas
\[
\mathcal{I}^{0}=\mathcal{N}(\mathbf{0})\cup\mathcal{R}_{L-2}^{-1,0}\cup\mathcal{R}_{2}^{0,+1}.
\]
Hence, the structures of $\mathcal{E}^{-1}$ and $\mathcal{E}^{+1}$
are exactly the same, but they differ from the structure of $\mathcal{E}^{0}$.
Figure \ref{fig6.1} illustrates this difference.
\end{rem}

Now, we define a graph structure on $\mathcal{O}^{a}\cup\mathcal{I}_{\mathfrak{0}}^{a}$.
\begin{defn}
\label{d_Ga}We fix spin $a$ and introduce a graph structure and
a Markov chain on $\mathcal{O}^{a}\cup\mathcal{I}_{\mathfrak{0}}^{a}$.
\begin{itemize}
\item \textbf{(Graph) }Vertex set $\mathscr{V}^{a}$ is defined by
\begin{equation}
\mathscr{V}^{a}=\mathcal{O}^{a}\cup\mathcal{I}_{\mathfrak{0}}^{a}.\label{e_Va}
\end{equation}
Then, the edge set $E(\mathscr{V}^{a})$ is defined as follows: $\{\sigma,\,\sigma'\}\in E(\mathscr{V}^{a})$
if and only if either $\sigma,\,\sigma'\in\mathcal{O}^{a}$ and $\sigma\sim\sigma'$,
or $\sigma\in\mathcal{O}^{a}$, $\sigma'\in\mathcal{I}_{\mathfrak{0}}^{a}$,
and $\sigma\sim\zeta$ for some $\zeta\in\mathcal{N}(\sigma')$.
\item \textbf{(Markov chain)} We define a transition rate $r^{\mathbf{a}}:\mathscr{V}^{a}\times\mathscr{V}^{a}\rightarrow[0,\,\infty)$
as follows: If $\{\sigma,\,\sigma'\}\notin E(\mathscr{V}^{a})$, then
$r^{\mathbf{a}}(\sigma,\,\sigma')=0$. If $\{\sigma,\,\sigma'\}\in E(\mathscr{V}^{a})$,
then 
\begin{equation}
r^{\mathbf{a}}(\sigma,\,\sigma')=\begin{cases}
1 & \text{if }\sigma,\,\sigma'\in\mathcal{O}^{a},\\
|\{\zeta\in\mathcal{N}(\sigma):\zeta\sim\sigma'\}| & \text{if }\sigma\in\mathcal{I}_{\mathfrak{0}}^{a},\;\sigma'\in\mathcal{O}^{a},\\
|\{\zeta\in\mathcal{N}(\sigma'):\zeta\sim\sigma\}| & \text{if }\sigma\in\mathcal{O}^{a},\;\sigma'\in\mathcal{I}_{\mathfrak{0}}^{a}.
\end{cases}\label{e_ra}
\end{equation}
Then, we define $\{Z^{a}(t)\}_{t\ge0}$ as the continuous-time Markov
chain on $\mathscr{V}^{a}$ with transition rate $r^{\mathbf{a}}(\cdot,\,\cdot)$.
As the rate is symmetric, the Markov chain $Z^{a}(\cdot)$ is reversible
with respect to its invariant distribution, which is the uniform distribution
on $\mathscr{V}^{a}$.
\end{itemize}
\end{defn}

Next, we prove that the process $Z^{a}(\cdot)$ approximates the Metropolis
dynamics on the edge typical configurations.
\begin{prop}
\label{p_proj}For each spin $a$, define a projection map $\Pi^{\mathbf{a}}:\mathcal{E}^{a}\rightarrow\mathscr{V}^{a}$
by
\[
\Pi^{\mathbf{a}}(\sigma)=\begin{cases}
\sigma & \text{if }\sigma\in\mathcal{O}^{a},\\
\zeta & \text{if }\sigma\in\mathcal{N}(\zeta)\text{ for some }\zeta\in\mathcal{I}_{\mathfrak{0}}^{a}.
\end{cases}
\]
Then, there exists a constant $C=C(K,\,L)>0$ such that
\begin{enumerate}
\item for $\sigma_{1},\,\sigma_{2}\in\mathcal{O}^{a}$, we have
\[
\Big|\frac{1}{3}e^{-\beta\Gamma}r^{\mathbf{a}}(\Pi^{\mathbf{a}}(\sigma_{1}),\,\Pi^{\mathbf{a}}(\sigma_{2}))-\mu_{\beta}(\sigma_{1})c_{\beta}(\sigma_{1},\,\sigma_{2})\Big|\le Ce^{-\beta(\Gamma+1)},
\]
\item for $\sigma_{1}\in\mathcal{O}^{a}$ and $\sigma_{2}\in\mathcal{I}_{\mathfrak{0}}^{a}$,
we have
\[
\Big|\frac{1}{3}e^{-\beta\Gamma}r^{\mathbf{a}}(\Pi^{\mathbf{a}}(\sigma_{1}),\,\Pi^{\mathbf{a}}(\sigma_{2}))-\sum_{\zeta\in\mathcal{N}(\sigma_{2})}\mu_{\beta}(\sigma_{1})c_{\beta}(\sigma_{1},\,\zeta)\Big|\le Ce^{-\beta(\Gamma+1)}.
\]
\end{enumerate}
\end{prop}

\begin{proof}
As the proof is identical to that of \cite[Proposition 7.7]{Kim-Seo Ising-Potts},
we omit the details.
\end{proof}

\subsection{Approximation to auxiliary process}

In this subsection, we prove that the auxiliary process analyzed in
Section \ref{secA.1} successfully represents the Markov chain $Z^{a}(\cdot)$.
First, we handle the case of $K<L$.

\begin{figure}
\includegraphics[width=13cm]{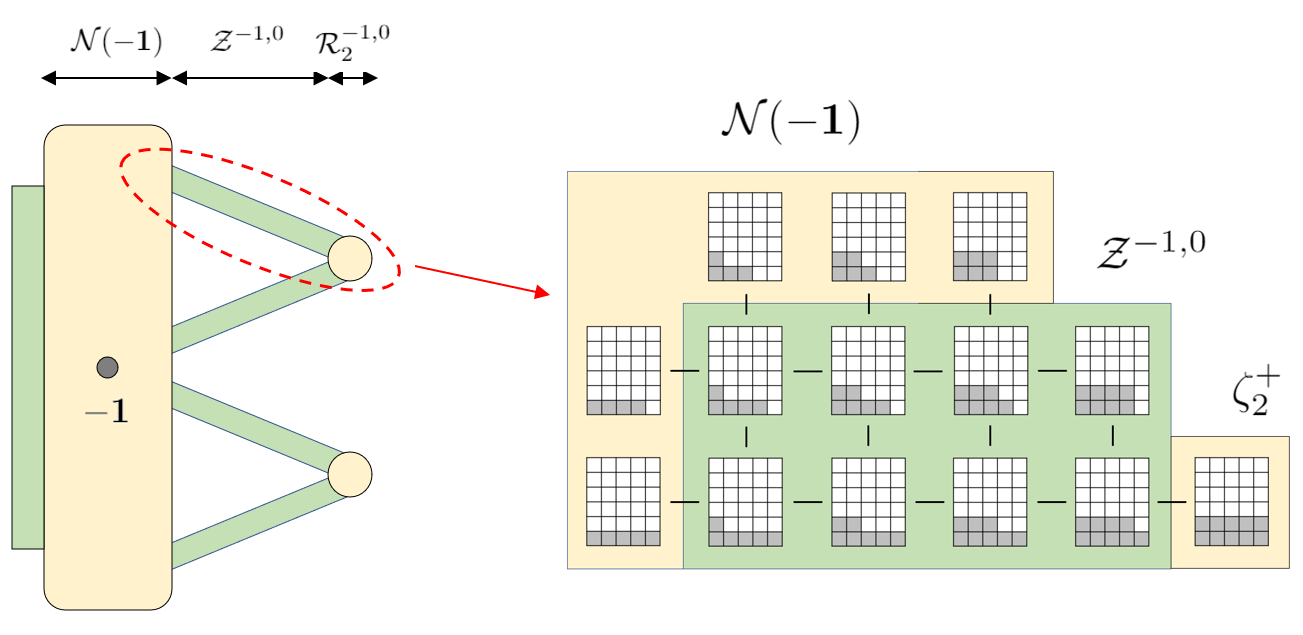}\caption{\label{fig6.2}Visualization of Lemma \ref{l_auxK<L} for $(K,\,L)=(5,\,6)$.}
\end{figure}

\begin{lem}
\label{l_auxK<L}Suppose that $K<L$. Fix a good pair $(a,\,b)$ and
recall the projected auxiliary process in Section \ref{secA.2}. Then,
there exists a surjective mapping $\Phi^{a,b}:\mathscr{V}^{a}\rightarrow V_{K}$
which satisfies:
\begin{enumerate}
\item for each $\{\sigma,\,\sigma'\}\in E(\mathscr{V}^{a})$ with $\{\sigma,\,\sigma'\}\cap\mathcal{Z}^{a,b}=\emptyset$,
we have $\Phi^{a,b}(\sigma)=\Phi^{a,b}(\sigma')$,
\item for each $\{\sigma,\,\sigma'\}\in E(\mathscr{V}^{a})$ with $\{\sigma,\,\sigma'\}\cap\mathcal{Z}^{a,b}\ne\emptyset$,
we have $\{\Phi^{a,b}(\sigma),\,\Phi^{a,b}(\sigma')\}\in E(V_{K})$
and $r^{\mathbf{a}}(\sigma,\,\sigma')=r_{K}(\Phi^{a,b}(\sigma),\,\Phi^{a,b}(\sigma'))$,
\item for each $\{x,\,y\}\in E(V_{K})$, there exist exactly \textbf{four}
edges $\{\sigma,\,\sigma'\}\in E(\mathscr{V}^{a})$ such that $\{\Phi^{a,b}(\sigma),\,\Phi^{a,b}(\sigma')\}=\{x,\,y\}$.
\end{enumerate}
\end{lem}

\begin{proof}
First, we assume that $(a,\,b)=(-1,\,0)$. We refer to Figure \ref{fig6.2}
to provide insight of the proof given here. We have $\mathcal{R}_{2}^{-1,0}=\{\zeta_{2}^{+},\,\zeta_{2}^{-}\}$
(cf. Definition \ref{d_can}). First, we focus on the landscape between
$\mathcal{N}(\mathbf{-1})$ and $\zeta_{2}^{+}$.

There are two possible $\sigma\in\mathcal{Z}^{-1,0}$ with $\sigma\sim\zeta_{2}^{+}$;
that is, $\zeta_{1,K-1}^{++}$ and $\zeta_{1,K-1}^{+-}$. we first
consider $\zeta_{1,K-1}^{++}$. All the possible paths from $\zeta_{1,K-1}^{++}$
to $\mathcal{N}(\mathbf{-1})$ are illustrated in Figure \ref{fig6.2}
(right) for the case of $K=5$ and $L=6$. Rigorously, we temporarily
denote by $\xi_{h}\in\mathcal{X}$, $h\in\llbracket1,\,K-1\rrbracket$
the configuration which has spins $0$ on
\[
\big[\llbracket1,\,K-1\rrbracket\times\{1\}\big]\cup\big[\llbracket1,\,h\rrbracket\times\{2\}\big]
\]
and spins $-1$ on the remainder. Then, we define $\Phi_{1}^{-1,0}:\{\mathbf{-1}\}\cup\bigcup_{h=1}^{K-1}\{\zeta_{1,h}^{++},\,\xi_{h}\}\cup\{\zeta_{2}^{+}\}\rightarrow V_{K}$
by $\Phi_{1}^{-1,0}(\zeta_{2}^{+})=0$, $\Phi_{1}^{-1,0}(\mathbf{-1})=\mathfrak{d}$,
and for $h\in\llbracket1,\,K-1\rrbracket$,
\[
\Phi_{1}^{-1,0}(\zeta_{1,h}^{++})=(0,\,K-h),\;\;\;\;\Phi_{1}^{-1,0}(\xi_{h})=(1,\,K-h).
\]
Then from Figures \ref{fig6.2} and \ref{figA.1}, it is straightforward
that $\Phi_{1}^{-1,0}$ is bijective and that it preserves the edge
structure.

If we consider $\zeta_{1,K-1}^{+-}$, we deduce as in the previous
case another separated landscape of configurations between $\mathbf{-1}$
and $\zeta_{1,K-1}^{+-}$. Then, we can define a similar bijective
function $\Phi_{2}^{-1,0}$ defined on the relevant configurations
to $V_{K}$ that preserves the edge structure.

Similarly, by examining the landscape between $\mathbf{-1}$ and $\zeta_{2}^{-}$,
we obtain two more bijective functions $\Phi_{3}^{-1,0}$ and $\Phi_{4}^{-1,0}$
that preserve the edge structure. Moreover, it is clear that the union
of $\mathrm{dom}\Phi_{i}^{-1,0}$, the domain of $\Phi_{i}^{-1,0}$,
for $i\in\llbracket1,\,4\rrbracket$ is indeed $\{\mathbf{-1}\}\cup\mathcal{Z}^{-1,0}\cup\mathcal{R}_{2}^{-1,0}$.

Now, we define $\Phi^{-1,0}:\mathscr{V}^{-1}\rightarrow V_{K}$ by
\[
\Phi^{-1,0}(\sigma)=\begin{cases}
\Phi_{i}^{-1,0}(\sigma) & \text{if }\sigma\in\mathrm{dom}\Phi_{i}^{-1,0},\\
\mathfrak{d} & \text{if }\sigma\notin\{\mathbf{-1}\}\cup\mathcal{Z}^{-1,0}\cup\mathcal{R}_{2}^{-1,0}.
\end{cases}
\]
In this way, the function $\Phi^{-1,0}$ is well defined because the
only possible intersection among $\mathrm{dom}\Phi_{i}^{-1,0}$, $i\in\llbracket1,\,4\rrbracket$
is $\{\mathbf{-1}\}$, on which $\Phi_{i}^{-1,0}$ is uniformly defined
as $\mathfrak{d}$.

Finally, we prove the assertions. $\Phi^{-1,0}$ is clearly surjective
as each $\Phi_{i}^{-1,0}$ is bijective. For part (1), if $\{\sigma,\,\sigma'\}\in E(\mathscr{V}^{-1})$
with $\{\sigma,\,\sigma'\}\cap\mathcal{Z}^{-1,0}=\emptyset$ then
we have $\sigma,\,\sigma'\in\mathcal{E}^{-1}\setminus(\mathcal{Z}^{-1,0}\cup\mathcal{R}_{2}^{-1,0})$,
so that $\Phi^{-1,0}(\sigma)=\Phi^{-1,0}(\sigma')=\mathfrak{d}$.
Part (2) is obvious from the bijective functions $\Phi_{i}^{-1,0}$.
As we have four such bijections, part (3) is now verified.

The other good pairs $(a,\,b)$ can be dealt with in a similar way;
thus, we do not repeat the tedious proof.
\end{proof}
Next, we deal with the case of $K=L$.
\begin{lem}
\label{l_auxK=00003DL}Suppose that $K=L$ and fix a good pair $(a,\,b)$.
Then, there exists a surjective mapping $\Phi^{a,b}:\mathscr{V}^{a}\rightarrow V_{K}$
which satisfies:
\begin{enumerate}
\item for each $\{\sigma,\,\sigma'\}\in E(\mathscr{V}^{a})$ with $\{\sigma,\,\sigma'\}\cap\mathcal{Z}^{a,b}=\emptyset$,
we have $\Phi^{a,b}(\sigma)=\Phi^{a,b}(\sigma')$,
\item for each $\{\sigma,\,\sigma'\}\in E(\mathscr{V}^{a})$ with $\{\sigma,\,\sigma'\}\cap\mathcal{Z}^{a,b}\ne\emptyset$,
we have $\{\Phi^{a,b}(\sigma),\,\Phi^{a,b}(\sigma')\}\in E(V_{K})$
and $r^{\mathbf{a}}(\sigma,\,\sigma')=r_{K}(\Phi^{a,b}(\sigma),\,\Phi^{a,b}(\sigma'))$,
\item for each $\{x,\,y\}\in E(V_{K})$, there exist exactly \textbf{eight}
edges $\{\sigma,\,\sigma'\}\in E(\mathscr{V}^{a})$ such that $\{\Phi^{a,b}(\sigma),\,\Phi^{a,b}(\sigma')\}=\{x,\,y\}$.
\end{enumerate}
\end{lem}

\begin{proof}
First, we assume that $(a,\,b)=(-1,\,0)$. The only difference to
Lemma \ref{l_auxK<L} is that we now have $\mathcal{R}_{2}^{-1,0}=\{\zeta_{2}^{+},\,\zeta_{2}^{-},\,\Theta(\zeta_{2}^{+}),\,\Theta(\zeta_{2}^{-})\}$,
where $\Theta$ is the operator defined in Definition \ref{d_can}.
Thus, the corresponding number of edges are exactly doubled compared
to Lemma \ref{l_auxK<L}. The rest of the proof is identical.
\end{proof}

\section{\label{sec7}Construction of Fundamental Test Functions and Flows}

\subsection{Fundamental test objects}

In this subsection, we construct two fundamental test functions which
are the main ingredients of the actual test functions to approximate
the capacities via the Dirichlet principle (cf. Theorem \ref{t_DP}).
More specifically, we construct two real test functions, namely, $g^{-1,0}$
and $g^{+1,0}$ on $\mathcal{X}$. Concisely, $g^{-1,0}$ (resp. $g^{+1,0}$)
describes the dynamical transitions from $\mathbf{-1}$ (resp. $\mathbf{+1}$)
to $\mathbf{0}$ in the sense of equilibrium potentials. Then, we
define two fundamental test flows according to \eqref{e_fcnflow}.
\begin{defn}[Test function $g^{-1,0}$]
\label{d_tfcn-10} Here, we construct the function $g^{-1,0}:\mathcal{X}\rightarrow\mathbb{R}$
which describes the metastable transition from $\mathbf{-1}$ to $\mathbf{0}$.
For the construction, we recall \eqref{e_T} and define $g^{-1,0}$
on the members of $\mathcal{T}$ separately, and then define on $\mathcal{X}\setminus\mathcal{T}$.
\begin{itemize}
\item \textbf{$\mathcal{B}^{-1,0}$: }For $\sigma\in\mathcal{B}^{-1,0}$,
where the number of spins $0$ in $\sigma$ is $z\in\llbracket2K,\,K(L-2)\rrbracket$,
we define
\[
g^{-1,0}(\sigma)=\frac{1}{\kappa}\Big[\frac{K(L-2)-z}{K(L-4)}\mathfrak{b}+\mathfrak{e}\Big].
\]
\item \textbf{$\mathcal{E}^{-1}$:} We define (cf. Proposition \ref{p_proj},
Lemmas \ref{l_auxK<L} and \ref{l_auxK=00003DL})
\[
g^{-1,0}(\sigma)=1-\frac{\mathfrak{e}}{\kappa}\cdot h_{0,\mathfrak{d}}^{K}\big((\Phi^{-1,0}\circ\Pi^{\mathbf{-1}})(\sigma)\big).
\]
\item \textbf{$\mathcal{E}^{0}$:} We define
\[
g^{-1,0}(\sigma)=\frac{\mathfrak{e}}{\kappa}\cdot h_{0,\mathfrak{d}}^{K}\big((\Phi^{0,-1}\circ\Pi^{\mathbf{0}})(\sigma)\big).
\]
\item \textbf{$\mathcal{B}^{0,+1}\cup\mathcal{E}^{+1}\cup(\mathcal{X}\setminus\mathcal{T})$:}
We define $g^{-1,0}\equiv0$.
\end{itemize}
\end{defn}

\begin{defn}[Test function $g^{+1,0}$]
\label{d_tfcn+10} We define $g^{+1,0}$ in exactly the same manner.
Rigorously, we define $\Xi:\mathcal{X}\rightarrow\mathcal{X}$ by
\[
(\Xi(\sigma))(x)=\begin{cases}
+1 & \text{if }\sigma(x)=-1,\\
-1 & \text{if }\sigma(x)=+1,\\
0 & \text{if }\sigma(x)=0.
\end{cases}
\]
Then, we define $g^{+1,0}(\sigma)=g^{-1,0}(\Xi(\sigma))$.
\end{defn}

\begin{defn}[Test flows $\phi^{-1,0}$ and $\phi^{+1,0}$]
\label{d_tfl} We define $\phi^{-1,0}=\Psi_{g^{-1,0}}$ and $\phi^{+1,0}=\Psi_{g^{+1,0}}$
(cf. \eqref{e_fcnflow}) on the typical configurations and zero on the remainder.
\end{defn}

\begin{rem}
\label{r_tfcntfl}To check that the functions are well defined, it
suffices to recognize that $g^{-1,0}$ is defined as $1-\mathfrak{e}/\kappa$
on $\mathcal{R}_{2}^{-1,0}=\mathcal{B}^{-1,0}\cap\mathcal{E}^{-1}$
and $\mathfrak{e}/\kappa$ on $\mathcal{R}_{L-2}^{-1,0}=\mathcal{B}^{-1,0}\cap\mathcal{E}^{0}$.
\end{rem}

\begin{rem}
\label{r_tfcntfl2}We remark that if $\sigma,\,\sigma'\in\widehat{\mathcal{N}}(\mathcal{S})$
with $\sigma\sim\sigma'$, then either $g^{-1,0}(\sigma)=g^{-1,0}(\sigma')$
or $g^{+1,0}(\sigma)=g^{+1,0}(\sigma')$ must hold. To prove this,
recall from Remark \ref{r_typedge} that
\[
E(\widehat{\mathcal{N}}(\mathcal{S}))=E(\mathcal{B}^{-1,0})\cup E(\mathcal{B}^{0,+1})\cup E(\mathcal{E}^{-1})\cup E(\mathcal{E}^{0})\cup E(\mathcal{E}^{+1}).
\]
By Definitions \ref{d_tfcn-10} and \ref{d_tfcn+10}, we only need
to consider the case of $\{\sigma,\,\sigma'\}\in E(\mathcal{E}^{0})$.
Then, by the proof of Lemma \ref{l_auxK<L}, $g^{-1,0}(\sigma)=g^{-1,0}(\sigma')$
unless $\{\sigma,\,\sigma'\}\in E(\mathcal{N}(\mathbf{0})\cup\mathcal{Z}^{0,-1}\cup\mathcal{R}_{2}^{0,-1})$
and $g^{+1,0}(\sigma)=g^{+1,0}(\sigma')$ unless $\{\sigma,\,\sigma'\}\in E(\mathcal{N}(\mathbf{0})\cup\mathcal{Z}^{0,+1}\cup\mathcal{R}_{2}^{0,+1})$.
As
\[
E(\mathcal{N}(\mathbf{0})\cup\mathcal{Z}^{0,-1}\cup\mathcal{R}_{2}^{0,-1})\cap E(\mathcal{N}(\mathbf{0})\cup\mathcal{Z}^{0,+1}\cup\mathcal{R}_{2}^{0,+1})=E(\mathcal{N}(\mathbf{0}))
\]
and both functions are constantly zero on $\mathcal{N}(\mathbf{0})$,
we obtain the desired result. In turn, if $\sigma,\,\sigma'\in\widehat{\mathcal{N}}(\mathcal{S})$
with $\sigma\sim\sigma'$, then we have either $\phi^{-1,0}(\sigma,\,\sigma')=0$
or $\phi^{+1,0}(\sigma,\,\sigma')=0$.
\end{rem}

\subsection{Properties of fundamental test functions}

Now, we calculate the Dirichlet form of the test functions.
\begin{prop}
\label{p_tfcn}We have
\[
D_{\beta}(g^{-1,0})=\frac{1+o(1)}{3\kappa}e^{-\beta\Gamma}\;\;\;\;\text{and}\;\;\;\;D_{\beta}(g^{+1,0})=\frac{1+o(1)}{3\kappa}e^{-\beta\Gamma}.
\]
\end{prop}

\begin{proof}
By symmetry, it suffices to estimate $D_{\beta}(g^{-1,0})$. By definition,
we write
\begin{equation}
D_{\beta}(g^{-1,0})=\Big[\sum_{\{\sigma,\zeta\}\subseteq\mathcal{T}}+\sum_{\sigma\in\mathcal{T}}\sum_{\zeta\in\mathcal{X}\setminus\mathcal{T}}+\sum_{\{\sigma,\zeta\}\subseteq\mathcal{X}\setminus\mathcal{T}}\Big]\mu_{\beta}(\sigma)c_{\beta}(\sigma,\,\zeta)[g^{-1,0}(\zeta)-g^{-1,0}(\sigma)]^{2}.\label{e_tfcn1}
\end{equation}
The third summation of \eqref{e_tfcn1} vanishes because $g^{-1,0}\equiv0$
on $\mathcal{X}\setminus\mathcal{T}$. For the second (double) summation
of \eqref{e_tfcn1}, if $\sigma\in\mathcal{T}$ and $\zeta\in\mathcal{X}\setminus\mathcal{T}$
with $\sigma\sim\zeta$, then as $\mathcal{T}=\widehat{\mathcal{N}}(\mathcal{S})$
by part (4) of Proposition \ref{p_typ}, we have $H(\zeta)\ge\Gamma+1$.
Hence, by \eqref{e_detbal},
\[
\mu_{\beta}(\sigma)c_{\beta}(\sigma,\,\zeta)=\min\{\mu_{\beta}(\sigma),\,\mu_{\beta}(\zeta)\}=\mu_{\beta}(\zeta)=O(e^{-\beta(\Gamma+1)}).
\]
Therefore, the second (double) summation is of scale $O(e^{-\beta(\Gamma+1)})$.

It remains to calculate the first summation of \eqref{e_tfcn1}. By
Remark \ref{r_typedge} and the fact that $g^{-1,0}$ is constant
on $\mathcal{B}^{0,+1}$ and on $\mathcal{E}^{+1}$, we can rewrite
the summation as
\begin{equation}
\Big[\sum_{\{\sigma,\zeta\}\subseteq\mathcal{B}^{-1,0}}+\sum_{\{\sigma,\zeta\}\subseteq\mathcal{E}^{-1}}+\sum_{\{\sigma,\zeta\}\subseteq\mathcal{E}^{0}}\Big]\mu_{\beta}(\sigma)c_{\beta}(\sigma,\,\zeta)[g^{-1,0}(\zeta)-g^{-1,0}(\sigma)]^{2}.\label{e_tfcn2}
\end{equation}
We first deal with the first summation of \eqref{e_tfcn2}. Recall
from Definition \ref{d_typ} that $\mathcal{B}^{-1,0}=\bigcup_{v=2}^{L-2}\mathcal{R}_{v}^{-1,0}\cup\bigcup_{v=2}^{L-3}\mathcal{Q}_{v}^{-1,0}$.
If $K<L$, then the first summation of \eqref{e_tfcn2} becomes
\begin{align*}
 & \sum_{v=2}^{L-3}\sum_{h=0}^{K-1}\mu_{\beta}(\zeta_{v,h}^{+\pm})c_{\beta}(\zeta_{v,h}^{+\pm},\,\zeta_{v,h+1}^{+\pm})[g^{-1,0}(\zeta_{v,h+1}^{+\pm})-g^{-1,0}(\zeta_{v,h}^{+\pm})]^{2}\\
 & +\sum_{v=2}^{L-3}\sum_{h=0}^{K-1}\mu_{\beta}(\zeta_{v,h}^{-\pm})c_{\beta}(\zeta_{v,h}^{-\pm},\,\zeta_{v,h+1}^{-\pm})[g^{-1,0}(\zeta_{v,h+1}^{-\pm})-g^{-1,0}(\zeta_{v,h}^{-\pm})]^{2},
\end{align*}
where the signs $\pm$ indicate shorthands for $+$ and $-$ (so that
the above formula actually consists of four double summations). By
\eqref{e_Zmu}, \eqref{e_detbal}, and Definition \ref{d_tfcn-10},
this asymptotically equals (cf. \eqref{e_be})
\[
4\sum_{v=2}^{L-3}\sum_{h=0}^{K-1}\frac{1}{3}e^{-\beta\Gamma}\cdot\frac{1}{\kappa^{2}}\frac{\mathfrak{b}^{2}}{K^{2}(L-4)^{2}}=\frac{4\mathfrak{b}^{2}}{3\kappa^{2}K(L-4)}e^{-\beta\Gamma}=\frac{\mathfrak{b}}{3\kappa^{2}}e^{-\beta\Gamma}.
\]
If $K=L$, then the first summation of \eqref{e_tfcn2} must be counted
twice the preceding computation due to the presence of transposed
configurations obtained by the operator $\Theta$ (cf. Definition
\ref{d_can}). Thus, the summation asymptotically equals (cf. \eqref{e_be})
\[
8\sum_{v=2}^{L-3}\sum_{h=0}^{K-1}\frac{1}{3}e^{-\beta\Gamma}\cdot\frac{1}{\kappa^{2}}\frac{\mathfrak{b}^{2}}{K^{2}(L-4)^{2}}=\frac{8\mathfrak{b}^{2}}{3\kappa^{2}K(L-4)}e^{-\beta\Gamma}=\frac{\mathfrak{b}}{3\kappa^{2}}e^{-\beta\Gamma}.
\]
Summing up, we have
\begin{equation}
\sum_{\{\sigma,\zeta\}\subseteq\mathcal{B}^{-1,0}}\mu_{\beta}(\sigma)c_{\beta}(\sigma,\,\zeta)[g^{-1,0}(\zeta)-g^{-1,0}(\sigma)]^{2}\simeq\frac{\mathfrak{b}}{3\kappa^{2}}e^{-\beta\Gamma}.\label{e_tfcn3}
\end{equation}

Next, we calculate the second summation of \eqref{e_tfcn2}. Recalling
the decomposition \eqref{e_Eadec}, we rewrite this as
\begin{align*}
 & \sum_{\{\sigma,\zeta\}\subseteq\mathcal{O}^{-1}}\mu_{\beta}(\sigma)c_{\beta}(\sigma,\,\zeta)[g^{-1,0}(\zeta)-g^{-1,0}(\sigma)]^{2}\\
 & +\sum_{\sigma\in\mathcal{O}^{-1}}\sum_{\zeta\subseteq\mathcal{I}_{\mathfrak{0}}^{-1}}\sum_{\zeta'\in\mathcal{N}(\zeta)}\mu_{\beta}(\sigma)c_{\beta}(\sigma,\,\zeta')[g^{-1,0}(\zeta')-g^{-1,0}(\sigma)]^{2}.
\end{align*}
By Proposition \ref{p_proj} and Definition \ref{d_tfcn-10}, this
is asymptotically equal to
\[
\Big[\sum_{\{\sigma,\zeta\}\subseteq\mathcal{O}^{-1}}+\sum_{\sigma\in\mathcal{O}^{-1}}\sum_{\zeta\in\mathcal{I}_{\mathfrak{0}}^{-1}}\Big]\frac{1}{3}e^{-\beta\Gamma}r^{\mathbf{-1}}(\sigma,\,\zeta)[g^{-1,0}(\zeta)-g^{-1,0}(\sigma)]^{2}.
\]
By Definition \ref{d_tfcn-10}, this becomes
\begin{equation}
\frac{1}{3}e^{-\beta\Gamma}\sum_{\{\sigma,\zeta\}\in E(\mathscr{V}^{-1})}r^{\mathbf{-1}}(\sigma,\,\zeta)\cdot\frac{\mathfrak{e}^{2}}{\kappa^{2}}\big[h_{0,\mathfrak{d}}^{K}(\Phi^{-1,0}(\zeta))-h_{0,\mathfrak{d}}^{K}(\Phi^{-1,0}(\sigma))\big]^{2}.\label{e_tfcn4}
\end{equation}
If $K<L$, then by Lemma \ref{l_auxK<L}, this becomes
\begin{align*}
 & \frac{4}{3}e^{-\beta\Gamma}\sum_{\{x,y\}\in E(V_{K})}r_{K}(x,\,y)\cdot\frac{\mathfrak{e}^{2}}{\kappa^{2}}\big[h_{0,\mathfrak{d}}^{K}(y)-h_{0,\mathfrak{d}}^{K}(x)\big]^{2}\\
 & =\frac{4\mathfrak{e}^{2}}{3\kappa^{2}}e^{-\beta\Gamma}\cdot|V_{K}|\mathrm{cap}_{K}(0,\,\mathfrak{d})=\frac{4\mathfrak{e}^{2}}{3\kappa^{2}}e^{-\beta\Gamma}\cdot\mathfrak{c}_{K}=\frac{\mathfrak{e}}{3\kappa^{2}}e^{-\beta\Gamma}.
\end{align*}
The last two equalities hold by \eqref{e_cK2} and \eqref{e_be},
respectively. If $K=L$, then by Lemma \ref{l_auxK=00003DL}, term
\eqref{e_tfcn4} equals
\[
\frac{8}{3}e^{-\beta\Gamma}\sum_{\{x,y\}\in E(V_{K})}r_{K}(x,\,y)\cdot\frac{\mathfrak{e}^{2}}{\kappa^{2}}\big[h_{0,\mathfrak{d}}^{K}(y)-h_{0,\mathfrak{d}}^{K}(x)\big]^{2}=\frac{\mathfrak{e}}{3\kappa^{2}}e^{-\beta\Gamma},
\]
which is again by \eqref{e_cK2} and \eqref{e_be}. Therefore, in
any cases, we have that
\begin{equation}
\sum_{\{\sigma,\zeta\}\subseteq\mathcal{E}^{-1}}\mu_{\beta}(\sigma)c_{\beta}(\sigma,\,\zeta)[g^{-1,0}(\zeta)-g^{-1,0}(\sigma)]^{2}\simeq\frac{\mathfrak{e}}{3\kappa^{2}}e^{-\beta\Gamma}.\label{e_tfcn5}
\end{equation}
Similarly, we have that the third summation of \eqref{e_tfcn2} is
asymptotically equal to the last displayed term. Gathering this fact,
\eqref{e_tfcn2}, \eqref{e_tfcn3}, and \eqref{e_tfcn5}, we have
that the first summation of \eqref{e_tfcn1} is asymptotically equal
to
\[
\frac{\mathfrak{b}}{3\kappa^{2}}e^{-\beta\Gamma}+\frac{2\mathfrak{e}}{3\kappa^{2}}e^{-\beta\Gamma}=\frac{1}{3\kappa}e^{-\beta\Gamma}.
\]
Therefore, we deduce that \eqref{e_tfcn1} asymptotically equals $e^{-\beta\Gamma}/(3\kappa)$,
which concludes the estimate of $D_{\beta}(g^{-1,0})$.
\end{proof}

\subsection{Properties of fundamental test flows}

We first estimate the flow norm.
\begin{prop}
\label{p_tflnorm}We have
\[
\|\phi^{-1,0}\|_{\beta}^{2}=\frac{1+o(1)}{3\kappa}e^{-\beta\Gamma}\;\;\;\;\text{and}\;\;\;\;\|\phi^{+1,0}\|_{\beta}^{2}=\frac{1+o(1)}{3\kappa}e^{-\beta\Gamma}.
\]
\end{prop}

\begin{proof}
The formulas are straightforward from \eqref{e_fcnflow2} and the last display of the proof of Proposition
\ref{p_tfcn}.
\end{proof}
Now, we deal with the divergence of the fundamental test flows. This
procedure is crucial to estimate the right-hand side of \eqref{e_gTP}
when we apply Theorem \ref{t_gTP}, the generalized Thomson principle.
As $\phi^{-1,0}$ and $\phi^{+1,0}$ have the same structure, we focus
on estimating the former test flow $\phi^{-1,0}$.
\begin{lem}
\label{l_tfldiv1}For $\sigma\in\mathcal{B}^{-1,0}\setminus(\mathcal{E}^{-1}\cup\mathcal{E}^{0})$,
it holds that $(\mathrm{div}\,\phi^{-1,0})(\sigma)=0$.
\end{lem}

\begin{proof}
By \eqref{e_Bab} and Proposition \ref{p_typ}, we have
\[
\mathcal{B}^{-1,0}\setminus(\mathcal{E}^{-1}\cup\mathcal{E}^{0})=\bigcup_{v=3}^{L-3}\mathcal{R}_{v}^{-1,0}\cup\bigcup_{v=2}^{L-3}\mathcal{Q}_{v}^{-1,0}.
\]
If $\sigma\in\mathcal{R}_{v}^{-1,0}$, $v\in\llbracket3,\,L-3\rrbracket$,
then $\sigma\in\{\zeta_{v}^{+},\,\zeta_{v}^{-}\}$ (or additionally
in $\{\Theta(\zeta_{v}^{+}),\,\Theta(\zeta_{v}^{-})\}$ if $K=L$).
Taking $\sigma=\zeta_{v}^{+}$ for instance, $(\mathrm{div}\,\phi^{-1,0})(\sigma)$
equals
\begin{align*}
 & \phi^{-1,0}(\zeta_{v}^{+},\,\zeta_{v,1}^{++})+\phi^{-1,0}(\zeta_{v}^{+},\,\zeta_{v,1}^{+-})+\phi^{-1,0}(\zeta_{v}^{+},\,\zeta_{v-1,K-1}^{++})+\phi^{-1,0}(\zeta_{v}^{+},\,\zeta_{v-1,K-1}^{+-})\\
 & =\frac{1}{Z_{\beta}}e^{-\beta\Gamma}\cdot\frac{\mathfrak{b}}{\kappa}\Big[\frac{1}{K(L-4)}+\frac{1}{K(L-4)}-\frac{1}{K(L-4)}-\frac{1}{K(L-4)}\Big]=0.
\end{align*}
Same computation works for the other cases as well. If $\sigma\in\mathcal{Q}_{v}^{-1,0}$,
$v\in\llbracket2,\,L-3\rrbracket$, then $\sigma\in\bigcup_{h=1}^{K-1}\{\zeta_{v,h}^{++},\,\zeta_{v,h}^{+-},\,\zeta_{v,h}^{-+},\,\zeta_{v,h}^{--}\}$
(or additionally in $\bigcup_{h=1}^{K-1}\{\Theta(\zeta_{v,h}^{++}),\,\Theta(\zeta_{v,h}^{+-}),\,\Theta(\zeta_{v,h}^{-+}),\,\Theta(\zeta_{v,h}^{--})\}$
if $K=L$). Taking $\sigma=\zeta_{v,h}^{++}$ for instance, $(\mathrm{div}\,\phi^{-1,0})(\sigma)$
equals
\[
\phi^{-1,0}(\zeta_{v,h}^{++},\,\zeta_{v,h+1}^{++})+\phi^{-1,0}(\zeta_{v,h}^{++},\,\zeta_{v,h-1}^{++})=\frac{1}{Z_{\beta}}e^{-\beta\Gamma}\cdot\frac{\mathfrak{b}}{\kappa}\Big[\frac{1}{K(L-4)}-\frac{1}{K(L-4)}\Big]=0.
\]
Again, same computation works for the remaining cases. Thus, we conclude
that $\phi^{-1,0}$ is divergence-free on $\mathcal{B}^{-1,0}\setminus(\mathcal{E}^{-1}\cup\mathcal{E}^{0})$.
\end{proof}
\begin{lem}
\label{l_tfldiv2}For $\sigma\in\mathcal{R}_{2}^{-1,0}\cup\mathcal{R}_{2}^{0,-1}$,
it holds that $(\mathrm{div}\,\phi^{-1,0})(\sigma)=0$.
\end{lem}

\begin{proof}
We only consider the set $\mathcal{R}_{2}^{-1,0}$, as the latter
set can be handled similarly. We claim that
\begin{equation}
\sum_{\sigma\in\mathcal{R}_{2}^{-1,0}}(\mathrm{div}\,\phi^{-1,0})(\sigma)=0,\label{e_tfldiv2.1}
\end{equation}
which in turn implies $(\mathrm{div}\,\phi^{-1,0})(\sigma)=0$ for
all $\sigma\in\mathcal{R}_{2}^{-1,0}$ because of the model symmetry.
Elements of $\mathcal{R}_{2}^{-1,0}$ are connected to elements of
both $\mathcal{B}^{-1,0}$ and $\mathcal{E}^{-1}$, so that
\begin{equation}
\sum_{\sigma\in\mathcal{R}_{2}^{-1,0}}(\mathrm{div}\,\phi^{-1,0})(\sigma)=\sum_{\sigma\in\mathcal{R}_{2}^{-1,0}}\Big[\sum_{\zeta\in\mathcal{E}^{-1}}+\sum_{\zeta\in\mathcal{B}^{-1,0}}\Big]\phi^{-1,0}(\sigma,\,\zeta).\label{e_tfldiv2.2}
\end{equation}
First, we consider the former double summation. By definition, this
is
\begin{align*}
 & \frac{\mathfrak{e}}{\kappa}\sum_{\sigma\in\mathcal{R}_{2}^{-1,0}}\sum_{\zeta\in\mathcal{E}^{-1}}\mu_{\beta}(\sigma)c_{\beta}(\sigma,\,\zeta)\big[h_{0,\mathfrak{d}}^{K}\big((\Phi^{-1,0}\circ\Pi^{\mathbf{-1}})(\zeta)\big)-1\big]\\
 & =-\frac{\mathfrak{e}}{Z_{\beta}\kappa}e^{-\beta\Gamma}\sum_{\sigma\in\mathcal{R}_{2}^{-1,0}}\sum_{\zeta\in\mathcal{O}^{-1}}\big[1-h_{0,\mathfrak{d}}^{K}(\Phi^{-1,0}(\zeta))\big].
\end{align*}
By Lemmas \ref{l_auxK<L}, \ref{l_auxK=00003DL}, and an elementary
property of capacities (cf. \cite[(7.1.39)]{B-denH}), this equals
\[
\begin{cases}
-\frac{\mathfrak{e}}{Z_{\beta}\kappa}e^{-\beta\Gamma}\cdot4\mathfrak{c}_{K} & \text{if }K<L,\\
-\frac{\mathfrak{e}}{Z_{\beta}\kappa}e^{-\beta\Gamma}\cdot8\mathfrak{c}_{K} & \text{if }K=L.
\end{cases}
\]
Therefore, by \eqref{e_be}, we have
\begin{equation}
\sum_{\sigma\in\mathcal{R}_{2}^{-1,0}}\sum_{\zeta\in\mathcal{E}^{-1}}\phi^{-1,0}(\sigma,\,\zeta)=-\frac{1}{Z_{\beta}\kappa}e^{-\beta\Gamma}.\label{e_tfldiv2.3}
\end{equation}

Next, we consider the latter double summation of \eqref{e_tfldiv2.2}.
We divide into two cases.
\begin{itemize}
\item Suppose that $K<L$, so that $\mathcal{R}_{2}^{-1,0}=\{\zeta_{2}^{+},\,\zeta_{2}^{-}\}$.
Then, we have by Definition \ref{d_tfl} that the summation equals
\begin{align*}
 & \phi^{-1,0}(\zeta_{2}^{+},\,\zeta_{2,1}^{++})+\phi^{-1,0}(\zeta_{2}^{+},\,\zeta_{2,1}^{+-})+\phi^{-1,0}(\zeta_{2}^{-},\,\zeta_{2,1}^{-+})+\phi^{-1,0}(\zeta_{2}^{-},\,\zeta_{2,1}^{--})\\
 & =\frac{1}{Z_{\beta}}e^{-\beta\Gamma}\cdot\frac{\mathfrak{b}}{\kappa}\Big[\frac{1}{K(L-4)}+\frac{1}{K(L-4)}+\frac{1}{K(L-4)}+\frac{1}{K(L-4)}\Big]=\frac{1}{Z_{\beta}\kappa}e^{-\beta\Gamma}.
\end{align*}
The last equality holds by \eqref{e_be}.
\item Suppose that $K=L$, so that $\mathcal{R}_{2}^{-1,0}=\{\zeta_{2}^{+},\,\zeta_{2}^{-},\,\Theta(\zeta_{2}^{+}),\,\Theta(\zeta_{2}^{-})\}$.
Then, the above summation must be exactly doubled, so that
\[
\sum_{\sigma\in\mathcal{R}_{2}^{-1,0}}\sum_{\zeta\in\mathcal{B}^{-1,0}}\phi^{-1,0}(\sigma,\,\zeta)=\frac{1}{Z_{\beta}\kappa}e^{-\beta\Gamma}\cdot\frac{8\mathfrak{b}}{K(L-4)}=\frac{1}{Z_{\beta}\kappa}e^{-\beta\Gamma},
\]
where the last equality still holds by \eqref{e_be}.
\end{itemize}
Therefore, in any cases we have
\begin{equation}
\sum_{\sigma\in\mathcal{R}_{2}^{-1,0}}\sum_{\zeta\in\mathcal{B}^{-1,0}}\phi^{-1,0}(\sigma,\,\zeta)=\frac{1}{Z_{\beta}\kappa}e^{-\beta\Gamma}.\label{e_tfldiv2.4}
\end{equation}
Combining \eqref{e_tfldiv2.2}, \eqref{e_tfldiv2.3}, and \eqref{e_tfldiv2.4}
yields \eqref{e_tfldiv2.1}, which concludes the proof.
\end{proof}
\begin{lem}
\label{l_tfldiv3}For $\sigma\in\mathcal{O}^{-1}\cup\mathcal{O}^{0}$,
it holds that $(\mathrm{div}\,\phi^{-1,0})(\sigma)=0$.
\end{lem}

\begin{proof}
By symmetry, we only prove $(\mathrm{div}\,\phi^{-1,0})(\sigma)=0$
for each $\sigma\in\mathcal{O}^{-1}$. If $\sigma\in\mathcal{O}^{-1}\setminus\mathcal{Z}^{-1,0}$,
then there is nothing to prove because by Lemmas \ref{l_auxK<L} and
\ref{l_auxK=00003DL}, we have $g^{-1,0}(\sigma)=g^{-1,0}(\sigma')=1$
for all $\sigma'\in\mathcal{E}^{-1}$ with $\sigma\sim\sigma'$. Now,
assume that $\sigma\in\mathcal{Z}^{-1,0}$. To this end, we may rewrite
as
\begin{equation}
(\mathrm{div}\,\phi^{-1,0})(\sigma)=\sum_{\zeta\in\mathscr{V}^{-1}}\phi^{-1,0}(\sigma,\,\zeta)=\sum_{\zeta\in\mathcal{O}^{-1}}\phi^{-1,0}(\sigma,\,\zeta)+\sum_{\zeta\in\mathcal{I}_{\mathfrak{0}}^{-1}}\sum_{\zeta'\in\mathcal{N}(\zeta)}\phi^{-1,0}(\sigma,\,\zeta').\label{e_tfldiv3.1}
\end{equation}
The summation of $\zeta\in\mathcal{O}^{-1}$ in \eqref{e_tfldiv3.1}
becomes
\begin{equation}
\sum_{\zeta\in\mathcal{O}^{-1}}\frac{\mathfrak{e}}{Z_{\beta}\kappa}e^{-\beta\Gamma}\big[h_{0,\mathfrak{d}}^{K}\big((\Phi^{-1,0}\circ\Pi^{\mathbf{-1}})(\zeta)\big)-h_{0,\mathfrak{d}}^{K}\big((\Phi^{-1,0}\circ\Pi^{\mathbf{-1}})(\sigma)\big)\big].\label{e_tfldiv3.2}
\end{equation}
The double summation in \eqref{e_tfldiv3.1} becomes
\begin{equation}
\sum_{\zeta\in\mathcal{I}_{\mathfrak{0}}^{-1}}\sum_{\zeta'\in\mathcal{N}(\zeta):\,\sigma\sim\zeta'}\frac{\mathfrak{e}}{Z_{\beta}\kappa}e^{-\beta\Gamma}\big[h_{0,\mathfrak{d}}^{K}\big((\Phi^{-1,0}\circ\Pi^{\mathbf{-1}})(\zeta)\big)-h_{0,\mathfrak{d}}^{K}\big((\Phi^{-1,0}\circ\Pi^{\mathbf{-1}})(\sigma)\big)\big].\label{e_tfldiv3.3}
\end{equation}
By \eqref{e_tfldiv3.2} and \eqref{e_tfldiv3.3}, we have that \eqref{e_tfldiv3.1}
equals
\[
\sum_{\zeta\in\mathscr{V}^{-1}}\frac{\mathfrak{e}}{Z_{\beta}\kappa}e^{-\beta\Gamma}r^{\mathbf{-1}}(\sigma,\,\zeta)\big[h_{0,\mathfrak{d}}^{K}((\Phi^{-1,0}(\zeta))-h_{0,\mathfrak{d}}^{K}(\Phi^{-1,0}(\sigma))\big].
\]
By Lemmas \ref{l_auxK<L} and \ref{l_auxK=00003DL}, the last displayed
term equals four (if $K<L$) or eight (if $K=L$) times
\[
\frac{\mathfrak{e}}{Z_{\beta}\kappa}e^{-\beta\Gamma}\sum_{y\in V_{K}}r_{K}(\Phi^{-1,0}(\sigma),\,y)\big[h_{0,\mathfrak{d}}^{K}(y)-h_{0,\mathfrak{d}}^{K}(\Phi^{-1,0}(\sigma))\big]=0,
\]
where the equality holds by an elementary property of stochastic generators
(e.g., \cite[(7.1.15)]{B-denH}). This concludes the proof.
\end{proof}
Gathering the preceding lemmas, we have the following proposition.
\begin{prop}
\label{p_tflowdiv1}For $\sigma\in\mathcal{X}\setminus(\mathcal{N}(\mathbf{-1})\cup\mathcal{N}(\mathbf{0}))$,
we have $(\mathrm{div}\,\phi^{-1,0})(\sigma)=0$. Similarly, for $\sigma\in\mathcal{X}\setminus(\mathcal{N}(\mathbf{0})\cup\mathcal{N}(\mathbf{+1}))$,
we have $(\mathrm{div}\,\phi^{0,+1})(\sigma)=0$.
\end{prop}

\begin{proof}
We only prove the first statement. By Definition \ref{d_tfl}, the
test flow $\phi^{-1,0}$ is divergence-free on $\mathcal{X}\setminus(\mathcal{B}^{-1,0}\cup\mathcal{E}^{-1}\cup\mathcal{E}^{0})$.
By Lemmas \ref{l_tfldiv1}, \ref{l_tfldiv2}, and \ref{l_tfldiv3},
$\phi^{-1,0}$ is divergence-free on
\[
\big[\mathcal{B}^{-1,0}\setminus(\mathcal{E}^{-1}\cup\mathcal{E}^{0})\big]\cup\big[\mathcal{R}_{2}^{-1,0}\cup\mathcal{R}_{2}^{0,-1}\big]\cup\big[\mathcal{O}^{-1}\cup\mathcal{O}^{0}\big].
\]
By Proposition \ref{p_typ} and \eqref{e_Ia}, the above set is precisely
$(\mathcal{B}^{-1,0}\cup\mathcal{E}^{-1}\cup\mathcal{E}^{0})\setminus(\mathcal{N}(\mathbf{-1})\cup\mathcal{N}(\mathbf{0}))$.
This observation concludes the proof.
\end{proof}
Finally, we provide estimates for the divergence on the remainder.
\begin{prop}
\label{p_tflowdiv2}We have
\begin{equation}
\sum_{\sigma\in\mathcal{N}(\mathbf{-1})}(\mathrm{div}\,\phi^{-1,0})(\sigma)\simeq\frac{1}{3\kappa}e^{-\beta\Gamma}\;\;\;\;\text{and}\;\;\;\;\sum_{\sigma\in\mathcal{N}(\mathbf{0})}(\mathrm{div}\,\phi^{-1,0})(\sigma)\simeq-\frac{1}{3\kappa}e^{-\beta\Gamma}.\label{e_tflowdiv2.1}
\end{equation}
Similarly, we have
\begin{equation}
\sum_{\sigma\in\mathcal{N}(\mathbf{+1})}(\mathrm{div}\,\phi^{+1,0})(\sigma)\simeq\frac{1}{3\kappa}e^{-\beta\Gamma}\;\;\;\;\text{and}\;\;\;\;\sum_{\sigma\in\mathcal{N}(\mathbf{0})}(\mathrm{div}\,\phi^{+1,0})(\sigma)\simeq-\frac{1}{3\kappa}e^{-\beta\Gamma}.\label{e_tflowdiv2.2}
\end{equation}
\end{prop}

\begin{proof}
First, we focus on the first formula of \eqref{e_tflowdiv2.1}. By
Definition \ref{d_tfl}, this becomes
\[
\sum_{\sigma\in\mathcal{N}(\mathbf{-1})}\sum_{\zeta\in\mathcal{O}^{-1}:\,\sigma\sim\zeta}\phi^{-1,0}(\sigma,\,\zeta)=\sum_{\zeta\in\mathcal{O}^{-1}}\sum_{\sigma\in\mathcal{N}(\mathbf{-1}):\,\sigma\sim\zeta}\phi^{-1,0}(\sigma,\,\zeta).
\]
Substituting the exact value of $\phi^{-1,0}$ and from the fact that
$\phi^{-1,0}$ is anti-symmetric, we compute this as
\begin{align*}
 & \sum_{\zeta\in\mathcal{O}^{-1}}\sum_{\sigma\in\mathcal{N}(\mathbf{-1}):\,\sigma\sim\zeta}\frac{\mathfrak{e}}{Z_{\beta}\kappa}e^{-\beta\Gamma}\big[h_{0,\mathfrak{d}}^{K}(\Phi^{-1,0}(\zeta))-h_{0,\mathfrak{d}}^{K}(\Phi^{-1,0}(\mathbf{-1}))\big]\\
 & =\sum_{\zeta\in\mathcal{O}^{-1}}\frac{\mathfrak{e}}{Z_{\beta}\kappa}e^{-\beta\Gamma}r^{\mathbf{-1}}(\zeta,\,\mathbf{-1})\big[h_{0,\mathfrak{d}}^{K}(\Phi^{-1,0}(\zeta))-h_{0,\mathfrak{d}}^{K}(\Phi^{-1,0}(\mathbf{-1}))\big].
\end{align*}
By Lemmas \ref{l_auxK<L} and \ref{l_auxK=00003DL}, this becomes
\[
\begin{cases}
\frac{\mathfrak{e}}{Z_{\beta}\kappa}e^{-\beta\Gamma}\cdot4\mathfrak{c}_{K} & \text{if }K<L,\\
\frac{\mathfrak{e}}{Z_{\beta}\kappa}e^{-\beta\Gamma}\cdot8\mathfrak{c}_{K} & \text{if }K=L,
\end{cases}
\]
which is exactly $e^{-\beta\Gamma}/(Z_{\beta}\kappa)$ by \eqref{e_be}.
This proves the first formula of \eqref{e_tflowdiv2.1} by \eqref{e_Zmu}.
The second formula of \eqref{e_tflowdiv2.1} similarly follows as
\[
\sum_{\sigma\in\mathcal{N}(\mathbf{0})}(\mathrm{div}\,\phi^{-1,0})(\sigma)=-\frac{1}{Z_{\beta}\kappa}e^{-\beta\Gamma}.
\]
Finally, the formulas in \eqref{e_tflowdiv2.2} can be proved in the
same manner.
\end{proof}

\section{\label{sec8}Capacity Estimates}

In this section, we provide precise estimates of the relevant capacities
and thereby prove Theorem \ref{t_Cap}.

\subsection{Proof of parts (1) and (2) of Theorem \ref{t_Cap}}

By symmetry, it suffices to estimate $\mathrm{Cap}_{\beta}(\mathbf{-1},\,\{\mathbf{0},\,\mathbf{+1}\})$
and $\mathrm{Cap}_{\beta}(\mathbf{-1},\,\mathbf{0})$. For both objects,
we use the test function $g^{-1,0}$ (cf. Definition \ref{d_tfcn-10})
and the test flow $\phi^{-1,0}$ (cf. Definition \ref{d_tfl}).
\begin{proof}[Proof of parts (1) and (2) of Theorem \ref{t_Cap}]
 First, note that $g^{-1,0}\in\mathfrak{C}(\{\mathbf{-1}\},\,\{\mathbf{0},\,\mathbf{+1}\})\subseteq\mathfrak{C}(\{\mathbf{-1}\},\,\{\mathbf{0}\})$.
Hence, by the Dirichlet principle and Proposition \ref{p_tfcn}, we
have
\begin{equation}
\mathrm{Cap}_{\beta}(\mathbf{-1},\,\{\mathbf{0},\,\mathbf{+1}\}),\,\mathrm{Cap}_{\beta}(\mathbf{-1},\,\mathbf{0})\le D_{\beta}(g^{-1,0})=\frac{1+o(1)}{3\kappa}e^{-\beta\Gamma}.\label{e_Cap1}
\end{equation}
Next, we consider the lower bounds using the generalized Thomson principle.
First, by Proposition \ref{p_tflnorm}, we have
\[
\|\phi^{-1,0}\|_{\beta}^{2}=\frac{1+o(1)}{3\kappa}e^{-\beta\Gamma}.
\]
Next, Proposition \ref{p_tflowdiv1} implies that $(\mathrm{div}\,\phi^{-1,0})(\sigma)=0$
for all $\sigma\notin\mathcal{N}(\mathbf{-1})\cup\mathcal{N}(\mathbf{0})$.
Moreover, by Lemma \ref{l_eqp}, there exists a constant $C=C(K,\,L)>0$
such that we have
\[
\max_{\zeta\in\mathcal{N}(\mathbf{-1})}\big|h(\zeta)-h(\mathbf{-1})\big|\le Ce^{-\beta},\;\;\;\;\max_{\zeta\in\mathcal{N}(\mathbf{0})}\big|h(\zeta)-h(\mathbf{0})\big|\le Ce^{-\beta}
\]
for both $h=h_{\mathbf{-1},\{\mathbf{0},\mathbf{+1}\}}^{\beta}$ and
$h_{\mathbf{-1},\mathbf{0}}^{\beta}$. Thus, we have
\begin{align*}
\sum_{\sigma\in\mathcal{X}}h(\sigma)(\mathrm{div}\,\phi^{-1,0})(\sigma) & =\sum_{\sigma\in\mathcal{N}(\mathbf{-1})\cup\mathcal{N}(\mathbf{0})}h(\sigma)(\mathrm{div}\,\phi^{-1,0})(\sigma)\\
 & \simeq h(\mathbf{-1})\sum_{\sigma\in\mathcal{N}(\mathbf{-1})}(\mathrm{div}\,\phi^{-1,0})(\sigma)+h(\mathbf{0})\sum_{\sigma\in\mathcal{N}(\mathbf{0})}(\mathrm{div}\,\phi^{-1,0})(\sigma).
\end{align*}
By Proposition \ref{p_tflowdiv2}, the last formula asymptotically
equals
\[
\frac{1}{3\kappa}e^{-\beta\Gamma}[h(\mathbf{-1})-h(\mathbf{0})]=\frac{1}{3\kappa}e^{-\beta\Gamma},
\]
because $h_{\mathbf{-1},\{\mathbf{0},\mathbf{+1}\}}^{\beta}(\mathbf{-1})=h_{\mathbf{-1},\mathbf{0}}^{\beta}(\mathbf{-1})=1$
and $h_{\mathbf{-1},\{\mathbf{0},\mathbf{+1}\}}^{\beta}(\mathbf{0})=h_{\mathbf{-1},\mathbf{0}}^{\beta}(\mathbf{0})=0$.
Summing up, we have
\[
\frac{1}{\|\phi^{-1,0}\|_{\beta}^{2}}\Big[\sum_{\sigma\in\mathcal{X}}h(\sigma)(\mathrm{div}\,\phi^{-1,0})(\sigma)\Big]^{2}\simeq\frac{1}{3\kappa}e^{-\beta\Gamma},
\]
which holds for both $h=h_{\mathbf{-1},\{\mathbf{0},\mathbf{+1}\}}^{\beta}$
and $h_{\mathbf{-1},\mathbf{0}}^{\beta}$. Hence, by the generalized
Thomson principle in Theorem \ref{t_gTP}, we have
\begin{equation}
\mathrm{Cap}_{\beta}(\mathbf{-1},\,\{\mathbf{0},\,\mathbf{+1}\}),\,\mathrm{Cap}_{\beta}(\mathbf{-1},\,\mathbf{0})\ge\frac{1+o(1)}{3\kappa}e^{-\beta\Gamma}.\label{e_Cap2}
\end{equation}
Therefore, by \eqref{e_Cap1} and \eqref{e_Cap2}, we conclude the
proof.
\end{proof}

\subsection{Proof of part (3) of Theorem \ref{t_Cap}}

We compute $\mathrm{Cap}_{\beta}(\mathbf{0},\,\{\mathbf{-1},\,\mathbf{+1}\})$.
\begin{proof}[Proof of part (3) of Theorem \ref{t_Cap}]
 Here, we use the test objects
\[
g=1-g^{-1,0}-g^{+1,0}\;\;\;\;\text{and}\;\;\;\;\phi=-\phi^{-1,0}-\phi^{+1,0}.
\]
First, Definitions \ref{d_tfcn-10} and \ref{d_tfcn+10} imply that
\begin{equation}
g^{-1,0}(\mathbf{s})=\begin{cases}
1 & \text{if }\mathbf{s}=\mathbf{-1}\\
0 & \text{if }\mathbf{s}=\mathbf{0},\,\mathbf{+1}
\end{cases}\;\;\;\;\text{and}\;\;\;\;g^{+1,0}(\mathbf{s})=\begin{cases}
1 & \text{if }\mathbf{s}=\mathbf{+1},\\
0 & \text{if }\mathbf{s}=\mathbf{-1},\,\mathbf{0}.
\end{cases}\label{e_Cap2.1}
\end{equation}
Thus, $g\in\mathfrak{C}(\{\mathbf{0}\},\,\{\mathbf{-1},\,\mathbf{+1}\})$.
Moreover, we write
\begin{align*}
D_{\beta}(g) & =\sum_{\{\sigma,\zeta\}\in E(\mathcal{X})}\mu_{\beta}(\sigma)c_{\beta}(\sigma,\,\zeta)\{g(\zeta)-g(\sigma)\}^{2}\\
 & =\sum_{\{\sigma,\zeta\}\in E(\mathcal{X})}\mu_{\beta}(\sigma)c_{\beta}(\sigma,\,\zeta)\{g^{-1,0}(\sigma)-g^{-1,0}(\zeta)+g^{+1,0}(\sigma)-g^{+1,0}(\zeta)\}^{2}.
\end{align*}
By Remark \ref{r_tfcntfl2}, if $\sigma\sim\zeta$, then $g^{-1,0}(\sigma)=g^{-1,0}(\zeta)$
or $g^{+1,0}(\sigma)=g^{+1,0}(\zeta)$. This implies that the last
summation equals
\begin{align*}
 & \sum_{\{\sigma,\zeta\}\in E(\mathcal{X})}\mu_{\beta}(\sigma)c_{\beta}(\sigma,\,\zeta)\big[\{g^{-1,0}(\sigma)-g^{-1,0}(\zeta)\}^{2}+\{g^{+1,0}(\sigma)-g^{+1,0}(\zeta)\}^{2}\big]\\
 & =D_{\beta}(g^{-1,0})+D_{\beta}(g^{+1,0}).
\end{align*}
Hence, by the Dirichlet principle and Proposition \ref{p_tfcn}, we
have
\begin{equation}
\mathrm{Cap}_{\beta}(\mathbf{0},\,\{\mathbf{-1},\,\mathbf{+1}\})\le D_{\beta}(g)=D_{\beta}(g^{-1,0})+D_{\beta}(g^{+1,0})=\frac{2+o(1)}{3\kappa}e^{-\beta\Gamma}.\label{e_Cap2.2}
\end{equation}
Next, we handle the lower bound. By Remark \ref{r_tfcntfl2}, we have
\[
\|\phi\|_{\beta}^{2}=\sum_{\{\sigma,\zeta\}\in E(\mathcal{X})}\frac{\phi(\sigma,\,\zeta)^{2}}{\mu_{\beta}(\sigma)c_{\beta}(\sigma,\,\zeta)}=\sum_{\{\sigma,\zeta\}\in E(\mathcal{X})}\frac{\phi^{-1,0}(\sigma,\,\zeta)^{2}}{\mu_{\beta}(\sigma)c_{\beta}(\sigma,\,\zeta)}+\sum_{\{\sigma,\zeta\}\in E(\mathcal{X})}\frac{\phi^{+1,0}(\sigma,\,\zeta)^{2}}{\mu_{\beta}(\sigma)c_{\beta}(\sigma,\,\zeta)},
\]
which is exactly $\|\phi^{-1,0}\|_{\beta}^{2}+\|\phi^{+1,0}\|_{\beta}^{2}$.
Hence, by Proposition \ref{p_tflnorm}, we have
\[
\|\phi\|_{\beta}^{2}=\|\phi^{-1,0}\|_{\beta}^{2}+\|\phi^{+1,0}\|_{\beta}^{2}=\frac{2+o(1)}{3\kappa}e^{-\beta\Gamma}.
\]
Moreover, we temporarily denote by $h=h_{\mathbf{0},\{\mathbf{-1},\mathbf{+1}\}}^{\beta}$.
Then, the same deduction as in the proof of parts (1) and (2) of Theorem
\ref{t_Cap} implies that
\[
\sum_{\sigma\in\mathcal{X}}h(\sigma)(\mathrm{div}\,\phi^{-1,0})(\sigma)\simeq\frac{1}{3\kappa}e^{-\beta\Gamma}[h(\mathbf{-1})-h(\mathbf{0})]=-\frac{1}{3\kappa}e^{-\beta\Gamma}
\]
and
\[
\sum_{\sigma\in\mathcal{X}}h(\sigma)(\mathrm{div}\,\phi^{+1,0})(\sigma)\simeq\frac{1}{3\kappa}e^{-\beta\Gamma}[h(\mathbf{+1})-h(\mathbf{0})]=-\frac{1}{3\kappa}e^{-\beta\Gamma}.
\]
Hence, we have
\[
\sum_{\sigma\in\mathcal{X}}h(\sigma)(\mathrm{div}\,\phi)(\sigma)\simeq\frac{1}{3\kappa}e^{-\beta\Gamma}+\frac{1}{3\kappa}e^{-\beta\Gamma}=\frac{2}{3\kappa}e^{-\beta\Gamma}.
\]
Summing up, we have
\[
\frac{1}{\|\phi\|_{\beta}^{2}}\Big[\sum_{\sigma\in\mathcal{X}}h(\sigma)(\mathrm{div}\,\phi)(\sigma)\Big]^{2}\simeq\frac{2}{3\kappa}e^{-\beta\Gamma}.
\]
Hence, by the generalized Thomson principle in Theorem \ref{t_gTP},
we have
\begin{equation}
\mathrm{Cap}_{\beta}(\mathbf{0},\,\{\mathbf{-1},\,\mathbf{+1}\})\ge\frac{2+o(1)}{3\kappa}e^{-\beta\Gamma}.\label{e_Cap2.3}
\end{equation}
Therefore, by \eqref{e_Cap2.2} and \eqref{e_Cap2.3}, we conclude
the proof.
\end{proof}

\subsection{Proof of part (4) of Theorem \ref{t_Cap}}

Finally, we prove part (4) of Theorem \ref{t_Cap} and thereby conclude
the proof of the main theorems.
\begin{proof}[Proof of part (4) of Theorem \ref{t_Cap}]
 Here, we use the test objects
\[
g=\frac{1}{2}(1+g^{-1,0}-g^{+1,0})\;\;\;\;\text{and}\;\;\;\;\phi=\phi^{-1,0}-\phi^{+1,0}.
\]
First, \eqref{e_Cap2.1} implies that $g\in\mathfrak{C}(\{\mathbf{-1}\},\,\{\mathbf{+1}\})$.
Moreover, as in the preceding proof, Remark \ref{r_tfcntfl2} and
Proposition \ref{p_tfcn} imply that
\[
D_{\beta}(g)=\frac{1}{4}\big[D_{\beta}(g^{-1,0})+D_{\beta}(g^{+1,0})\big]=\frac{1+o(1)}{6\kappa}e^{-\beta\Gamma}.
\]
Hence, by the Dirichlet principle, we have
\begin{equation}
\mathrm{Cap}_{\beta}(\mathbf{-1},\,\mathbf{+1})\le D_{\beta}(g)=\frac{1+o(1)}{6\kappa}e^{-\beta\Gamma}.\label{e_Cap3.1}
\end{equation}
Next, again using Remark \ref{r_tfcntfl2} and Proposition \ref{p_tflnorm},
we first have
\[
\|\phi\|_{\beta}^{2}=\|\phi^{-1,0}\|_{\beta}^{2}+\|\phi^{+1,0}\|_{\beta}^{2}=\frac{2+o(1)}{3\kappa}e^{-\beta\Gamma}.
\]
Moreover, we temporarily denote by $h=h_{\mathbf{-1},\mathbf{+1}}^{\beta}$.
Then, the same deduction as above and Theorem \ref{t_eqp} imply that
\[
\sum_{\sigma\in\mathcal{X}}h(\sigma)(\mathrm{div}\,\phi^{-1,0})(\sigma)\simeq\frac{1}{3\kappa}e^{-\beta\Gamma}[h(\mathbf{-1})-h(\mathbf{0})]\simeq\frac{1}{6\kappa}e^{-\beta\Gamma}
\]
and
\[
\sum_{\sigma\in\mathcal{X}}h(\sigma)(\mathrm{div}\,\phi^{+1,0})(\sigma)\simeq\frac{1}{3\kappa}e^{-\beta\Gamma}[h(\mathbf{+1})-h(\mathbf{0})]\simeq-\frac{1}{6\kappa}e^{-\beta\Gamma}.
\]
Hence, we have
\[
\sum_{\sigma\in\mathcal{X}}h(\sigma)(\mathrm{div}\,\phi)(\sigma)=\frac{1}{6\kappa}e^{-\beta\Gamma}+\frac{1}{6\kappa}e^{-\beta\Gamma}=\frac{1}{3\kappa}e^{-\beta\Gamma}.
\]
Summing up, we have
\[
\frac{1}{\|\phi\|_{\beta}^{2}}\Big[\sum_{\sigma\in\mathcal{X}}h(\sigma)(\mathrm{div}\,\phi)(\sigma)\Big]^{2}\simeq\frac{1}{6\kappa}e^{-\beta\Gamma}.
\]
Hence, by the generalized Thomson principle, we have
\begin{equation}
\mathrm{Cap}_{\beta}(\mathbf{-1},\,\mathbf{+1})\ge\frac{1+o(1)}{6\kappa}e^{-\beta\Gamma}.\label{e_Cap3.2}
\end{equation}
Therefore, by \eqref{e_Cap3.1} and \eqref{e_Cap3.2}, we conclude
the proof.
\end{proof}

\section{\label{sec9}Periodic Boundary Conditions}

In this section, we briefly discuss the model with periodic boundary
conditions imposed. Thus, throughout this section, we asume that \textit{$\Lambda$
is given periodic boundary conditions}; that is, $\Lambda=\mathbb{T}_{K}\times\mathbb{T}_{L}$.
Compared to the logic established thus far for the open boundary case,
the storyline for the periodic boundary case is nearly the same, although
certain slight technical differences exist between the two. As our
companion paper \cite{Kim-Seo Ising-Potts} thoroughly examines the
similar Potts model (with $q=3$) with periodic boundary conditions
imposed, we refer interested readers to \cite{Kim-Seo Ising-Potts}
and provide a short summary in this section.

We handle two issues here: the energy barrier between the ground states
that appears in Theorem \ref{t_Ebarrier} and the sub-exponential
prefactor that appears in Theorem \ref{t_EK}.

\subsubsection*{Energy barrier between ground states}

Recall that Theorem \ref{t_Ebarrier} in the periodic case is interpreted
as
\[
\Gamma_{-1,0}=\Gamma_{0,+1}=\Gamma_{-1,+1}=2K+2.
\]
It can be observed that the energy barrier in this case is twice that
of the open boundary model. To explain this, we recall the canonical
path defined in Definition \ref{d_can}. The exact same canonical
path also attains the energy barrier in the periodic case. However,
in the periodic case, the maximal energy of the canonical path is
doubled, because the sites on the edges of $\Lambda$ are also connected
to the corresponding sites on the other end of $\Lambda$. Therefore,
in the periodic case, we can easily determine that (cf. Remark \ref{r_H})
\[
H(\sigma)=\begin{cases}
2K & \text{if }\sigma\in\mathcal{R}_{v}^{-1,0}\cup\mathcal{R}_{v}^{0,+1}\text{ for }v\in\llbracket1,\,L-1\rrbracket,\\
2K+2 & \text{if }\sigma\in\mathcal{Q}_{v}^{-1,0}\cup\mathcal{Q}_{v}^{0,+1}\text{ for }v\in\llbracket1,\,L-2\rrbracket,
\end{cases}
\]
so that the canonical paths are $(2K+2)$-paths connecting the ground
states in $\mathcal{S}$. Moreover, the deduction in Section \ref{sec4.2}
can be modified slightly to verify that the energy barrier is precisely
$2K+2$.

As noted in Remark \ref{r_bdry}, once the energy barrier $\Gamma=2K+2$
is settled, the large deviation-type main results in Theorem \ref{t_LDTresults}
hold without any modification. Theorem \ref{t_eqp} follows in the
same manner.

\subsubsection*{Sub-exponential prefactor}

As explained in Remark \ref{r_bdry}, the exact quantitative estimates
of the metastable transitions differ between the two boundary conditions.
The constant $\kappa$ in Theorem \ref{t_EK}, which constitutes the
sub-exponential prefactor of the Eyring--Kramers law, must be modified
to $\kappa'$ in this case. We provide the correct versions of Theorems
\ref{t_EK} and \ref{t_MC} in the periodic case.
\begin{thm}
\label{t_EK-1}Under periodic boundary conditions on $\Lambda$, there
exists a constant $\kappa'=\kappa'(K,\,L)>0$ such that parts (1)
to (4) of Theorem \ref{t_EK} hold with $\kappa'$ instead of $\kappa$.
Moreover, the constant $\kappa'$ satisfies (cf. \eqref{e_KL})
\begin{equation}
\lim_{K\rightarrow\infty}\kappa'(K,\,L)=\begin{cases}
1/4 & \text{if }K<L,\\
1/8 & \text{if }K=L.
\end{cases}\label{e_EKkappa-1}
\end{equation}
\end{thm}

Moreover, as an analogue of \eqref{e_LMC}, we define the limiting
Markov chain $\{X'(t)\}_{t\ge0}$ on $\mathcal{S}$ as the continuous-time
Markov chain associated with the transition rate given by

\begin{equation}
r_{X'}(\mathbf{s},\,\mathbf{s}')=\begin{cases}
(\kappa')^{-1} & \text{if }\{\mathbf{s},\,\mathbf{s}'\}=\{\mathbf{-1},\,\mathbf{0}\}\text{ or }\{\mathbf{0},\,\mathbf{+1}\},\\
0 & \text{otherwise}.
\end{cases}\label{e_LMC-1}
\end{equation}

\begin{thm}
\label{t_MC-1}Under periodic boundary conditions on $\Lambda$, parts
(1) and (2) of Theorem \ref{t_MC} hold with $X'(\cdot)$ instead
of $X(\cdot)$.
\end{thm}

As can be observed from Theorems \ref{t_EK-1} and \ref{t_MC-1},
the difference between the two boundary conditions lies in the constants
$\kappa$ and $\kappa'$. That is, according to \eqref{e_EKkappa}
and \eqref{e_EKkappa-1}, the constants $\kappa$ and $\kappa'$ differ
by the factor $KL$ (in the limit $K\rightarrow\infty$). We refer
to \cite[Section 17]{Kim-Seo Ising-Potts} for a thorough heuristic
explanation of this factor $KL$. We provide the precise definition
of $\kappa'$, which is an analogue of Definition \ref{d_const}.
The constant $\kappa'$ satisfies $\kappa'=\mathfrak{b}'+2\mathfrak{e}'$,
where the bulk constant $\mathfrak{b}'=\mathfrak{b}'(K,\,L)$ is defined
as
\[
\mathfrak{b}'=\begin{cases}
\frac{(K+2)(L-4)}{4KL} & \text{if }K<L\\
\frac{(K+2)(L-4)}{8KL} & \text{if }K=L
\end{cases}
\]
and the edge constant $\mathfrak{e}'=\mathfrak{e}'(K,\,L)$ is defined
in the same manner as $\mathfrak{e}$ which satisfies
\begin{equation}
\mathfrak{e}'\le\frac{C}{KL}\;\;\;\;\text{for some constant }C>0.\label{e_e'}
\end{equation}
Thus, the estimate \eqref{e_EKkappa-1} holds for $\kappa'$.

\appendix

\section{\label{secA}Auxiliary Process}

\subsection{\label{secA.1}Original auxiliary process}

\begin{figure}
\includegraphics[width=13cm]{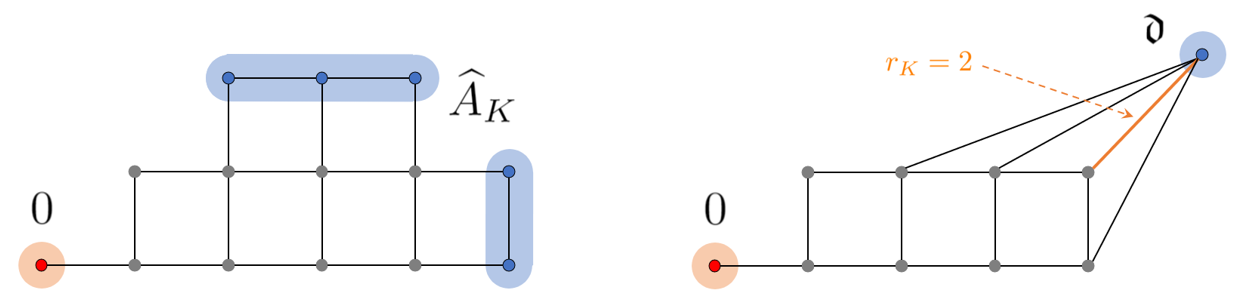}\caption{\label{figA.1}(Left) graph structure $(\widehat{V}_{K},\,E(\widehat{V}_{K}))$.
(Right) graph structure $(V_{K},\,E(V_{K}))$.}
\end{figure}
In this subsection, we define an auxiliary process which successfully
represents the Metropolis dynamics on the edge typical configurations.
For $K\ge5$, we define a graph structure $(\widehat{V}_{K},\,E(\widehat{V}_{K}))$
(see Figure \ref{figA.1} (left) for an illustration for the case
of $K=5$). First, the vertex set $\widehat{V}_{K}\subseteq\mathbb{R}^{2}$
is defined by
\begin{equation}
\widehat{V}_{K}=\{(a,\,b)\in\mathbb{R}^{2}:0\le b\le a\le K\text{ and }b\le2\}\setminus\{(K,\,2)\}.\label{e_VK}
\end{equation}
Then, the edge structure $E(\widehat{V}_{K})$ is inherited by the
Euclidean lattice. We abbreviate by $0=(0,\,0)\in\widehat{V}_{K}$
and define
\[
\widehat{A}_{K}=\{(a,\,b)\in\widehat{V}_{K}:a=K\text{ or }b=2\}.
\]
Then, we define $\{\widehat{Z}_{K}(t)\}_{t\ge0}$ as the continuous-time
random walk on the aforementioned graph whose transition rate is uniformly
$1$. In other words, the transition rate $\widehat{r}_{K}:\widehat{V}_{K}\times\widehat{V}_{K}\rightarrow[0,\,\infty)$
is given by
\[
\widehat{r}_{K}(x,\,y)=\begin{cases}
1 & \text{if }\{x,\,y\}\in E(\widehat{V}_{K}),\\
0 & \text{otherwise}.
\end{cases}
\]
Obviously, the process is reversible with respect to the uniform distribution
on $\widehat{V}_{K}$.

We denote by $\widehat{h}_{\cdot,\cdot}^{K}(\cdot)$ and $\widehat{\mathrm{cap}}_{K}(\cdot,\,\cdot)$
the equilibrium potential and capacity with respect to $\widehat{Z}_{K}(\cdot)$,
respectively, in the sense of Definition \ref{d_pot}. We define a
constant $\mathfrak{c}_{K}>0$ by
\begin{equation}
\mathfrak{c}_{K}=|\widehat{V}_{K}|\widehat{\mathrm{cap}}_{K}(0,\,\widehat{A}_{K}).\label{e_cK1}
\end{equation}
Then, we have the following asymptotic lemma.
\begin{lem}
\label{l_auxproc}There exists a positive constant $\delta$ with
$|\delta-0.435|<0.0001$ such that
\[
\lim_{K\rightarrow\infty}\mathfrak{c}_{K}=\delta.
\]
\end{lem}

\begin{proof}
We explicitly compute the equilibrium potential $\widehat{h}_{0,\widehat{A}_{K}}^{K}(\cdot)$.
For simplicity, we write $h=\widehat{h}_{0,\widehat{A}_{K}}^{K}$
and abbreviate by $h(a,\,b)=h((a,\,b))$ for $(a,\,b)\in\widehat{V}_{K}$.
We define
\[
a_{i}=h(K-i,\,0)\;\;\;\;;\;i\in\llbracket0,\,K\rrbracket\;\;\;\;\text{and}\;\;\;\;b_{i}=h(K-i,\,1)\;\;\;\;;\;i\in\llbracket0,\,K-1\rrbracket.
\]
Then, we trivially have $a_{K}=h(0,\,0)=1$,
\begin{equation}
a_{0}=h(K,\,0)=0\;\;\;\;\text{and}\;\;\;\;b_{0}=h(K,\,1)=0.\label{e_auxproc1}
\end{equation}
 Moreover, by the Markov property, the following recurrence relations
hold:
\begin{align}
3a_{i} & =a_{i+1}+a_{i-1}+b_{i}\;\;\;\;;\;i\in\llbracket1,\,K-1\rrbracket,\label{e_auxproc2}\\
4b_{i} & =b_{i+1}+b_{i-1}+a_{i}\;\;\;\;;\;i\in\llbracket1,\,K-2\rrbracket,\label{e_auxproc3}
\end{align}
\begin{equation}
2b_{K-1}=b_{K-2}+a_{K-1}\;\;\;\;\text{and}\;\;\;\;3a_{K-1}=1+a_{K-2}+b_{K-1}.\label{e_auxproc4}
\end{equation}
Then, \eqref{e_auxproc2} and \eqref{e_auxproc3} induce the following
relations:
\begin{align}
a_{i+2}-7a_{i+1}+13a_{i}-7a_{i-1}+a_{i-2} & =0\;\;\;;\;i\in\llbracket2,\,K-3\rrbracket,\label{e_auxproc5}\\
b_{i+2}-7b_{i+1}+13b_{i}-7b_{i-1}+b_{i-2} & =0\;\;\;;\;i\in\llbracket2,\,K-3\rrbracket.\label{e_auxproc6}
\end{align}
Hence, we solve $t^{4}-7t^{3}+13t^{2}-7t+1=0$, which is equivalent
to $(t+t^{-1})^{2}-7(t+t^{-1})+11=0$. This gives $t+t^{-1}=(7\pm\sqrt{5})/2$.
Thus, we define the positive constants $\alpha_{1},\,\alpha_{2},\,\alpha_{3},\,\alpha_{4}>0$
such that $\alpha_{1}>\alpha_{2}$ are the solutions of $t+t^{-1}=(7+\sqrt{5})/2$
and $\alpha_{3}>\alpha_{4}$ are the solutions of $t+t^{-1}=(7-\sqrt{5})/2$.
Then, there exist constants $p_{n}$ and $q_{n}$, $n\in\llbracket1,\,4\rrbracket$
such that we have, for $i\in\llbracket0,\,K-1\rrbracket$,
\begin{equation}
a_{i}=\sum_{n=1}^{4}p_{n}\alpha_{n}^{i}\;\;\;\;\text{and}\;\;\;\;b_{i}=\sum_{n=1}^{4}q_{n}\alpha_{n}^{i}.\label{e_auxproc7}
\end{equation}
Based on the last formula, \eqref{e_auxproc1} implies
\begin{equation}
p_{1}+p_{2}+p_{3}+p_{4}=q_{1}+q_{2}+q_{3}+q_{4}=0,\label{e_auxproc8}
\end{equation}
and \eqref{e_auxproc2} implies, for $i\in\llbracket1,\,K-1\rrbracket$,
\[
\sum_{n=1}^{4}\alpha_{n}^{i-1}\{p_{n}\alpha_{n}^{2}-(3p_{n}-q_{n})\alpha_{n}+p_{n}\}=0.
\]
As $K\ge5$, this implies that
\[
\begin{bmatrix}1 & 1 & 1 & 1\\
\alpha_{1} & \alpha_{2} & \alpha_{3} & \alpha_{4}\\
\alpha_{1}^{2} & \alpha_{2}^{2} & \alpha_{3}^{2} & \alpha_{4}^{2}\\
\alpha_{1}^{3} & \alpha_{2}^{3} & \alpha_{3}^{3} & \alpha_{4}^{3}
\end{bmatrix}\begin{bmatrix}p_{1}\alpha_{1}^{2}-(3p_{1}-q_{1})\alpha_{1}+p_{1}\\
p_{2}\alpha_{2}^{2}-(3p_{2}-q_{2})\alpha_{2}+p_{2}\\
p_{3}\alpha_{3}^{2}-(3p_{3}-q_{3})\alpha_{3}+p_{3}\\
p_{4}\alpha_{4}^{2}-(3p_{4}-q_{4})\alpha_{4}+p_{4}
\end{bmatrix}=0.
\]
As the square matrix is invertible (cf. Vandermonde matrix), we must
have $p_{n}\alpha_{n}^{2}-(3p_{n}-q_{n})\alpha_{n}+p_{n}=0$ for all
$n\in\llbracket1,\,4\rrbracket$, which implies that
\[
3p_{n}-q_{n}=\frac{7+\sqrt{5}}{2}p_{n}\;\;\;\;;\;n=1,\,2,\;\;\;\;\text{and}\;\;\;\;3p_{n}-q_{n}=\frac{7-\sqrt{5}}{2}p_{n}\;\;\;\;;\;n=3,\,4.
\]
Hence, substituting \eqref{e_auxproc8} and the last display to \eqref{e_auxproc7}
gives, for $i\in\llbracket0,\,K-1\rrbracket$,
\begin{equation}
a_{i}=p_{1}(\alpha_{1}^{i}-\alpha_{2}^{i})+p_{3}(\alpha_{3}^{i}-\alpha_{4}^{i}),\;\;\;b_{i}=-\frac{1+\sqrt{5}}{2}p_{1}(\alpha_{1}^{i}-\alpha_{2}^{i})+\frac{-1+\sqrt{5}}{2}p_{3}(\alpha_{3}^{i}-\alpha_{4}^{i}).\label{e_auxproc9}
\end{equation}
Substituting the last formula to \eqref{e_auxproc4} implies
\begin{align*}
 & (2+\sqrt{5})p_{1}(\alpha_{1}^{K-1}-\alpha_{2}^{K-1})+(2-\sqrt{5})p_{3}(\alpha_{3}^{K-1}-\alpha_{4}^{K-1})\\
 & =\frac{1+\sqrt{5}}{2}p_{1}(\alpha_{1}^{K-2}-\alpha_{2}^{K-2})+\frac{1-\sqrt{5}}{2}p_{3}(\alpha_{3}^{K-2}-\alpha_{4}^{K-2})
\end{align*}
and
\begin{align*}
 & \frac{7+\sqrt{5}}{2}p_{1}(\alpha_{1}^{K-1}-\alpha_{2}^{K-1})+\frac{7-\sqrt{5}}{2}p_{3}(\alpha_{3}^{K-1}-\alpha_{4}^{K-1})\\
 & =1+p_{1}(\alpha_{1}^{K-2}-\alpha_{2}^{K-2})+p_{3}(\alpha_{3}^{K-2}-\alpha_{4}^{K-2}).
\end{align*}
Solving the last two equations, we can express $p_{1}$ and $p_{3}$
in terms of $\alpha_{1},\,\alpha_{2},\,\alpha_{3},\,\alpha_{4}$.
Substituting these to the first equation of \eqref{e_auxproc9} for
$i=K-1$, we deduce that $a_{K-1}$ equals
\[
-\Big[(2-\sqrt{5})-\frac{1-\sqrt{5}}{2}\frac{\alpha_{3}^{K-2}-\alpha_{4}^{K-2}}{\alpha_{3}^{K-1}-\alpha_{4}^{K-1}}\Big]+\Big[(2+\sqrt{5})-\frac{1+\sqrt{5}}{2}\frac{\alpha_{1}^{K-2}-\alpha_{2}^{K-2}}{\alpha_{1}^{K-1}-\alpha_{2}^{K-1}}\Big]
\]
divided by
\begin{align*}
 & \Big[(2+\sqrt{5})-\frac{1+\sqrt{5}}{2}\frac{\alpha_{1}^{K-2}-\alpha_{2}^{K-2}}{\alpha_{1}^{K-1}-\alpha_{2}^{K-1}}\Big]\Big[\frac{7-\sqrt{5}}{2}-\frac{\alpha_{3}^{K-2}-\alpha_{4}^{K-2}}{\alpha_{3}^{K-1}-\alpha_{4}^{K-1}}\Big]\\
 & -\Big[(2-\sqrt{5})-\frac{1-\sqrt{5}}{2}\frac{\alpha_{3}^{K-2}-\alpha_{4}^{K-2}}{\alpha_{3}^{K-1}-\alpha_{4}^{K-1}}\Big]\Big[\frac{7+\sqrt{5}}{2}-\frac{\alpha_{1}^{K-2}-\alpha_{2}^{K-2}}{\alpha_{1}^{K-1}-\alpha_{2}^{K-1}}\Big].
\end{align*}
As $\alpha_{1}>\alpha_{2}$ and $\alpha_{3}>\alpha_{4}$, we have
$(\alpha_{2}/\alpha_{1})^{K-1}\rightarrow0$ and $(\alpha_{4}/\alpha_{3})^{K-1}\rightarrow0$
as $K\rightarrow\infty$. Thus, we may calculate
\begin{align*}
\lim_{K\rightarrow\infty}a_{K-1} & =\frac{-\Big[(2-\sqrt{5})-\frac{1-\sqrt{5}}{2}\frac{1}{\alpha_{3}}\Big]+\Big[(2+\sqrt{5})-\frac{1+\sqrt{5}}{2}\frac{1}{\alpha_{1}}\Big]}{\Big[(2+\sqrt{5})-\frac{1+\sqrt{5}}{2}\frac{1}{\alpha_{1}}\Big]\Big[\frac{7-\sqrt{5}}{2}-\frac{1}{\alpha_{3}}\Big]-\Big[(2-\sqrt{5})-\frac{1-\sqrt{5}}{2}\frac{1}{\alpha_{3}}\Big]\Big[\frac{7+\sqrt{5}}{2}-\frac{1}{\alpha_{1}}\Big]}\\
 & =\frac{2\sqrt{5}-\frac{1+\sqrt{5}}{2}\alpha_{2}+\frac{1-\sqrt{5}}{2}\alpha_{4}}{5\sqrt{5}+\frac{3-5\sqrt{5}}{2}\alpha_{2}+\frac{-3-5\sqrt{5}}{2}\alpha_{4}+\sqrt{5}\alpha_{2}\alpha_{4}}.
\end{align*}
In the second equality, we used that $\alpha_{1}\alpha_{2}=\alpha_{3}\alpha_{4}=1$.
By substituting the exact values of $\alpha_{i}$, this is asymptotically
$0.5649853624$. Moreover, as $0=(0,\,0)\in\widehat{V}_{K}$ is connected
only to $(1,\,0)\in\widehat{V}_{K}$, we have by \cite[(7.1.39)]{B-denH}
that
\[
\widehat{\mathrm{cap}}_{K}(0,\,\widehat{A}_{K})=\frac{1}{|\widehat{V}_{K}|}[h(0,\,0)-h(1,\,0)]=\frac{1-a_{K-1}}{|\widehat{V}_{K}|}.
\]
Therefore, we have
\[
\lim_{K\rightarrow\infty}\mathfrak{c}_{K}=\lim_{K\rightarrow\infty}|\widehat{V}_{K}|\widehat{\mathrm{cap}}_{K}(0,\,\widehat{A}_{K})=1-\lim_{K\rightarrow\infty}a_{K-1}\approx0.4350146376,
\]
which concludes the proof.
\end{proof}
\begin{rem}
\label{r_bdry3}In the periodic boundary case, we need a completely
different auxiliary process to estimate the structure of edge typical
configurations. Namely, the desired process is a Markov chain on the
collection of subtrees of a $K\times2$-shaped ladder graph with semi-periodic
boundary conditions (i.e., open on the horizontal boundaries and periodic
on the vertical ones). In this case, we deduce an upper bound for
the corresponding capacity, which is sufficient to obtain the bound
\eqref{e_e'}. We refer to \cite[Proposition 7.9]{Kim-Seo Ising-Potts}
for more information on this estimate.
\end{rem}

\subsection{\label{secA.2}Projected auxiliary process}

Based on the original auxiliary process defined in the preceding subsection,
we define a projected auxiliary process which is obtained by simply
projecting the elements in $\widehat{A}_{K}$ to a single element
$\mathfrak{d}$. Rigorously, we define a graph structure $(V_{K},\,E(V_{K}))$
(see Figure \ref{figA.1} (right) for an illustration for the case
of $K=5$). The vertex set $V_{K}\subseteq\mathbb{R}^{2}$ is defined
by
\begin{equation}
V_{K}=(\widehat{V}_{K}\setminus\widehat{A}_{K})\cup\{\mathfrak{d}\}.\label{e_VK-1}
\end{equation}
Then, the edge structure $E(V_{K})$ is inherited by $E(\widehat{V}_{K})$;
we have $\{x,\,y\}\in E(V_{K})$ for $\{x,\,y\}\in E(\widehat{V}_{K})$,
$x,\,y\in\widehat{V}_{K}\setminus\widehat{A}_{K}$, and we have $\{x,\,\mathfrak{d}\}\in E(V_{K})$
for
\[
x\in\widehat{V}_{K}\setminus\widehat{A}_{K}\;\;\;\;\text{satisfying}\;\;\;\;\exists y\in\widehat{A}_{K}\;\;\;\;\text{with}\;\;\;\;\{x,\,y\}\in E(\widehat{V}_{K}).
\]
Then, we define $\{Z_{K}(t)\}_{t\ge0}$ as the continuous-time Markov
chain on $(V_{K},\,E(V_{K}))$ whose transition rate $r_{K}$ is defined
by $r_{K}(x,\,y)=\widehat{r}_{K}(x,\,y)$ if $x,\,y\ne\mathfrak{d}$
and
\[
r_{K}(x,\,\mathfrak{d})=r_{K}(\mathfrak{d},\,x)=\sum_{y\in\widehat{A}_{K}}\widehat{r}_{K}(x,\,y).
\]
This process is reversible with respect to the uniform distribution
on $V_{K}$.

We denote by $h_{\cdot,\cdot}^{K}(\cdot)$, $\mathrm{cap}_{K}(\cdot,\,\cdot)$,
$D_{K}(\cdot)$ the equilibrium potential, capacity, and Dirichlet
form with respect to $Z_{K}(\cdot)$, respectively, in the sense of
Definition \ref{d_pot}. Then, by the strong Markov property, it is
immediate from the definition that
\[
h_{0,\mathfrak{d}}^{K}(x)=\widehat{h}_{0,\widehat{A}_{K}}^{K}(x)\;\;\;\;;\;x\in\widehat{V}_{K}\setminus\widehat{A}_{K}\;\;\;\;\text{and}\;\;\;\;h_{0,\mathfrak{d}}^{K}(\mathfrak{d})=\widehat{h}_{0,\widehat{A}_{K}}^{K}(y)=0\;\;\;\;;\;y\in\widehat{A}_{K}.
\]
Therefore, by \eqref{e_Capdef}, we have (cf. \eqref{e_cK1})
\begin{equation}
|V_{K}|\mathrm{cap}_{K}(0,\,\mathfrak{d})=|\widehat{V}_{K}|\widehat{\mathrm{cap}}_{K}(0,\,\widehat{A}_{K})=\mathfrak{c}_{K}.\label{e_cK2}
\end{equation}

\begin{acknowledgement*}
S. Kim was supported by NRF-2019-Fostering Core Leaders of the Future
Basic Science Program/Global Ph.D. Fellowship Program and the National
Research Foundation of Korea (NRF) grant funded by the Korean government
(MSIT) (No. 2018R1C1B6006896).
\end{acknowledgement*}

\end{document}